\documentclass[11pt]{aims}
\usepackage{amsmath, amssymb, mathrsfs}
  \usepackage{paralist}
  \usepackage{graphics} 
  \usepackage{epsfig} 
\usepackage{graphicx}  \usepackage{epstopdf}
 \usepackage[colorlinks=true]{hyperref}
 \hypersetup{urlcolor=blue, citecolor=red}
 \usepackage[toc,page]{appendix}
\usepackage{chngcntr}
\usepackage{hyperref}

\usepackage{titlesec}
\titleformat{\section}{\Large\bfseries}{\thesection.}{4pt}{}
\titleformat{\subsection}{\large\bfseries}{\thesection.\arabic{subsection}.}{4pt}{}
\titleformat{\subsubsection}{\bfseries}{\thesection.\arabic{subsection}.\arabic{subsubsection}.}{4pt}{}
\titleformat*{\paragraph}{\bfseries}
\titleformat*{\subparagraph}{\bfseries}
  \def\currentyear{}
   \def\currentmonth{}
    \def\ppages{}
     \def\DOI{}

\usepackage[margin=1in]{geometry}

\newtheorem{theorem}{Theorem}[section]

\newtheorem{lemma}[theorem]{Lemma}
\newtheorem{proposition}[theorem]{Proposition}
\theoremstyle{definition}
\newtheorem{definition}[theorem]{Definition}
\newtheorem{remark}[theorem]{Remark}

\newcommand{\Rd}{\mathbb{R}^d}
\newcommand{\Rb}{\mathbb{R}}

\newcommand{\Cc}{\mathcal{C}}
\newcommand{\Dc}{\mathcal{D}}
\newcommand{\Ec}{\mathcal{E}}
\newcommand{\Fc}{\mathcal{F}}

\newcommand{\Hc}{\mathcal{H}}
\newcommand{\Lc}{\mathcal{L}}
\newcommand{\Oc}{\mathcal{O}}

\newcommand{\Nc}{\mathcal{N}}

\newcommand{\Sc}{\mathcal{S}}

\newcommand{\Uc}{\mathcal{U}}
\newcommand{\Vc}{\mathcal{V}}

\newcommand{\py}{\partial_y}
\newcommand{\ps}{\partial_s}

\newcommand{\pt}{\partial_t}

\newcommand{\Ls}{\mathscr{L}}

\newcommand{\As}{\mathscr{A}}
\newcommand{\Es}{\mathscr{E}}

\numberwithin{equation}{section}

\title[1-corotational energy supercritical harmonic heat flow] 
      {On the stability of type II blowup for the 1-corotational energy supercritical  harmonic heat flow}

\author[T. Ghoul, S. Ibrahim,  and V. T. Nguyen]{}

\email[S. Ibrahim]{ibrahim@math.uvic.ca}
 \email[T. Ghoul]{teg6@nyu.edu}
 \email[V. T. Nguyen]{Tien.Nguyen@nyu.edu}
 
\thanks{\today}
\begin{document}
\maketitle


\centerline{\scshape T. Ghoul$^a$, S. Ibrahim$^{a,b}$  and V. T. Nguyen$^a$}
\medskip
{\footnotesize
 \centerline{$^a$Department of Mathematics, New York University in Abu Dhabi,}
   \centerline{Saadiyat Island, P.O. Box 129188, Abu Dhabi, United Arab Emirates.}
}
\medskip
{\footnotesize
 \centerline{$^b$Department of Mathematics and Statistics, University of Victoria,}
   \centerline{PO Box 3060 STN CSC, Victoria, BC, V8P 5C3, Canada.}
}

\bigskip

\begin{abstract} We consider the energy supercritical harmonic heat flow from $\Rb^d$ into the $d$-sphere $\mathbb{S}^d$ with $d \geq 7$. Under an additional assumption of 1-corotational symmetry, the problem reduces to the one dimensional semilinear heat equation
$$\partial_t u = \partial^2_r u + \frac{(d-1)}{r}\partial_r u - \frac{(d-1)}{2r^2}\sin(2u).$$
We construct for this equation a family of $\mathcal{C}^{\infty}$ solutions which blow up in finite time  via concentration of the universal profile
$$u(r,t) \sim Q\left(\frac{r}{\lambda(t)}\right),$$
where $Q$ is the stationary solution of the equation and the speed is given by the quantized rates
$$\lambda(t) \sim c_u(T-t)^\frac{\ell}{\gamma}, \quad \ell \in \mathbb{N}^*, \;\; 2\ell > \gamma = \gamma(d) \in (1,2].$$
The construction relies on two arguments: the reduction of the problem to a finite-dimensional one thanks to a robust universal energy method and modulation techniques developed by Merle, Rapha\"el and Rodnianski \cite{MRRcjm15} for the energy supercritical nonlinear Schr\"odinger equation and by Rapha\"el and Schweyer \cite{RSapde2014} for the energy critical harmonic heat flow, then we proceed by contradiction to solve the finite-dimensional problem and conclude using the Brouwer fixed point theorem. 
Moreover, our constructed solutions are in fact $(\ell - 1)$ codimension stable under perturbations of the initial data. As a consequence, the case $\ell = 1$ corresponds to a stable type II blowup regime.

\end{abstract}

\section{Introduction.} We consider the harmonic map heat flow which is defined as the negative gradient flow of the Dirichlet energy of maps between manifolds. Indeed, if $\Phi$ is a map from $\Rb^d\times[0,T)$ to a compact riemannian  manifold $\mathcal{M}\subset \Rb^n$, with second fundamental form $\Upsilon$,  then $\Phi$ solves
\begin{equation}\label{p0}
\left\{\begin{array}{l}
\partial_t \Phi - \Delta \Phi = \Upsilon(\Phi)(\nabla \Phi,\nabla\Phi),\\
\Phi(t = 0) = \Phi_0.
\end{array}
 \right.
\end{equation}
We assume that the target manifold is the $d$-sphere $\mathbb{S}^d\subset\Rb^{d+1}$. Then, \eqref{p0} becomes
\begin{equation}\label{p1}
\left\{\begin{array}{l}
\partial_t \Phi - \Delta \Phi = |\nabla \Phi|^2\Phi,\\
\Phi(t = 0) = \Phi_0.
\end{array}
 \right.
\end{equation}
We will study the problem \eqref{p1} under an additional assumption of 1-corotational symmetry, namely that a solution of \eqref{p1} takes the form
\begin{equation}
\Phi(x,t) = \binom{\cos(u(|x|,t))}{\frac{x}{|x|}\sin(u(|x|,t))}.
\end{equation}
Under this ansatz, the problem \eqref{p1} reduces to the one dimensional semilinear heat equation
 \begin{equation}\label{Pb}
\left\{\begin{array}{rl}
\partial_t u &= \partial^2_r u + \frac{(d-1)}{r}\partial_r u - \frac{(d-1)}{2r^2}\sin(2u),\\
u(t=0) &= u_0,
\end{array}\right.
\end{equation}
where $u(t): r \in \Rb_+ \to u(r,t) \in [0, \pi]$. The set of solutions to \eqref{Pb} is invariant by the scaling symmetry
$$u_\lambda(r,t) = u\left(\frac{r}{\lambda}, \frac{t}{\lambda^2} \right), \quad \forall\lambda > 0.$$
The energy associated to \eqref{Pb} is given by 
\begin{equation}\label{def:Enut}
\Ec[u](t) = \int_0^{+\infty} \left(|\partial_r u|^2 + \frac{(d-1)}{r^2}\sin^2(u)\right)r^{d-1}dr,
\end{equation}
which satisfies
$$\Ec[u_\lambda] = \lambda^{d-2}\Ec[u].$$
The criticality of the problem is  reflected by the fact that the energy  \eqref{def:Enut} is left invariant by the scaling property when $d = 2$, hence, the case $d \geq 3$ corresponds to the energy supercritical case. 

The problem \eqref{Pb} is locally wellposed for data which are close in $L^\infty$ to a uniformly continuous map (see Koch and Lamm \cite{KLajm12}) or in $BMO$ by Wang \cite{Warma11}.
Actually, Eells and Sampson \cite{ESajm64} introduced the harmonic map heat flow as a process to deform any smooth map $\Phi_0$ into a harmonic map via \eqref{p1}. They also proved that the solution exists globally if the sectional curvature of the target manifold is negative. There exist other assumptions for the global existence such as the image of the initial data $u_0$ is contained in a ball of radius $\frac{\pi}{2\sqrt{\kappa}}$, where $\kappa$ is an upper bound on the sectional curvature of the target manifold $\mathcal{M}$ (see Jost \cite{Jmm81} and Lin-Wang \cite{LWwsp08}).
Without these assumptions, the solution $u(r,t)$ may develop singularities in some finite time (see for examples, Coron and Ghidaglia \cite{CGcrsa89}, Chen and Ding \cite{CDinvm90} for $d \geq 3$, Chang, Ding and Yei \cite{CDYjdg92} for $d = 2$). In this case, we say that $u(r,t)$ blows up in a finite time $T < +\infty$ in the sense that 
$$\lim_{t \to T}\|\nabla u(t)\|_{L^\infty} = + \infty.$$
Here we call $T$ the blowup time of $u(x,t)$. The blowup has been divided by Struwe \cite{Sams96} into two types:
\begin{align*}
\qquad & \text{$u$ blows up with type I if:} \quad \limsup_{t \to T} (T-t)^\frac{1}{2}\|\nabla u(t)\|_{L^\infty} < +\infty,\\
\qquad & \text{$u$ blows up with type II if:} \quad \limsup_{t \to T} (T-t)^\frac{1}{2}\|\nabla u(t)\|_{L^\infty} = +\infty.
\end{align*}

In \cite{Sjdg88}, Struwe shows that the type I singularities are asymptotically self-similar, that is their profile is given by a smooth shrinking function
$$u(r,t) = \phi\left(\frac{r}{\sqrt{T-t}}\right), \quad \forall t \in [0,T),$$
where $\phi$ solves the equation 
\begin{equation}\label{eq:phis}
\phi'' + \left(\frac{d-1}{y} + \frac{y}{2}\right)\phi' - \frac{d-1}{2y^2}\sin(2\phi) = 0.
\end{equation}
Thus, the study of Type I blowup reduces to the study of nonconstant solutions of equation \eqref{eq:phis}.\\

When $3 \leq d \leq 6$, by using a shooting method, Fan \cite{Fscsm99} proved that there exists an infinite sequence of globally regular solutions $\phi_n$ of \eqref{eq:phis} which are called "shrinkers (corresponding to the existence of Type I blowup solutions of \eqref{Pb}), where the integer index $n$ denotes the number of intersections of the function $\phi_n$ with $\frac{\pi}{2}$. More detailed quantitative properties of such solutions were studied by Biernat and Bizo{\'n} \cite{BBnon11}, where the authors conjectured that $\phi_1$ is linear stable and provide numerical evidences supporting that $\phi_1$ corresponds to a generic profile of Type I blow-up. Very recently, Biernat, Donninger and Sch\"orkhuber \cite{BDSar16} proved the existence of a stable self-similar blowup solution for $d=3$.
Since \eqref{p1} is not time reversible, there exist another family of self-similar solutions that are called "expanders" which were introduced by Germain and Rupflin in \cite{GRihp08}. These "expanders" have been lately proved to be nonlinearly stable by Germain, Ghoul and Miura \cite{GGMar16}. Up to our knowledge, the question on the existence of Type II blowup solutions  for \eqref{Pb} remains open for $3 \leq d \leq 6$.\\

When $d \geq 7$, Bizo{\'n} and Wasserman \cite{BWimrn15} proved that equation \eqref{Pb} has no self-similar shrinking solutions. According to Struwe \cite{Sjdg88}, this result implies that in dimensions $d \geq 7$, all singularities for equation \eqref{Pb} must be of type II (see also Biernat \cite{BIEnon2015} for a recent analysis of such singularities). Recently, Biernat and Seki \cite{BSarxiv2016}, via the matched asymptotic method developed by Herero and Vel\'azquez \cite{HVcras94}, construct for equation \eqref{Pb} a countable family of Type II blowup solutions, each characterized by a different blowup rate
\begin{equation}\label{eq:blrate}
\lambda(t) \sim (T-t)^\frac{\ell}{\gamma} \quad \text{as} \quad t \to T,
\end{equation} 
where $\ell \in \mathbb{N}^*$ such that $2\ell > \gamma$ and $\gamma = \gamma(d)$ is given by
\begin{equation}\label{def:gamome}
\gamma(d) = \frac{1}{2}(d - 2 - \tilde{\gamma}) \in (1,2] \quad \text{for}\;\; d \geq 7,
\end{equation}
where $\tilde{\gamma}=\sqrt{d^2-8d+8}$.
The blowup rate \eqref{eq:blrate} is in fact driven by the asymptotic behavior of a stationary solution of \eqref{Pb}, say $Q$, which is the unique  (up to scaling) solution of the equation
\begin{equation}\label{eq:Qr}
Q'' + \frac{(d-1)}{r}Q' - \frac{(d-1)}{2r^2}\sin(2Q) = 0, \quad Q(0) = 0, \; Q'(0) = 1,
\end{equation} 
and admits the behavior for $r$ large,
\begin{equation}\label{exp:Qr}
Q(r) = \dfrac{\pi}{2} - \dfrac{a_0}{ r^{\gamma}} + \Oc\left(\dfrac 1{r^{2 + \gamma}}\right) \quad \text{for some}\;\; a_0 = a_0(d) > 0,
\end{equation}
(see Appendix in \cite{BIEnon2015} for a proof of the existence of $Q$). Note that the case $2\ell = \gamma$ only happens in dimension $d = 7$. In this case, Biernat \cite{BIEnon2015} used the method of \cite{HVcras94} and formally derived the blowup rate
\begin{equation}\label{eq:rated7}
\lambda(t) \sim \frac{(T-t)^\frac 12}{|\log (T-t)|} \quad \text{as} \quad t \to T.
\end{equation}
He also provided numerical evidences supporting that the case $\ell = 1$ in \eqref{eq:blrate} or \eqref{eq:rated7} corresponds to a generic blowup solution.\\

In the energy critical case, i.e. $d = 2$, Van de Berg, Hulshof and King \cite{VHKsiam03}, through a formal analysis based on the matched asymptotic technique of Herrero and Vel\'azquez \cite{HVcras94}, predicted that there are type II blowup solutions to \eqref{Pb} of the form 
\begin{equation*}
u(r,t) \sim Q\left(\frac{r}{\lambda(t)}\right),
\end{equation*} 
where 
\begin{equation}\label{eq:Qrd2}
Q(r) = 2\tan^{-1}(r)
\end{equation}
is the unique (up to scaling) solution of \eqref{eq:Qr}, and the blowup speed governed by the quantized rates
\begin{equation*}
\lambda(t) \sim \frac{(T-t)^\ell}{|\log (T-t)|^\frac{2\ell}{2\ell - 1}} \qquad \text{for}\;\; \ell \in \mathbb{N}^*.
\end{equation*}
This result was later confirmed by Rapha\"el and Schweyer \cite{RSapde2014}. Note that the case $\ell = 1$ was treated in \cite{RScpam13} and corresponds to a stable blowup. In particular, the authors in \cite{RSapde2014}, \cite{RScpam13} adapted the strategy developed by Merle, Rapha\"el and Rodnianski \cite{RRmihes12}, \cite{MRRmasp11} for the study of wave and Schr\"odinger maps to construct for equation \eqref{Pb} type II blowup solutions. Their method is based on two main steps:
\begin{itemize}
\item Construction of an adapted approximate blowup profile by solving elliptic equations. From the computation of the tails of the blowup profile, they are able to formally derive the blowup speed. 

\item  Control of the error by using an energy method involving  "Lyapunov" functionals adapted to the linearized flow around the ground state.
This energy method doesn't involve neither spectral estimates nor maximum principles. 
\end{itemize}

\bigskip

In this work, by considering $d \geq 7$, we ask whether we can carry out the analysis of \cite{RSapde2014} for the energy critical case $d=2$, to the construction of blowup solutions for equation \eqref{Pb} in the case $d \geq 7$. It happens that the asymptotic behavior \eqref{exp:Qr} is perfectly suitable to replace the explicit profile \eqref{eq:Qrd2} for an implementation of the strategy of \cite{RSapde2014}. The following theorem is the main result of this paper. 
\begin{theorem}[Existence of type II blowup solutions to \eqref{Pb} with prescibed behavior] \label{Theo:1} Let $d \geq 7$ and $\gamma$ be defined as in \eqref{def:gamome}, we fix an integer 
$$\ell \in \mathbb{N}^* \quad  \text{such that}\quad 2\ell > \gamma,$$
and an arbitrary Sobolev exponent
$$\frak{s} \in \mathbb{N},\;\; \frak{s} = \frak{s}(\ell) \to +\infty \quad \text{as}\;\; \ell \to +\infty.$$
Then there exists a  smooth corotational radially symmetric initial data $u_0$ such that the corresponding solution to \eqref{Pb} is of the form:
\begin{equation}\label{eq:uQq}
u(r,t) = Q\left(\frac{r}{\lambda(t)}\right) + q\left(\frac{r}{\lambda(t)}, t\right)
\end{equation}
where 
\begin{equation}\label{eq:quanblrate}
\lambda(t) = c(u_0)(T-t)^\frac{\ell}{\gamma} (1 + o_{t \to T}(1)), \quad c(u_0) > 0,
\end{equation}
and 
\begin{equation}\label{eq:asypqs}
\lim_{t \to T}\|\nabla^{\sigma} q(t)\|_{L^2} = 0, \quad \forall \sigma \in \left(d/2 + 3, \frak{s}\right].
\end{equation}
Moreover, the case $\ell = 1$ corresponds to a stable blowup regime.
\end{theorem}
\begin{remark} Since $\gamma = 2$ for $d=7$ and $\gamma \in (1,2)$ for $d \geq 8$, the condition $2\ell > \gamma$ means that $\ell \geq 2$ for $d = 7$ and $\ell \geq 1$ for $d \geq 8$. Note that the condition $2\ell > \gamma$ allows to avoid the presence of logarithmic corrections in the construction of the approximate profile. In other words, the case $2\ell = \gamma$ (equivalent to $\ell = 1$ and $d = 7$) would involve an additional logarithmic gain related to the growth of the approximate profile at infinity, which turns out to be essential for the derivation of the speed \eqref{eq:rated7}. Although our analysis could be naturally extended to this case ($\ell = 1$ and $d = 7$) with some complicated computation, we hope to treat this case in a separate work.
\end{remark}
\begin{remark} The quantization of the blowup rate \eqref{eq:quanblrate} is the same as the one obtained in \cite{BSarxiv2016}. Note that the authors of \cite{BSarxiv2016} only claim the existence result of a type II blowup solution with the rate \eqref{eq:quanblrate} and say nothing about the dynamical description of the solution. On the contrary, our result shows that the constructed solution blows up in finite time by concentration of a stationary state in the supercritical regime. Moreover, our constructed solution is in fact $(\ell - 1)$ codimension stable in the sense that we will precise shortly.
\end{remark}
\begin{remark} Fix $\ell \in \mathbb{N}^*$ such that $2\ell > \gamma$, an integer $L \gg \ell$ and $\frak{s} \sim L \gg 1$, our initial data that we are considering has the following form 
\begin{equation}\label{eq:datau0}
u_0 = Q_{b(0)} + \epsilon_0,
\end{equation}
where $Q_b$ is a deformation of the ground state $Q$.
In addition, $b = (b_1, \cdots, b_{L})$ corresponds to unstable directions of the flow in the topology $\dot{H}^{\frak{s}}$ in a certain neighborhood of the ground state $Q$. We will prove that for all $\epsilon_0 \in \dot{H}^\sigma \cap\dot{H}^\frak s$ (for some $\sigma = \sigma(d) > \frac{d}{2}$) small enough, for all $(b_1(0), b_{\ell + 1}(0), \cdots, b_L(0))$ small enough, under a specific choice of the unstable directions $(b_2(0), \cdots, b_\ell(0))$ the solution of \eqref{Pb} with the data \eqref{eq:datau0} verifies the conclusion of Theorem \ref{Theo:1}. This implies that our constructed solution is $(\ell - 1)$ codimension stable. In particular, the case $\ell = 1$ corresponds to a stable type II blowup regime, which is in agreement with numerical evidences given in \cite{BIEnon2015}.
\end{remark}

\begin{remark} It is worth to remark that the harmonic heat flow shares many features with the semilinear heat equation
\begin{equation}\label{eq:she}
\partial_t u = \Delta u + |u|^{p-1}u \quad \text{in}\; \Rb^d.
\end{equation}
A remarkable fact that two important critical exponents appear when considering the dynamics of \eqref{eq:she}:
$$p_S = \frac{d + 2}{d-2} \quad \text{and} \quad p_{JL} = \left\{\begin{array}{ll} +\infty &\quad \text{for}\;\; d \leq 10,\\
 1 + \frac{4}{d - 4 - 2\sqrt{d - 1}}&\quad \text{for}\;\; d \geq 11,
\end{array}  \right.$$
correspond to the cases $d = 2$ and $d = 7$ in the study of equation \eqref{Pb} respectively. 

When $1 < p \leq p_S$, Giga and Kohn \cite{GKiumj87}, Giga, Matsui and Sasayama \cite{GMSiumj04} showed that all blowup solutions are of type I. Here the type I blowup means that
$$\limsup_{t \to T}(T-t)^\frac{1}{p-1}\|u(t)\|_{L^\infty} < +\infty,$$
otherwise we say the blowup solution is of type II.

When $p = p_S$, Filippas, Herrero and Vel\'azquez \cite{FHVslps00} formally constructed for \eqref{eq:she} type II blowup solutions in dimensions $3 \leq d \leq 6$, however, they could not do the same in dimensions $d \geq 7$. This formal result is partly confirmed by Schweyer \cite{Schfa12} in dimension $d = 4$. Interestingly, Collot, Merle and Rapha\"el \cite{CMRar16} show that type II blowup is ruled out in dimension $d \geq 7$ near the solitary wave.

When $p_S < p < p_{JL}$, Matano and Merle \cite{MMcpam04} (see also Mizoguchi \cite{Made04}) proved that only type I blowup occurs in the radial setting. 

When $p > p_{JL}$, Herrero and Vel\'azquez formally derived in \cite{HVcras94} the existence of type II blowup solutions with the quantized rates:
$$\|u(t)\|_{L^\infty} \sim (T-t)^ \frac{2\ell}{(p-1)\alpha(d,p)}, \quad \ell \in \mathbb{N}, \; 2\ell > \alpha.$$
The formal result was clarified in \cite{MMjfa09}, \cite{Mma07} and \cite{Car16}. The collection of these works yields a complete classification of the type II blowup scenario for the radially symmetric energy supercritical case. 

In comparison to the case of the semilinear heat equation \eqref{eq:she}, it might be possibe to prove that all blowup solutions to equation \eqref{Pb} are of type I in dimension $3\leq d \leq 6$. However, due to the lack of monotonicity of the nonlinear term, the analysis of the harmonic heat flow \eqref{Pb} is much more difficult than the case of the semilinear heat equation \eqref{eq:she} treated in \cite{MMcpam04}. 
\end{remark}

\bigskip

Let us briefly explain the main steps of the proof of Theorem \ref{Theo:1}, which follows the method of  \cite{RSapde2014} treated for the critical case $d = 2$. We would like to mention that this kind of method has been successfully applied for various nonlinear evolution equations. In particular in the dispersive setting for the nonlinear Schr\"odinger equation both in the mass critical \cite{MRam05, MRcmp05, MRim04, MRgfa03} and mass supercritical \cite{MRRcjm15} cases; the mass critical gKdV equation \cite{MMRasp15, MMRjems15, MMRam14};  the energy critical \cite{DKMcjm13}, \cite{HRapde12} and supercritical \cite{Car161} wave equation; the two dimensional critical geometric equations: the wave maps \cite{RRmihes12}, the Schr\"odinger maps \cite{MRRim13} and the harmonic heat flow \cite{RScpam13, RSapde2014}; the semilinear heat equation \eqref{eq:she} in the energy critical \cite{Sjfa12} and supercritical \cite{Car16} cases; and the two dimensional Keller-Segel model \cite{RSma14}, \cite{GMarx16}. In all these works, the method is based on two arguments:
\begin{itemize}
\item Reduction of an infinite dimensional problem to a finite dimensional one, through the derivation of suitable Lyapunov functionals and the robust energy method as mentioned in the two step procedure above.
\item The control of the finite dimensional problem thanks to a topological argument based on index theory.
\end{itemize}
Note that this kind of topological arguments has benn proved to be successful also for the construction of type I blowup solutions for the semilinear heat equation \eqref{eq:she} in \cite{BKnon94}, \cite{MZdm97}, \cite{NZens16} (see also \cite{NZsns16} for the case of logarithmic perturbations, \cite{Breiumj90}, \cite{Brejde92} and \cite{GNZpre16a} for the exponential source, \cite{NZcpde15} for the complex-valued case), the Ginzburg-Landau equation in \cite{MZjfa08} (see also \cite{ZAAihn98} for an earlier work), a non-variational parabolic system in \cite{GNZpre16c} and the semilinear wave equation in \cite{CZcpam13}.\\

For the reader's convenience and for a better explanation, let's first introduce notations used throughout this paper.\\
\noindent - \textbf{Notation.} For each $d \geq 7$, we define
\begin{equation}\label{def:kdeltaplus}
\left\{\begin{array}{ll}
\hbar &= \left\lfloor \frac{1}{2}\left(\frac{d}{2} - \gamma\right) \right\rfloor \in \mathbb{N},\\
& \\
\delta &= \frac{1}{2}\left(\frac{d}{2} - \gamma\right) - \hbar, \quad  \delta \in (0,1),
\end{array} \right.
\end{equation}
where $\lfloor x \rfloor \in \mathbb{Z}$ stands for the integer part of $x$ which is defined by $\lfloor x \rfloor \leq x < \lfloor x \rfloor + 1$. Note that $\delta \ne 0$. Indeed, if $\delta = 0$, then there is $m \in \mathbb{N}$ such that $2\gamma = d - 4m \in \mathbb{N}$. This only happens when $\gamma = 2$ or $\gamma = \frac{3}{2}$ because $\gamma \in (1, 2]$. The case $\gamma = 2$ gives $d = 7$ and $m = \frac{3}{4} \not \in \mathbb{N}$. The case $\gamma = \frac{3}{2}$ gives $d = \frac{17}{2} \not \in \mathbb{N}$.

Given a large integer $L \gg 1$, we set
\begin{equation}\label{def:kbb}
\Bbbk = L + \hbar + 1.
\end{equation}

Given $b_1 > 0$ and $\lambda > 0$, we define 
\begin{equation}\label{def:B0B1}
B_0 = \frac{1}{\sqrt b_1}, \quad B_1 = B_0^{1 + \eta}, \quad 0 < \eta \ll 1,
\end{equation}
and denote by
$$f_\lambda(r) = f(y) \quad \text{with} \quad y = \frac{r}{\lambda}.$$

Let $\chi \in \Cc_0^\infty([0, +\infty))$ be a positive nonincreasing cutoff function with $\text{supp}(\chi) \subset [0,2]$ and $\chi \equiv 1$ on $[0,1]$. For all $M > 0$, we define
\begin{equation}\label{def:chiM}
\chi_M(y) = \chi\left(\frac y M\right).
\end{equation}
We also introduce the differential operator 
$$\Lambda f = y\partial_y f,$$
and the Schr\"odinger operator
\begin{equation}\label{def:Lc}
\Ls  = -\partial_{yy} - \frac{(d-1)}{y}\partial_y  + \frac{Z}{y^2}, \quad \text{with}\;\; Z(y)= (d-1)\cos(2Q(y)).
\end{equation}

\medskip

\noindent - \textbf{Strategy of the proof.}  We now resume the main ideas of the proof of Theorem \ref{Theo:1}, which follows the route map in \cite{RSapde2014} and \cite{MRRcjm15}.\\

\noindent $(i)$ \textit{Renormalized flow and iterated resonances.} Following the scaling invariance of \eqref{Pb}, let us make the change of variables
\begin{equation*}
w(y,s) = u(r,t), \quad y = \frac{r}{\lambda(t)}, \quad \frac{ds}{dt} = \frac{1}{\lambda^2(t)},
\end{equation*}
which leads to the following renormalized flow:
\begin{equation}\label{eq:wys_i}
\partial_sw = \partial_y^2 w + \frac{(d-1)}{y}\partial_y w  - b_1\Lambda w - \frac{(d-1)}{2y^2}\sin(2w), \quad b_1 = -\frac{\lambda_s}{\lambda}.
\end{equation}
Assuming that the leading part of the solution $w(y,s)$ is given by the ground state profile $Q$ admitting the asymptotic behavior \eqref{exp:Qr}, hence the remaining part is governed by the Schr\"odinger operator $\Ls$ defined by \eqref{def:Lc}. The linear operator $\Ls$ admits the factorization (see Lemma \ref{lemm:factorL} below) 
\begin{equation}\label{eq:facL_i}
\Ls = \As^*\As, \quad \As f = -\Lambda Q \py \left(\frac{f}{\Lambda Q} \right), \quad \As^*f = \frac{1}{y^{d-1}\Lambda Q}\py\left(y^{d-1}\Lambda Q f\right), 
\end{equation}
which directly implies
$$\Ls(\Lambda Q) = 0,$$
where from a direct computation,
$$\Lambda Q \sim \frac{c_0}{y^\gamma} \quad \text{as} \quad y \to +\infty \quad \text{with $\gamma$ definied in \eqref{def:gamome}}.$$
More generally, we can compute the kernel of the powers of $\Ls$ through the iterative scheme
\begin{equation}\label{def:Tk_i}
\Ls T_{k + 1} = - T_k, \quad T_0 = \Lambda Q,
\end{equation}
which displays a non trivial tail at infinity (see Lemma \ref{lemm:GenLk} below),
\begin{equation}\label{eq:Tk_i}
T_k(y) \sim c_ky^{2k - \gamma} \quad \text{for}\quad y \gg 1.
\end{equation}

\noindent $(ii)$ \textit{Tail computation.} Following the approach in \cite{RSapde2014}, we look for a slowly modulated approximate solution to \eqref{eq:wys_i} of the form 
$$w(y,s) = Q_{b(s)}(y),$$
where 
\begin{equation}\label{def:Qb_i}
b = (b_1, \cdots, b_L), \quad Q_{b(s)}(y) = Q(y) + \sum_{i = 1}^Lb_iT_i(y) + \sum_{i = 2}^{L+2}S_i(y)
\end{equation}
with a priori bounds
$$b_i \sim b_1^i, \quad |S_i(y)| \lesssim b_1^i y^{2(i - 1) - \gamma},$$
so that $S_i$ is in some sense homogeneous of degree $i$ in $b_1$, and behaves better than $T_i$ at infinity. The construction of $S_i$ with the above a priori bounds is possible for a specific choice of the universal dynamical system which drives the modes $(b_i)_{1 \leq i \leq L}$. This procedure is called the \textit{tail computation}. Let us illustrate the procedure of the \textit{tail computation}. We plug the decomposition \eqref{def:Qb_i} into \eqref{eq:wys_i} and choose the law for $(b_i)_{1 \leq i \leq L}$ which cancels the  leading order terms at infinity.\\
- At the order $\Oc(b_1)$: it is not possible to adjust the law of $b_1$ for the first term \footnote{if $(b_1)_s = -c_1 b_1$, then $-\lambda_s/\lambda \sim b_1 \sim e^{-c_1 s}$, hence after an integration in time, $|\log \lambda| \lesssim 1$ and there is no blowup.} and obtain from \eqref{eq:wys_i}, 
$$b_1(\Lc T_1 + \Lambda Q) = 0.$$
- At the order $\Oc(b_1^2, b_2)$: We obtain
$$(b_1)_sT_1 + b_1^2\Lambda T_1 + b_2\Ls T_2 + \Ls S_2 = b_1^2 NL_1(T_1, Q),$$
where $NL_1(T_1, Q)$ corresponds to the interaction of the nonlinear terms. Note from \eqref{eq:Tk_i} and \eqref{def:Tk_i}, we have
$$\Lambda T_1 \sim (2 - \gamma)T_1 \quad \text{for}\quad y \gg 1, \quad \Ls T_2 = - T_1,$$
and thus, 
$$(b_1)_sT_1 + b_1^2\Lambda T_1 + b_2\Ls T_2 \sim \big[(b_1)_s + (2 - \gamma)b_1^2 - b_2\big]T_1.$$
Hence, we cancel the highest order growth for $y$ large through the choice 
$$(b_1)_s + (2 - \gamma)b_1^2 - b_2 = 0.$$
We then solve for 
$$\Ls S_2 = -b_1^2(\Lambda T_1 - (2 - \gamma)T_1) + b_1^2 NL_1(T_1, Q),$$
and check the improved decay 
$$|S_2(y)| \lesssim b_1^2y^{2 - \gamma} \quad \text{for} \quad y \gg 1.$$
- At the order $\Oc(b_1^{k+1}, b_{k+1})$: we get an elliptic equation of the following form
$$(b_k)_sT_k + b_1b_k\Lambda T_k + b_{k + 1}\Ls T_{k+1} + \Ls S_{k+1} = b_1^{k+1}NL_k(T_1, \cdots, T_k, Q).$$ 
From \eqref{eq:Tk_i} and \eqref{def:Tk_i}, we have 
$$(b_k)_sT_k + b_1b_k\Lambda T_k + b_{k + 1}\Ls T_{k+1} \sim \big[(b_k)_s + (2k - \gamma)b_1b_k - b_{k+1}\big]T_{k},$$
which leads to the choice 
$$(b_k)_s + (2k - \gamma)b_1b_k - b_{k+1} = 0,$$
to cancel at infinity the leading order growth.
We then solve for  $S_{k+1}$ term and verify that $|S_{k+1}(y)|\lesssim b_1^{k+1}y^{2 k - \gamma}$ for $y$ large. We refer to Proposition \ref{prop:1} for all details of the \textit{tail computation}.\\

\noindent $(iii)$ \textit{The universal system of ODEs.} From the above procedure we deduce after $L$ iterations the following universal system of ODEs:
\begin{equation}\label{sys:bk_i}
\left\{ \begin{array}{l}
(b_k)_s + (2k - \gamma)b_1b_k - b_{k+1} = 0, \quad 1 \leq k \leq L, \quad b_{L+1} = 0,\\
\quad \\
-\dfrac{\lambda_s}{\lambda} = b_1, \quad \dfrac{ds}{dt} = \frac{1}{\lambda^2}.
\end{array}\right.
\end{equation}
Unlike the critical case treated in \cite{RSapde2014}, we do not need a logarithmic correction. The set of solutions to \eqref{sys:bk_i} (see Lemma \ref{lemm:solSysb} below) is explicitly given by 
\begin{equation}\label{eq:solbk_i}
\left\{\begin{array}{l}
b_k^e(s) = \frac{c_k}{s^k}, \quad 1 \leq k \leq L,\\
c_1 = \frac{\ell}{2\ell - \gamma}, \quad \ell \in \mathbb{N}^*, \; 2\ell > \gamma,\\
c_{k + 1} = -\frac{\gamma(\ell - k)}{2\ell - \gamma}c_k, \quad 1 \leq k \leq \ell -1,\ell\geq2\\
c_j = 0, \quad j \geq \ell + 1.\\
\lambda(s) \sim s^{-\frac{\ell}{2\ell - \gamma}}.
\end{array}\right.
\end{equation}
In the original time variable $t$, this implies that $\lambda(t)$ goes to zero in finite time $T$ with the asymptotic 
$$\lambda(t) \sim (T-t)^\frac{\ell}{\gamma}.$$
Furthermore, the linearized flow of \eqref{sys:bk_i} around the solution \eqref{eq:solbk_i} is explicit and manisfests $(\ell - 1)$ unstable directions (see Lemma \ref{lemm:lisysb} below). This induces that the case $\ell = 1$ corresponds to a stable type II blowup regime.\\

\noindent $(iv)$ \textit{The flow decomposition and the equations of modulation.} Consider $Q_b$ as the approximate solution given by \eqref{def:Qb_i} which from the construction will engender an approximate solution of the renormalized flow \eqref{eq:wys_i},
$$\Psi_b = \ps Q_b - \Delta Q_b + b \Lambda Q_b + \frac{(d-1)}{2y^2}\sin(2Q_b) = \textup{Mod}(t) + O(b_1^{2L + 2}),$$
where the modulation equation term is approximately given by:
$$\textup{Mod}(t) = \sum_{i = 1}^L \big[(b_i)_s + (2i - \gamma)b_1b_i - b_{i+1}\big]T_i.$$
We localize $Q_b$ in the zone $y \leq B_1$ to prevent the appearance of the irrelevant growing tails for $y \gg \frac{1}{\sqrt{b_1}}$. We then take initial data of the form
$$u_0(y) = Q_{b(0)}(y) + q_0(y),$$
where $q_0$ is small in some suitable sense and $b(0)$  is chosen to be close to the exact solution \eqref{eq:solbk_i}. The following decomposition of the flow is stanrdard consequence of the modulation theory
\begin{equation}\label{eq:dec_i}
u(r,t) = w(y,s) = \big(Q_{b(s)} + q\big)(y,s) = \big(Q_{b(t)} + v\big)\left(\frac{r}{\lambda(t)},t\right),
\end{equation}
where $(b(t), \lambda(t))$ are the $L+1$ modulation parameters and they are fixed to insure the following orthogonality conditions:
\begin{equation}\label{eq:orh_i}
\left<q, \Ls^i \Phi_M \right> = 0, \quad 0 \leq i \leq L,
\end{equation}
where $\Phi_M$ (see \eqref{def:PhiM}) is a well adapted approximation of the kernel of $\Ls$ which depends on a certain large constant $M$.
As a consequence, $\Ls^i \Phi_M$ produces also a well adapted approximation of the powers of $\Ls$ kernel. Thanks to the orthogonal decomposition \eqref{eq:dec_i} we compute the modulation equations driving the parameters $(b(t), \lambda(t))$ (see Lemmas \ref{lemm:mod1} and \ref{lemm:mod2} below),
\begin{equation}\label{eq:mod_i}
\left|\frac{\lambda_s}{\lambda} + b_1\right| +  \sum_{i = 1}^L \big|(b_i)_s + (2i - \gamma)b_1b_i - b_{i+1}\big| \lesssim \|q\|_{loc} + b_1^{L + 1 + \nu(\delta, \eta)},
\end{equation}
where $\|q\|_{loc}$ is the error measured in certain norm well adapted to our estimations.

\noindent $(v)$ \textit{Control of Sobolev norms.} From \eqref{eq:mod_i}, we see that it is important to prove that the local norms of $q$ remain small enough and do not disturb the dynamical system \eqref{sys:bk_i}. This is accomplished via the control of the error $q$ through norms that are well adapted to our estimations and coercive to some weighted sobolev norms. In particular, we have the following property which gives the control of some weighted sobolev norms under the orthogonality conditions \eqref{eq:orh_i} (see Lemma \ref{lemm:coeLk}),
$$\Es_{2\Bbbk}(s)  = \int |\Ls^\Bbbk q|^2 \gtrsim \int |\nabla^{2\Bbbk} q|^2 + \int \frac{|q|^2}{1 + y^{4\Bbbk}},$$
where $\Bbbk$ is given by \eqref{def:kbb}. Here the factorization \eqref{eq:facL_i} will simplify our proof. As in \cite{RRmihes12}, \cite{RSapde2014} and \cite{MRRcjm15}, the control of $\Es_{2\Bbbk}$ is done through the use of the linearized equation in the original variables $(r,t)$, i.e. we work with $v$ in \eqref{eq:dec_i} and not $q$. The energy estimate is of the form (see Proposition \ref{prop:E2k})
\begin{equation}\label{eq:Ek_i}
\frac{d}{ds} \left\{\frac{\Es_{2\Bbbk}}{\lambda^{4\Bbbk - d}}\right\} \lesssim \frac{b_1^{2L + 1 + 2\nu(\delta, \eta)}}{\lambda^{4\Bbbk - d}}, \quad \nu(\delta, \eta) > 0,
\end{equation}
where the right hand side is the size of the error $\Psi_b$ of the construction of our approximate profile $Q_b$ above. An integration of \eqref{eq:Ek_i} in time by using initial smallness assumptions, $b_1 \sim b_1^e$ and $\lambda(s) \sim b_1^{\frac{\ell}{2\ell - \gamma}}$ yields the estimate
$$
\int |\nabla^{2\Bbbk} q|^2 + \int \frac{|q|^2}{1 + y^{4\Bbbk}} \lesssim \Es_{2\Bbbk}(s) \lesssim b_1^{2L + 2\nu(\delta, \eta)},$$
which is good enough to control the local norms of $q$ and close the modulation equations \eqref{eq:mod_i}.

Note that we also need to control lower energies $\Es_{2m}$ for $\hbar + 2 \leq m \leq \Bbbk - 1$ because the control of the high energy $\Es_{2\Bbbk}$ alone is not sufficient to control the nonlinear term appearing after the linearization around $Q_b$. In particular, we exhibit a Lyapunov functional with the dynamical estimate
$$\frac{d}{ds}\left\{\frac{\Es_{2m}}{\lambda^{4m - d}}\right\} \lesssim \frac{b_1^{2(m - \hbar) - 1 + 2\nu'(\delta, \eta)}}{\lambda^{4m - d}}, \quad \nu'(\delta, \eta) > 0,$$
then, an integration in time yields
$$ \Es_{2m}(s) \lesssim \left\{ \begin{array}{ll}
b_1^{\frac{\ell}{2\ell - \gamma}(4m - d)}& \quad \text{for}\quad \hbar + 2 \leq m \leq \ell + \hbar,\\
b_1^{2(m - \hbar -1) + 2\nu'(\delta, \eta)}&\quad \text{for}\quad \hbar + \ell + 1 \leq m \leq \Bbbk - 1,
\end{array}
\right.$$
which is enough to control the nonlinear term. Let us remark that the condition $m \geq \hbar + 2$ ensures $4m - d > 0$ so that $\Es_{2m}$ is always controlled. By the coercivity of $\Es_{2m}$, this means that we are only able to control the Sobolev norms $\|\nabla^{2\sigma}q\|^2_{L^2}$ for $\sigma \geq \hbar +2$ resulting in the asymptotic \eqref{eq:asypqs}.\\

The above framework allows us to design an appreciated shrinking set (see Definition \ref{def:Skset} for a precise definition) which traps the solution according to the asymptotic dynamic described in Theorem \ref{Theo:1}. Following Lemma \ref{lemm:solSysb} and \ref{lemm:lisysb}, the control of the solution in this set reduces to the control of $(\ell - 1)$ unstable modes $(b_2, \cdots, b_\ell)$, which is a finite-dimensional problem.  We then solve this finite-dimensional problem through a topological argument based on the Brouwer fixed point theorem (see the proof of Proposition \ref{prop:exist}), and the proof of Theorem \ref{Theo:1} follows.\\

\bigskip 

The paper is organized as follows. In Section \ref{sec:2}, we give the construction of the approximate solution $Q_b$ of \eqref{Pb} and derive estimates on the generated error term $\Psi_b$ (Proposition \ref{prop:1}) as well as its localization (Proposition \ref{prop:localProfile}). We also give in this section some elementary facts on the study of the system \eqref{sys:bk_i} (Lemmas \ref{lemm:solSysb} and \ref{lemm:lisysb}). Section \ref{sec:3} is devoted to the proof of Theorem \ref{Theo:1} assuming a main technical result (Proposition \ref{prop:redu}). In particular, we give the proof of the existence of the solution trapped in some shrinking set to zero (Proposition \ref{prop:exist}) such that the constructed solution satisfies the conclusion of Theorem \ref{Theo:1}. Readers not interested in technical details may stop there. In Section \ref{sec:4}, we give the proof of Proposition \ref{prop:redu} which gives the reduction of the problem to a finite-dimensional one; and this is the heart of our analysis. 

\section{Construction of an approximate profile.}\label{sec:2}
This section is devoted to the construction of a suitable approximate solution to \eqref{Pb}  by using the same approach developed in \cite{RRmihes12}. Similar approachs can also be found in \cite{RScpam13}, \cite{HRapde12}, \cite{RSma14}, \cite{Sjfa12}, \cite{MRRcjm15}. The construction is mainly based on the fact that the kernel of the linearized operator $\Ls$ around $Q$ is given explicitly in the radial setting.

Following the scaling invariance of \eqref{Pb}, we introduce the following change of variables:
\begin{equation}\label{def:simiVars}
w(y,s) = u(r,t), \quad y = \frac{r}{\lambda(t)}, \quad \frac{ds}{dt} = \frac{1}{\lambda^2(t)},
\end{equation}
which leads to the following renormalized flow:
\begin{equation}\label{eq:wys}
\partial_sw = \partial_y^2 w + \frac{(d-1)}{y}\partial_y w  + \frac{\lambda_s}{\lambda}\Lambda w - \frac{(d-1)}{2y^2}\sin(2w),
\end{equation}
where $\lambda_s = \frac{d\lambda}{ds}$. Noticing that in the setting \eqref{def:simiVars}, we have 
$$\partial_r u(r,t) = \frac{1}{\lambda(t)}\partial_y w(y,s)$$
and since we deal with the finite time blowup of the problem \eqref{Pb}, we would naturally impose the condition 
$$\lambda(t) \to 0 \quad \text{as} \quad t \to T,$$
for some $T \in (0, +\infty)$. Hence, $\partial_r u(r,t)$ blows up in finite time $T$.

Let us assume that the leading part of the solution of \eqref{eq:wys} is given by the harmonic map $Q$, which is unique  solution (up to scaling) of the equation
\begin{equation}\label{eq:Qy}
Q'' + \frac{(d-1)}{y}Q' - \frac{(d-1)}{2y^2}\sin(2Q) = 0, \quad Q(0) = 0, \; Q'(0) = 1.
\end{equation}
We aim at constructing an approximate solution of \eqref{eq:wys} close to $Q$. The natural way is to linearize equation \eqref{eq:wys} around $Q$, which generates the Schr\"odinger operator defined by \eqref{def:Lc}. Let us now recall the main properties of $\Ls$ in the following subsection.

\subsection{Structure of the linearized Hamiltonian.}
In this subsection, we recall the main properties of the linearized operator  around $Q$, which are fundamental for the construction of the approximate profile. Actually, the properties of the linearized operator are also important to derive the coercivity properties used for the high Sobolev energy estimates. Let us start by recalling the following result from Biernat \cite{BIEnon2015}, which gives the asymptotic behavior of the harmonic map $Q$:
\begin{lemma}[Development of the harmonic map $Q$] Let $d \geq 7$, there exists a unique solution $Q$ to equation \eqref{eq:Qy}, which admits the following asymptotic behavior: For any $k \in\Nc^*$,\\
$(i)$ (Asymptotic behavior of $Q$) 
\begin{equation}\label{eq:asymQ}
Q(y) = \left\{\begin{array}{ll}
y + \sum \limits_{i = 1}^kc_iy^{2i + 1} + \Oc(y^{2k + 3}) &\text{as}\quad y \to 0, \\
&\\
\dfrac{\pi}{2} - \dfrac{a_0}{ y^{\gamma}}\left[1 + \Oc\left(\dfrac 1{y^{2}}\right) + \Oc\left(\dfrac 1 {y^{ \tilde{\gamma}}}\right)\right]\quad &\text{as} \quad y \to + \infty,
\end{array}
\right.
\end{equation}
where  $\gamma$ is defined in \eqref{def:gamome}, $\tilde{\gamma} = \sqrt{d^2 - 8d + 8}$ and the constant $a_0 = a_0(d) > 0$.\\
$(ii)$ (Degeneracy)
\begin{equation}\label{eq:asymLamQ}
\Lambda Q > 0, \quad \Lambda Q(y) = \left\{\begin{array}{ll}
y + \sum\limits_{i = 1}^kc_i' y^{2i + 1} + \Oc(y^{2k + 3}) &\text{as}\quad y \to 0, \\
&\\
\dfrac{a_0 \gamma}{ y^{\gamma}}\left[1 + \Oc\left(\dfrac 1{y^{2}}\right) + \Oc\left(\dfrac 1 {y^{ \tilde{\gamma}}}\right)\right]\quad &\text{as} \quad y \to + \infty.
\end{array}
\right.
\end{equation}
\end{lemma}
\begin{proof} The proof of \eqref{eq:asymQ} is done through the introduction of the variables $x = \log y$ and $v(x) = 2Q(y) - \pi$ and consists of the phase portrait analysis of the autonomous equation 
$$v''(x) + (d - 2)v'(x) + (d-2)\sin(v(x)) = 0.$$
All details of the proof can be found at pages 184-185 in \cite{BIEnon2015}. The proof of \eqref{eq:asymLamQ} directly follows from the expansion \eqref{eq:asymQ}. 
\end{proof}

The linearized operator $\Ls$ displays a remarkable structure given by the following lemma:
\begin{lemma}[Factorization of $\Ls$] \label{lemm:factorL}  Let $d \geq 7$ and define the first order operators
\begin{align}
\As w  &= -\partial_y w + \frac{V}{y}w = - \Lambda Q \partial_y \left(\frac{w}{\Lambda Q}\right), \label{def:As}\\ 
\As^* w &= \frac{1}{y^{d-1}}\partial_y \big(y^{d-1}w\big) + \frac{V}{y}w  =  \frac{1}{y^{d-1}\Lambda Q} \partial_y \left(y^{d-1} \Lambda Q w\right),\label{def:Astar}
\end{align}
where
\begin{equation}\label{eq:asympV}
V(y) := \Lambda \log(\Lambda Q) =  \left\{\begin{array}{ll}
1 + \Oc(y^2)\quad &\text{as}\quad y \to 0, \\
&\\
-\gamma + \Oc\left(\dfrac 1{y^{2}}\right)+ \Oc\left(\dfrac 1{y^{\tilde{\gamma}}}\right)\quad &\text{as} \quad y \to + \infty,
\end{array}
\right.
\end{equation}
We have
\begin{equation}\label{eq:reLAAst}
\Ls = \As^* \As, \quad \tilde{\Ls} = \As \As^*,
\end{equation}
where $\tilde{\Ls}$ stands for the conjugate Hamiltonian.
\end{lemma}

\begin{remark} The adjoint operator $\As^*$ is defined with respect to the Lebesgue measure
$$\int _0^{+\infty}(\As u) w y^{d-1}dy = \int_{0}^{+\infty} u(\As^* w)y^{d-1}dy.$$
\end{remark}

\begin{remark} We have 
\begin{equation}\label{eq:relLsLam}
\Ls(\Lambda w) = \Lambda (\Ls w) + 2 \Ls w - \frac{\Lambda Z}{y^2}w.
\end{equation}
Since $\Ls(\Lambda Q) = 0$, one can express the definition of $Z$ through the potential $V$ as follows: 
\begin{equation}\label{def:ZbyV}
Z(y) = V^2 + \Lambda V + (d-2)V.
\end{equation}
Let $\tilde{Z}$ be defined by 
\begin{equation}\label{def:LstilbyZtil}
\tilde{\Ls} = -\partial_{yy} - \frac{d-1}{y}\py + \frac{\tilde{Z}}{y^2},
\end{equation}
then, a direct computation yields
\begin{equation}\label{def:ZtilbyV}
\tilde{Z}(y) = (V + 1)^2 + (d-2)(V+1) - \Lambda V.
\end{equation}
\end{remark}

\bigskip

\noindent From \eqref{def:As} and \eqref{def:Astar}, we see that the kernel of $\As$ and $\As^*$ are explicit:
$$\left\{\begin{array}{ll}
\As w = 0 & \quad \text{if and only if}\quad w \in \text{Span}(\Lambda Q),\\
\As^* w = 0& \quad \text{if and only if}\quad w \in \text{Span}\left(\frac 1{y^{d-1}\Lambda Q}\right).\\
\end{array}\right.$$
Hence, the elements of the kernel of $\Ls$ are given by 
\begin{equation}\label{eq:kernalLc}
\Ls w = 0 \quad \text{if and only if}\quad w \in \text{Span}(\Lambda Q, \Gamma),
\end{equation}
where $\Gamma$ can be found from the Wronskian relation
\begin{equation}\label{eq:relWrons}
\Gamma' \Lambda Q - \Gamma (\Lambda Q)' = \frac{1}{y^{d-1}},
\end{equation}
that is
\begin{equation*}
\Gamma(y) = \Lambda Q(y) \int_1^y \frac{d\xi}{\xi^{d-1} (\Lambda Q(\xi))^2},
\end{equation*}
which admits the asymptotic behavior:
\begin{equation}\label{eq:asymGamma}
\Gamma(y) = \left\{\begin{array}{ll}
\dfrac{1}{d y^{d-1}} + \Oc(y)\;\; &\text{as}\;\; y \to 0, \\
&\\
\dfrac{1}{a_0 \gamma (d - 2 - 2\gamma) y^{d - 2 - \gamma}} + \Oc\left(\dfrac 1{y^{d - \gamma}}\right)\;\; &\text{as} \;\; y \to + \infty,
\end{array}
\right.
\end{equation}
From \eqref{eq:kernalLc}, we may invert $\Ls$ as follows:
\begin{equation}\label{eq:invLc}
\Ls^{-1}f =  -\Gamma(y)\int_0^y f(x)\Lambda Q(x) x^{d-1}dx + \Lambda Q(y)\int_0^y f(x) \Gamma(x)x^{d-1}dx. 
\end{equation}  
The factorization of $\Ls$ grants us to compute $\Ls^{-1}$ easily. In particular, we have the following:
\begin{lemma}[Inversion of $\Ls$] \label{lemm:inversionL} Let  $f$ be a $\Cc^\infty$ radially symmetric function and $w = \Ls^{-1}f$ be given by \eqref{eq:invLc}, then 
\begin{equation}\label{eq:relaAL}
\Ls w = f, \quad \As w = \frac{1}{y^{d-1}\Lambda Q} \int_0^yf(x) \Lambda Q(x) x^{d-1}dx, \quad w = -\Lambda Q\int_0^y\frac{\As w(x)}{\Lambda Q(x)}dx.
\end{equation}
\end{lemma}
\begin{proof} From the relation \eqref{eq:relWrons}, we compute
$$\As \Gamma = - \frac{1}{y^{d-1}\Lambda Q}.$$
Applying $\As$ to \eqref{eq:invLc} and using the cancellation $\As (\Lambda Q) = 0$, we obtain
$$\As w = \frac{1}{y^{d-1}\Lambda Q} \int_0^yf(x) \Lambda Q(x) x^{d-1}dx.$$
From the definition \eqref{def:As} of $\As$, we write
$$w = - \Lambda Q \int_0^y \frac{\As w}{\Lambda Q} dx.$$
This concludes the proof of Lemma \ref{lemm:inversionL}.
\end{proof}
\subsection{Admissible functions.} We define a class of admissible functions which exhibit a relevant behavior both at the origin and infinity.
\begin{definition}[Admissible function] \label{def:Admitfunc} Fix $\gamma > 0$, we say that a smooth function $f \in \Cc^\infty(\Rb_+, \Rb)$ is admissible of degree $(p_1, p_2) \in \mathbb{N} \times \mathbb{Z}$ if \\
$(i)\;$ $f$ admits a Taylor expansion to all orders around the origin, 
 $$f(y) = \sum_{k = p_1}^p c_ky^{2k + 1} + \Oc(y^{2p + 3});$$
$(ii)\;$ $f$ and its derivatives admit the bounds, for $y \geq 1$,
$$\forall k \in \mathbb{N}, \quad |\partial^k_y f(y)| \lesssim y^{2p_2 - \gamma - k}.$$
\end{definition}
\begin{remark} Note from \eqref{eq:asymLamQ} that $\Lambda Q$ is admissible of degree $(0,0)$.
\end{remark}

One note that $\Ls$ naturally acts on the class of admissible function in the following way:
\begin{lemma}[Action of $\Ls$ and $\Ls^{-1}$ on admissible functions] \label{lemm:actionLL} Let $f$ be an admissible function of degree $(p_1, p_2) \in \mathbb{N} \times \mathbb{Z}$, then:\\
$(i)\;$ $\Lambda f$ is admissible of degree $(p_1, p_2)$.\\
$(ii)\,$ $\Ls f$ is admissible of degree $(\max\{0,p_1 - 1\}, p_2 - 1)$.\\
$(iii)\,$ $\Ls^{-1}f$ is admissible of degree $(p_1 + 1, p_2 + 1)$.
\end{lemma}
\begin{proof} $(i) - (ii)$ This is simply a consequence of Definition \ref{def:Admitfunc}. \\
$(iii)$ We aim at proving that if $f$ is admissible of degree $(p_1, p_2)$, then $w = \Ls^{-1}f$ is admissible of degree $(p_1 + 1, p_2 + 1)$. To do so, we use Lemma \ref{lemm:inversionL} to estimate \\
- for $y \ll 1$, 
$$\As w = \frac{1}{y^{d-1}\Lambda Q}\int_0^yf\Lambda Q x^{d-1}dx =  \Oc \left(\frac{1}{y^d}\int_0^y x^{2p_1 + 1 + d}dx\right) = \Oc(y^{2p_1 + 2}),$$
$$w = - \Lambda Q \int_0^y\frac{\As w}{\Lambda Q}dx = \Oc\left(y\int_0^y x^{2p_1 + 1}dx \right) = \Oc(y^{2(p_1 + 1) + 1}),$$
- for $y \geq 1$, 
$$\As w =  \Oc \left(\frac{1}{y^{d-1 - \gamma}}\int_0^y x^{2p_2 -2\gamma + d - 1}dx\right) = \Oc(y^{2p_2 + 1 - \gamma}),$$
$$w = \Oc\left(\frac{1}{y^\gamma}\int_0^y x^{2p_2 + 1}\right) = \Oc(y^{2(p_2 + 1) - \gamma}).$$
From the last formula in \eqref{eq:relaAL} and \eqref{eq:asympV}, we estimate
$$\partial_y w = - \partial_y \Lambda Q \int_0^y \frac{\As w}{\Lambda Q}dx - \As w = -\frac{\partial_y \Lambda Q}{\Lambda Q} w  - \As w = \Oc(y^{2(p_2 + 1) - \gamma - 1}).$$
Using $\Ls w = f$, we get
$$\partial_{yy}w = \Oc \left(\frac{|\partial_y w|}{y} + \frac{|w|}{y^2} + |f|\right) =\Oc (y^{2(p_2+1) - \gamma - 2}).$$
By taking radial derivatives of $\Ls w = f$, we obtain by induction
$$|\partial_y^k w| \lesssim y^{2(p_2 + 1) - \gamma - k}, \quad k \in \mathbb{N}, \; y \geq 1.$$
This concludes the proof of Lemma \ref{lemm:actionLL}.
\end{proof}

The following lemma is a consequence of Lemma \ref{lemm:actionLL}:
\begin{lemma}[Generators of the kernel of $\Ls^k$] \label{lemm:GenLk} Let the sequence of profiles 
\begin{equation}\label{def:Tk}
T_k = (-1)^k\Ls^{-k} \Lambda Q, \quad k \in \mathbb{N},
\end{equation}
then\\
$(i)\;$ $T_k$ is admissible of degree $(k,k)$ for $k \in \mathbb{N}$.\\
$(ii)\;$ $\Lambda T_k - (2k - \gamma)T_k$ is admissible of degree $(k, k-1)$ for $k \in \mathbb{N}^*$.
\end{lemma}
\begin{proof} $(i)$ We note from \eqref{eq:asymLamQ} that $\Lambda Q$ is admissible of degree $(0,0)$. By induction and part $(iii)$ of Lemma \ref{lemm:actionLL}, the conclusion then follows.\\

$(ii)$ We proceed by induction. For $k = 1$, we explicitly compute $T_1 = -\Ls^{-1}\Lambda Q$ by using Lemma \ref{lemm:inversionL} and the expansion \eqref{eq:asymLamQ} to get
$$\forall m \in \mathbb{N}, \quad \partial_y^m T_1(y) =  e_{1,m}y^{2 - \gamma - m} + \Oc \left(y^{-\gamma - m}\right) \quad \text{as} \quad y \to +\infty.$$
By induction, one can easily check that $\partial_y^m \Lambda f = \Lambda \partial_y^mf + m \partial_y^mf$ for $m \in \mathbb{N}^*$. Hence,
$$\partial_y^m\Big[\Lambda T_1 - (2 - \gamma)T_1\Big] = \Lambda \partial_y^m T_1 - (2 - \gamma - m)\partial_y^mT_1 = \Oc \left(y^{-\gamma - m}\right) \quad \text{as} \quad y \to +\infty.$$
Since $T_1$ and $\Lambda T_1$ are admissible of degree $(1,1)$, we deduce that $\Lambda T_1 - (2 - \gamma)T_1$ is admissible of degree $(1,0)$.

We now assume the claim for $k \geq 1$, namely that $\Lambda T_k - (2k - \gamma)T_k$ is admissible of degree $(k, k - 1)$. Let us prove that $\Lambda T_{k+1} - (2(k+1) - \gamma)T_{k+1}$ is admissible of degree $(k+1, k)$. We use formula \eqref{eq:relLsLam} and definition \eqref{def:Tk} to write
\begin{align}
&\Ls \big(\Lambda T_{k+1} - (2k + 2 - \gamma)T_{k+1}\big)\nonumber\\
&\qquad \qquad = \Lambda \Ls T_{k+1} - (2k - \gamma)\Ls T_{k+1} - \frac{\Lambda Z}{y^2}T_{k+1}\nonumber\\
&\qquad \qquad  = \Lambda T_k - (2k - \gamma)T_k - \frac{\Lambda Z}{y^2}T_{k+1}.\label{eq:tmpLTk1}
\end{align}
From part $(i)$, we know that $T_{k+1}$ is admissible of degree $(k+1, k+1)$. From \eqref{def:ZbyV} and \eqref{eq:asympV}, one can check that $\frac{\Lambda Z}{y^2}T_{k+1}$ admits the asymptotic:
$$\frac{\Lambda Z}{y^2}T_{k+1} = \Oc(y^{2k + 1})  \quad \text{as} \quad y \to 0,$$
and 
$$\partial_y^j\left(\frac{\Lambda Z}{y^2}T_{k+1}\right) = \Oc(y^{2(k+1) - j - \gamma - 3}) \ll y^{2(k-1) + j - \gamma} \quad \text{as}\quad y \to +\infty.$$
Together with the induction hypothesis, we deduce that the right hand side of \eqref{eq:tmpLTk1} is admissible of degree $(k, k-1)$. The conclusion then follows by using part $(iii)$ of Lemma \ref{lemm:actionLL}. This ends the proof of Lemma \ref{lemm:GenLk}.
\end{proof}

We close this subsection by defining a simple notion of homogeneous admissible function.
\begin{definition}[Homogeneous admissible function] Let $L \gg 1$ be an integer and $m = (m_1, \cdots, m_L) \in \mathbb{N}^L$, we say that a function $f(b,y)$ with $b = (b_1, \cdots, b_L)$ is homogeneous of degree $(p_1, p_2, p_3) \in \mathbb{N} \times \mathbb{Z}\times \mathbb{N}$ if it is a finite linear combination of monomials
$$\tilde f (y)\prod_{k = 1}^Lb_k^{m_k},$$
with $\tilde{f}(y)$ admissible of degree $(p_1, p_2)$ in the sense of Definition \ref{def:Admitfunc} and
$$(m_1, \cdots, m_L) \in \mathbb{N}^L, \quad\sum_{k = 1}^L km_k = p_3.$$
We set 
$$\text{deg}(f):= (p_1, p_2, p_3).$$  
\end{definition}

\subsection{Slowly modulated blow-up profile.}
In this subsection, we use the explicit structure of the linearized operator $\Ls$ to construct an approximate blow-up profile. In particular, we claim the following:
\begin{proposition}[Construction of the approximate profile] \label{prop:1}  Let $d \geq 7$ and $L \gg 1$ be an integer. Let $M > 0$ be a large enough universal constant, then there exist a small enough universal constant $b^*(M,L) > 0$ such that the following holds true. Let a $\Cc^1$ map
 $$b = (b_1, \cdots, b_L):[s_0,s_1] \mapsto (-b^*, b^*)^L,$$
with a priori bounds in $[s_0,s_1]$:
\begin{equation}\label{eq:relb1bk}
0 < b_1 < b^*, \quad |b_k| \lesssim b_1^k, \quad 2 \leq k \leq L, 
\end{equation}
Then there exist homogeneous profiles 
$$S_1 = 0, \quad S_k = S_k(b,y), \quad 2 \leq k \leq L + 2,$$
such that
\begin{equation}\label{eq:Qbform}
Q_{b(s)}(y) = Q(y) + \sum_{k = 1}^L b_k(s)T_k(y) + \sum_{k = 2}^{L+2}S_k(b,y) \equiv Q(y) + \Theta_{b(s)}(y),
\end{equation}
produces an approximate solution to the remormalized flow \eqref{eq:wys}:
\begin{equation}\label{def:Psib}
\partial_s Q_{b} - \partial_{yy}Q_b - \frac{(d-1)}{y}\partial_yQ_b + b_1 \Lambda Q_b + \frac{(d-1)}{2y^2}\sin(2Q_b) = \Psi_b + \textup{Mod}(t),
\end{equation}
with the following property:\\
$(i)$ (Modulation equation)
\begin{equation}\label{eq:Modt}
\textup{Mod}(t) = \sum_{k = 1}^L\Big[(b_k)_s + (2k - \gamma)b_1b_k - b_{k + 1}\Big] \left[T_k + \sum_{j = k + 1}^{L+2}\frac{\partial S_j}{\partial b_k}\right], 
\end{equation}
where we use the convention $b_{j} = 0$ for $j \geq L+1$.\\
$(ii)$ (Estimate on the profiles) The profiles $(S_k)_{2 \leq k \leq L+2}$ are homogeneous with 
\begin{align*}
&\text{deg}(S_k) = (k,k-1,k) \quad \text{for} \quad 2 \leq k \leq L+2,\\
&\frac{\partial S_k}{\partial b_m} = 0 \quad \text{for}\quad 2 \leq k \leq m \leq L.
\end{align*}
$(iii)$ (Estimate on the error $\Psi_b$) For all $0 \leq m \leq L$, there holds:\\
- (global weight bound)
\begin{equation}\label{eq:estGlobalPsib}
\int_{y \leq 2B_1}|\Ls^{\hbar + m + 1} \Psi_b|^2 + \int_{y \leq 2B_1} \frac{|\Psi_b|^2}{1 + y^{4(\hbar + m + 1 )}} \lesssim b_1^{2m + 4 + 2(1 - \delta) - C_L\eta},
\end{equation}
where $B_1$, $\hbar$, $\delta$ are defined in \eqref{def:B0B1} and \eqref{def:kdeltaplus}.\\
- (improved local bound)
\begin{equation}\label{eq:estlocalPsib}
\forall M \geq 1, \quad \int_{y \leq 2M}|\Ls^{\hbar + m + 1} \Psi_b|^2 \lesssim M^Cb_1^{2L + 6}.
\end{equation}
\end{proposition}

\begin{proof} We aim at constructing the profiles $(S_k)_{2 \leq k \leq L+2}$ such that $\Psi_b(y)$ defined from \eqref{def:Psib} has the \emph{least possible growth} as $y \to +\infty$. The structure of the linearized operator $\Ls$ defined in \eqref{def:Lc} is the main ingredient for the construction of the approximate profile. This procedure will lead to the leading-order modulation equation
\begin{equation}
(b_k)_s = -(2k - \gamma)b_1b_k + b_{k+1} \quad \text{for}\quad 1 \leq k \leq L,
\end{equation}
which actually cancels the worst growth of $S_k$ as $y \to +\infty$.

\paragraph{$\bullet$ Expansion of $\Psi_b$}. From \eqref{def:Psib} and \eqref{eq:Qy}, we write
\begin{align*}
&\partial_s Q_{b} - \partial_{yy}Q_b - \frac{(d-1)}{y}\partial_yQ_b + b_1 \Lambda Q_b + \frac{(d-1)}{2y^2}\sin(2Q_b)\\
&= b_1\Lambda Q + \partial_s \Theta_b - \partial_{yy}\Theta_b - \frac{(d-1)}{y} \partial_y \Theta_y + \frac{(d-1)}{y^2}\cos(2Q)\Theta_b + b_1 \Lambda \Theta_b\\
&\qquad+ \frac{(d-1)}{2y^2}\left[\sin(2Q + 2\Theta_b) - \sin(2Q) - 2\cos(2Q)\Theta_b\right]:= A_1 + A_2.
\end{align*}
Using the expression \eqref{eq:Qbform} of $\Theta_b$ and the definition \eqref{def:Tk} of $T_k$ (note that $\Ls T_k = - T_{k-1}$ with the convention $T_0 = \Lambda Q$), we write 
\begin{align*}
A_1&= b_1 \Lambda Q + \sum_{k = 1}^L \Big[(b_k)_s T_k + b_k \Ls T_k + b_1b_k\Lambda T_k \Big] + \sum_{k = 2}^{L+2}\Big[\partial_s S_k + \Ls S_k + b_1 \Lambda S_k\Big]\\
&=  \sum_{k = 1}^L \Big[(b_k)_sT_k - b_{k+1}T_k + b_1b_k \Lambda T_k \Big] + \sum_{k = 2}^{L+2}\Big[\partial_s S_k + \Ls S_k + b_1 \Lambda S_k\Big]\\
&= \sum_{k = 1}^L \Big[(b_k)_s - b_{k+1} + (2k - \gamma)b_1b_k\Big]T_k\\
&\qquad + \sum_{k = 1}^L\Big[\Ls S_{k + 1} + \partial_s S_k + b_1b_k \big[\Lambda T_k - (2k - \gamma)T_k\big] + b_1 \Lambda S_k\Big]\\
& \qquad\qquad+ \Big[\Ls S_{L+2} + \partial_s S_{L+1} + b_1\Lambda S_{L+1}\Big] + \Big[\partial_sS_{L+2} + b_1\Lambda S_{L+2}\Big].
\end{align*}
We now write
\begin{align*}
\partial_sS_k = \sum_{j = 1}^L(b_j)_s\frac{\partial S_k}{\partial b_j} &= \sum_{j = 1}^L\Big[(b_j)_s + (2j - \gamma)b_1b_j - b_{j + 1} \Big]\frac{\partial S_k}{\partial b_j} \\
&\qquad - \sum_{j = 1}^L\Big[(2j - \gamma)b_1b_j - b_{j + 1} \Big]\frac{\partial S_k}{\partial b_j}.
\end{align*} 
Hence, 
\begin{align*}
A_1 = \textup{Mod}(t) &+ \sum_{k = 1}^{L+1}\left[\Ls S_{k + 1} + E_k\right] + E_{L+2},
\end{align*}
where for $k = 1, \cdots, L$,
\begin{equation}\label{def:Ek1}
E_k= b_1b_k \big[\Lambda T_k - (2k - \gamma)T_k\big] + b_1 \Lambda S_k - \sum_{j = 1}^{k-1}\Big[(2j - \gamma)b_1b_j - b_{j + 1} \Big]\frac{\partial S_k}{\partial b_j},
\end{equation} 
and for $k = L+1, L+2$,
\begin{equation}\label{def:EkL12}
E_k = b_1 \Lambda S_k - \sum_{j = 1}^{L}\Big[(2j - \gamma)b_1b_j - b_{j + 1} \Big]\frac{\partial S_k}{\partial b_j}.
\end{equation}

For the expansion of the nonlinear term $A_2$, let us denote 
$$f(x) = \sin(2x)$$
and use a Taylor expansion to write (see pages 1740 in \cite{RSapde2014} for a similar computation)
\begin{align*}
A_2 = \frac{(d-1)}{2y^2}\left[\sum_{i = 2}^{L+2}\frac{f^{(i)}(Q)}{i!}\Theta_b^i + R_2\right] = \frac{(d-1)}{2y^2}\left[\sum_{i = 2}^{L+2}P_i + R_1 + R_2\right],
\end{align*}
where 
\begin{equation}\label{def:Pj}
P_i = \sum_{j = 2}^{L+2}\frac{f^{(j)}(Q)}{j!}\sum_{|J|_1 = j, |J|_2 = i}c_J \prod_{k = 1}^Lb_k^{i_k}T_k^{i_k}\prod_{k=2}^{L+2}S_k^{j_k},
\end{equation}
\begin{equation}\label{def:R_1}
R_1 = \sum_{j = 2}^{L+2}\frac{f^{(j)}(Q)}{j!}\sum_{|J|_1 = j, |J|_2 \geq L+3}c_J \prod_{k = 1}^Lb_k^{i_k}T_k^{i_k}\prod_{k=2}^{L+2}S_k^{j_k},
\end{equation}
\begin{equation}\label{def:R2}
R_2 = \frac{\Theta_b^{L+3}}{(L+2)!}\int_0^1(1 - \tau)^{L+2}f^{(L+3)}(Q + \tau \Theta_b)d\tau,
\end{equation}
with $J = (i_1, \cdots, i_L, j_2, \cdots, j_{L+2}) \in \mathbb{N}^{2L + 1}$ and 
\begin{equation}\label{def:J1J2}
|J|_1 = \sum_{k = 1}^L i_k + \sum_{k = 2}^{L+2}j_k, \quad |J|_2 = \sum_{k=1}^L ki_k + \sum_{k = 2}^{L+2}kj_k.
\end{equation}
In conclusion, we have 
\begin{equation}\label{eq:expanPsib}
\Psi_b = \sum_{k = 1}^{L+1}\left[\Ls S_{k + 1} + E_k + \frac{(d-1)}{2y^2}P_{k+1}\right] + E_{L+2} + \frac{(d-1)}{2y^2}(R_1 + R_2)
\end{equation}

\paragraph{$\bullet$ Construction of $S_k$.} From the expression of $\Psi_b$ given in \eqref{eq:expanPsib}, we construct iteratively the sequences of profiles $(S_k)_{1 \leq k \leq L+2}$ through the scheme
\begin{equation}\label{def:Sk}
\left\{\begin{array}{ll}
S_1 &= 0, \\
S_k &= - \Ls^{-1}F_k, \quad 2 \leq k \leq L+2,
\end{array}\right.
\end{equation}
where 
$$F_k = E_{k-1} + \frac{(d-1)}{2y^2}P_{k} \quad \text{for} \quad 2\leq k \leq L+2.$$
We claim by induction on $k$ that $F_k$ is homogeneous with 
\begin{equation}\label{eq:degFk}
\text{deg}(F_k) = (k-1, k-2, k) \quad \text{for} \quad 2 \leq k \leq L+2,
\end{equation}
and 
\begin{equation}\label{eq:estparFk}
\frac{\partial F_k}{\partial b_m} = 0 \quad \text{for}\quad 2 \leq k \leq m \leq L+2.
\end{equation}
From item $(iii)$ of Lemma \ref{lemm:actionLL} and \eqref{eq:degFk}, we deduce that $S_k$ is homogeneous of degree
$$\text{deg}(S_k) = (k,k-1,k) \quad \text{for}\quad 2 \leq k \leq L+2,$$
and from \eqref{eq:estparFk}, we get
$$\frac{\partial S_k}{\partial b_m} = 0 \quad \text{for} \quad 2 \leq k \leq m \leq L+2,$$
which is the conclusion of item $(ii)$.\\

Let us now give the proof of \eqref{eq:degFk} and \eqref{eq:estparFk}. We proceed by induction.\\
\noindent - Case $k = 2$: We compute explicitly from \eqref{def:Ek1} and \eqref{def:Pj},
$$F_2 = E_1 + \frac{(d-1)}{2y^2}P_2 = b_1^2\left[\Lambda T_1 - (2 - \gamma)T_1 + \frac{(d-1)f''(Q)}{2y^2}T_1^2\right],$$
which directly follows \eqref{eq:estparFk}. From Lemma \ref{lemm:GenLk}, we know that $T_1$ and $\Lambda T_1 - (2 - \gamma)T_1$ are admissible of degree $(1,1)$ and $(1,0)$ respectively. Using \eqref{eq:asymQ}, one can check the bound
\begin{equation}\label{eq:estfjm}
\forall m,j \in \mathbb{N}^2, \quad\left|\partial_y^m \left(\frac{f^{(j)}(Q)}{y^2}\right)\right| \lesssim y^{-\gamma - 2 - m} \quad \text{as}\quad y \to +\infty.
\end{equation}
Since $T_1$ is admissible of degree $(1,1)$, we have that
$$\forall m \in \mathbb{N}, \quad |\partial_y^m(T_1^2)| \lesssim y^{4- 2\gamma - m}  \quad \text{as}\quad y \to +\infty.$$
By the Leibniz rule and the fact that $2\gamma - 2 > 0$, we get that
$$\forall m,j \in \mathbb{N}^2, \quad\left|\partial_y^m \left(\frac{f^{(j)}(Q)}{y^2} T_1^2\right)\right| \lesssim y^{- \gamma - m - (2\gamma - 2)} \lesssim y^{-\gamma - m}.$$
We also have the expansion near the origin, 
$$\frac{f^{(j)}(Q)}{y^2}T_1^2 = \sum_{i = 2}^k c_iy^{2i + 1} + \Oc(y^{2k + 3}), \quad k \geq 1.$$
Hence, $\frac{f''(Q)}{y^2}T_1^2$ is admissible of degree $(2,0)$, which concludes the proof of \eqref{eq:degFk} for $k = 2$.\\

\noindent - Case $k \to k+1$: Estimate \eqref{eq:estparFk} holds by direct inspection. Let us now assume that $S_k$ is homogeneous of degree $(k, k-1, k)$ and prove that $S_{k + 1}$ is homogeneous of degree $(k+1, k, k+1)$. In particular, the claim immediately follows from part $(iii)$ of Lemma \ref{lemm:actionLL} once we show that $F_{k+1}$ is homogeneous with
\begin{equation}\label{eq:Fk1EkPk1}
\text{deg}(F_{k+1}) = \text{deg}\left(E_k + \frac{P_{k+1}}{y^2} \right) = (k, k-1, k+1).
\end{equation}
From part $(ii)$ of Lemma \ref{lemm:GenLk} and the a priori assumption \eqref{eq:relb1bk}, we see that $b_1b_k(\Lambda T_k - (2k - \gamma)T_k)$ is homogeneous of degree $(k, k-1, k+1)$. From part $(i)$ of Lemma \ref{lemm:actionLL} and the induction hypothesis, $b_1 \Lambda S_k$ is also homogeneous of degree $(k, k-1, k+1)$. By definition, $b_1\frac{\partial S_k}{\partial b_1}$ is homogeneous and has the same degree as $S_k$. Thus, 
$$\left((2j - \gamma)b_1 - \frac{b_{2}}{b_{1}}\right)\left(b_1\frac{\partial S_k}{\partial b_1}\right)$$
is homogeneous of degree $(k, k-1, k+1)$. From definitions \eqref{def:Ek1} and \eqref{def:EkL12}, we derive  
$$\text{deg}(E_k) = (k, k-1, k+1), \quad k \geq 1.$$
It remains to control the term $\frac{P_{k+1}}{y^2}$. From the definition \eqref{def:Pj}, we see that $\frac{P_{k+1}}{y^2}$ is a linear combination of monomials of the form 
$$M_J(y) = \frac{f^{(j)}(Q)}{y^2}\prod_{m = 1}^L b_m^{i_m}T_m^{i_m}\prod_{m = 2}^{L+2}S_m^{j_m},$$
with 
$$J= (i_1, \cdots, i_L, j_2, \cdots, j_{L+2}), \quad |J|_1 = j, \; |J|_2 = k+1, \; 2 \leq j \leq k+1.$$
Recall from part $(i)$ of Lemma \ref{lemm:GenLk} the bound
$$\forall n \in \mathbb{N}, \quad |\partial_y^n T_m| \lesssim y^{2m - \gamma - n} \quad \text{as} \quad y \to +\infty,$$
and from the induction hypothesis and the a priori bound \eqref{eq:relb1bk},
$$\forall n \in \mathbb{N}, \quad |\partial_y^n S_m| \lesssim b_1^{m}y^{2(m-1) - \gamma - n} \quad \text{as} \quad y \to +\infty. $$
Together with the bound \eqref{eq:estfjm}, we obtain the following bound at infinity,
$$| M_J|\lesssim b_1^{|J|_2} y^{2|J|_2 - \gamma - |J|_1 \gamma - 2 - 2\sum_{m=2}^{L+2}j_m} \lesssim b_1^{k+1}y^{2(k-1) - \gamma}.$$
The control of $\partial_y^n M_J$ follows by the Leibniz rule and the above estimates. One can also check that $M_J$ is of order $y^{2|J|_2 + |J|_1 - 1}$ near the origin. This concludes the proof of \eqref{eq:Fk1EkPk1} as well as part $(ii)$ of Proposition \ref{prop:1}.\\

\paragraph{$\bullet$ Estimate on $\Psi_b$.} From \eqref{eq:expanPsib} and \eqref{def:Sk}, the expression of $\Psi_b$ is now reduced to 
$$\Psi_b = E_{L+2} + \frac{(d-1)}{y^2}(R_1 + R_2),$$
where $E_{L+2}$, $R_1$ and $R_2$ are given by \eqref{def:EkL12}, \eqref{def:R_1} and \eqref{def:R2}.

We start by estimating $E_{L+2}$ term defined by \eqref{def:EkL12}. Since $S_{L+2}$ is homogeneous of degree $(L+2, L+1, L+2)$ and thus so are $\Lambda S_{L+2}$ and $b_1 \frac{\partial S_{L+2}}{\partial b_1}$. This follows that $E_{L+2}$ is homogeneous of degree $(L+2, L+1, L+3)$. Using part $(ii)$ of Lemma \ref{lemm:actionLL} and the relation $d - 2\gamma - 4\hbar = 4\delta$ (see \eqref{def:kdeltaplus}), we estimate for all $0 \leq m \leq L$
\begin{align*}
\int_{y \leq 2B_1}|\Ls^{\hbar + m + 1}E_{L+2}|^2 &\lesssim b_1^{2L+6}\int_{y \leq 2B_1}|y^{2(L+1) - \gamma - 2(\hbar + m + 1)}|^2 y^{d-1}dy\\
&\quad\lesssim b_1^{2L+6}\int_{y \leq 2B_1}y^{4(L - m + \delta) - 1}dy\\
&\qquad \lesssim b_1^{(2L + 6) - 2(L - m + \delta)(1 + \eta)}\\
&\qquad \quad \lesssim b_1^{2m + 4 + 2(1 - \delta) - C_L\eta},
\end{align*}
where $\eta = \eta(L)$, $0 < \eta \ll 1$.

We now turn to the control of the term $\frac{R_1}{y^2}$, which is a linear combination of terms of the form (see \eqref{def:R_1})
$$\tilde{M}_J = \frac{f^{(j)}(Q)}{y^2}\prod_{n = 1}^Lb_n^{i_n}T_n^{i_n}\prod_{n = 2}^{L+2}S_k^{j_n},$$
with 
$$J = (i_1, \cdots, i_L, j_2, \cdots, j_{L+2}), \; |J|_1 = j, \; |J|_2 \geq L+3, \; 2 \leq j \leq L+2.$$
Using the admissibility of $T_n$ and the homogeneity of $S_n$, we get the bounds
$$|\tilde{M}_J| \lesssim b_1^{L + 3}y^{2|J|_2 + j - 1} \lesssim b_1^{L+3}y^{2L + 6} \quad \text{as} \quad y \to 0,$$
and 
$$|\tilde{M}_J| \lesssim b_1^{|J|_2}y^{2|J|_2 - j\gamma - 2 - \gamma}  \quad \text{as}\quad y \to +\infty,$$
where we used the fact that $j \geq 2$ and $2 - j\gamma < 0$, and similarly for higher derivatives by the Leibniz rule. Thus, we obtain the round estimate for all $0 \leq m \leq L$,
\begin{align*}
\int_{y \leq 2B_1}\left|\Ls^{\hbar + m + 1}\left(\frac{R_1}{y^2}\right)\right|^2 &\lesssim b_1^{2|J|_2}\int_{y \leq 2B_1}|y^{2|J|_2 -j \gamma-\gamma-2 - 2(\hbar + m + 1)}|^2 y^{d-1}dy\\
&\quad\lesssim b_1^{2m + 4 + 2(1 - \delta) - C_L\eta}.
\end{align*}
The term $\frac{R_2}{y^2}$ is estimated exactly as for the term $\frac{R_1}{y^2}$ using the definition \eqref{def:R2}. Similarly, the control of $\int_{y \leq 2B_1} \frac{|\Psi_b|^2}{(1 + y^{4(\hbar + m + 1)})}$ is obtained along the exact same lines as above. This concludes the proof of \eqref{eq:estGlobalPsib}. The local estimate \eqref{eq:estlocalPsib} directly follows from the homogeneity of $S_k$ and the admissibility of $T_k$. This concludes the proof of Proposition \ref{prop:1}.
\end{proof}

\bigskip

We now proceed to a simple localization of the profile $Q_b$ to avoid the growth of tails in the region $y \geq 2B_1 \gg B_0$. More precisely, we claim the following:

\begin{proposition}[Estimates on the localized profile] \label{prop:localProfile} Under the assumptions of Proposition \ref{prop:1}, we assume in addition the a priori bound 
\begin{equation}\label{eq:apriorib1}
|(b_1)_s| \lesssim b_1^2.
\end{equation}
Consider the localized profile
\begin{equation}\label{def:Qbtil}
\tilde{Q}_{b(s)}(y) = Q(y) + \sum_{k = 1}^Lb_k\tilde{T}_k + \sum_{k = 2}^{L+2}\tilde{S}_k \quad \text{with}\quad \tilde{T}_k = \chi_{B_1}T_k, \; \tilde{S}_k = \chi_{B_1}S_k,
\end{equation}
where $B_1$ and $\chi_{B_1}$ are defined as in \eqref{def:B0B1} and \eqref{def:chiM}. Then
\begin{equation}\label{def:Psibtilde}
\partial_s \tilde Q_{b} - \partial_{yy}\tilde Q_b - \frac{(d-1)}{y}\partial_y \tilde Q_b + b_1 \Lambda \tilde Q_b + \frac{(d-1)}{2y^2}\sin(2\tilde Q_b) = \tilde \Psi_b + \chi_{B_1}\, \textup{Mod}(t),
\end{equation}
where $\tilde \Psi_{b}$ satisfies the bounds:\\

\noindent $(i)\;$ (Large Sobolev bound) For all $0 \leq m \leq L-1$,
\begin{align}
&\int |\Ls^{\hbar + m + 1} \tilde{\Psi}_b|^2 + \int \frac{|\As\Ls^{\hbar + m} \tilde{\Psi}_b|^2}{1 + y^2} \nonumber\\
& \qquad + \int \frac{|\Ls^{\hbar + m} \tilde{\Psi}_b|^2}{1 + y^4} + \int \frac{|\tilde \Psi_b|^2}{1 + y^{4(\hbar + m + 1)}} \lesssim b_1^{2m + 2 + 2(1 - \delta)  - C_L \eta},\label{eq:estPsibLarge1}
\end{align}
and 
\begin{align}
&\int |\Ls^{\hbar + L + 1} \tilde{\Psi}_b|^2 + \int \frac{|\As\Ls^{\hbar + L} \tilde{\Psi}_b|^2}{1 + y^2} \nonumber\\
& \qquad + \int \frac{|\Ls^{\hbar + L} \tilde{\Psi}_b|^2}{1 + y^4} + \int \frac{|\tilde \Psi_b|^2}{1 + y^{4(\hbar + L + 1)}} \lesssim b_1^{2L + 2 + 2(1 - \delta)(1 + \eta)},\label{eq:estPsibLarge2}
\end{align}
where $\hbar$ and $\delta$ are defined by \eqref{def:kdeltaplus}. \\

\noindent $(ii)$ (Very local bound) For all $M \leq \frac{B_1}{2}$ and $0 \leq m \leq L$, 
\begin{equation}\label{eq:estPsiblocalTilde}
\int_{y \leq 2M} |\Ls^{\hbar +  m + 1} \tilde{\Psi}_b|^2 \lesssim M^Cb_1^{2L + 6}.
\end{equation}

\noindent $(iii)$ (Refined local bound near $B_0$) For all $0 \leq m \leq L$, 
\begin{equation}\label{eq:estPsiblocalB0}
\int_{y \leq 2B_0}|\Ls^{\hbar + m + 1}\tilde\Psi_b|^2 + \int_{y \leq 2B_0}\frac{|\tilde \Psi_b|^2}{1 + y^{4(\hbar +m + 1)}} \lesssim b_1^{2m + 4 + 2(1 - \delta) - C_L\eta}.
\end{equation}
%
\end{proposition}

\begin{proof} By a direct computation, we have
\begin{align*}
&\partial_s \tilde Q_{b} - \partial_{yy}\tilde Q_b - \frac{(d-1)}{y}\partial_y \tilde Q_b + b_1 \Lambda \tilde Q_b + \frac{(d-1)}{2y^2}\sin(2\tilde Q_b)\\
& =\chi_{B_1}\left[\partial_s Q_{b} - \partial_{yy}Q_b - \frac{(d-1)}{y}\partial_y Q_b + b_1 \Lambda Q_b + \frac{(d-1)}{2y^2}\sin(2 Q_b) \right] \\
& \quad + \Theta_b \left[ \partial_s \chi_{B_1} - \left(\partial_{yy}\chi_{B_1} + \frac{d-1}{y}\partial_y \chi_{B_1}\right)  + b_1 \Lambda \chi_{B_1}\right]  - 2\partial_y \chi_{B_1}\partial_y\Theta_b + b_1(1 - \chi_{B_1})\Lambda Q \\
& \qquad + \frac{(d-1)}{2y^2}\left[\sin(2\tilde Q_b) - \sin(2Q) - \chi_{B_1}(\sin(2Q_b) - \sin(2Q)) \right].
\end{align*}
According to \eqref{def:Psib} and \eqref{def:Psibtilde}, we write
$$\tilde{\Psi}_b = \chi_{B_1}\Psi_b + \hat\Psi_b,$$
where 
\begin{align*}
\hat \Psi_b &=\underbrace{ b_1(1 - \chi_{B_1})\Lambda Q}_{\hat \Psi_b^{(1)}}\\
& \quad + \underbrace{\frac{(d-1)}{2y^2}\left[\sin(2\tilde Q_b) - \sin(2Q) - \chi_{B_1}(\sin(2Q_b) - \sin(2Q)) \right]}_{\hat \Psi_b^{(2)}}\\
& \qquad +\underbrace{ \Theta_b \left[ \partial_s \chi_{B_1} - \left(\partial_{yy}\chi_{B_1} + \frac{d-1}{y}\partial_y \chi_{B_1}\right)  + b_1 \Lambda \chi_{B_1}\right]  - 2\partial_y \chi_{B_1}\partial_y\Theta_b}_{\hat \Psi_b^{(3)}}.
\end{align*}

The contribution of the term $\chi_{b_1}\Psi_b$ to the bounds \eqref{eq:estPsibLarge1}, \eqref{eq:estPsibLarge2}, \eqref{eq:estPsiblocalTilde} and \eqref{eq:estPsiblocalB0} follows exactly the same as in the proof of \eqref{eq:estGlobalPsib} and \eqref{eq:estlocalPsib}. We therefore left to estimate the term $\hat{\Psi}_b$. Since all the terms in the expression of $\hat{\Psi}_b$ are localized in $B_1 \leq y \leq 2B_1$ except for the first one whose support is a subset of $\{y \geq B_1\}$. Hence, the estimates \eqref{eq:estPsiblocalTilde} and \eqref{eq:estPsiblocalB0} directly follow from \eqref{eq:estlocalPsib} and \eqref{eq:estGlobalPsib}. 

Let us now find the contribution of $\hat \Psi_b$ to the bounds \eqref{eq:estPsibLarge1} and \eqref{eq:estPsibLarge2}. We estimate 
\begin{equation*}
\forall n \in \mathbb{N}, \quad \left|\frac{d^n}{dy^n}(1 - \chi_{B_1})\Lambda Q\right| \lesssim \frac{1}{y^{\gamma + n}}\mathbf{1}_{y \geq B_1},
\end{equation*}
hence, using the relation $d - 2\gamma -4\hbar = 4\delta$ (see \eqref{def:kdeltaplus}) and the definition \eqref{def:B0B1} of $B_1$, we estimate for all $0 \leq m \leq L$,
\begin{equation*}
\int|\Ls^{\hbar + m + 1} \hat \Psi_b^{(1)}|^2 \lesssim b_1^2 \int_{y \geq B_1}\frac{y^{d-1}}{y^{4(\hbar + m + 1) + 2\gamma}} \lesssim b_1^{2m + 2 + 2(1 - \delta)(1 + \eta) + 2m\eta}.
\end{equation*}
For the nonlinear term $\hat \Psi_b^{(2)}$, we note from the admissibility of $T_k$ and the homogeneity of $S_k$ that the $T_k$ terms dominate for $y \geq B_1$ in $\Theta_b$. Thus, for $y \geq B_1$,
\begin{equation}\label{eq:estThetab}
\forall n \in \mathbb{N}, \quad \left|\partial_y^n \Theta_b \right| \lesssim \sum_{k = 1}^L b_1^ky^{2k - \gamma - n}\mathbf{1}_{y \geq B_1}.
\end{equation}
Using \eqref{eq:estThetab} and noting that $\hat \Psi_b^{(2)}$ is localized in $B_1 \leq y \leq 2B_1$, we obtain the round bound
\begin{align*}
\left|\partial_y^n \hat \Psi_b^{(2)}\right|\quad \lesssim \sum_{k = 1}^L b_1^ky^{2(k-1) - \gamma - n} \mathbf{1}_{B_1 \leq y \leq 2B_1}  \lesssim\frac{b_1}{y^{\gamma + n}}\sum_{k = 1}^Lb_1^{-(k-1)\eta}\mathbf{1}_{B_1 \leq y \leq 2B_1}.
\end{align*}
We then estimate for $0 \leq m \leq L$, 
\begin{align*}
\int|\Ls^{\hbar + m + 1} \hat \Psi_b^{(2)}| &\lesssim b_1^2\sum_{k=1}^L b_1^{-2(k-1)\eta} \int_{B_1 \leq y \leq 2B_1}\frac{y^{d-1}}{y^{4(\hbar + m + 1) + 2\gamma}}dy\\
& \quad \lesssim b_1^{2m + 2 + 2(1 - \delta)(1 + \eta)}\sum_{k=1}^L b_1^{(2m - 2k + 2)\eta}.
\end{align*}
To control $\hat \Psi_b^{(3)}$, we first note from the definition \eqref{def:chiM} and the assumption \eqref{eq:apriorib1} that 
$$|\partial_s \chi_{B_1}| \lesssim \frac{(b_1)_s}{b_1} \frac{y}{B_1} \mathbf{1}_{B_1\leq y \leq 2B_1}\lesssim b_1 \mathbf{1}_{B_1\leq y \leq 2B_1}.$$
Using \eqref{eq:estThetab}, we estimate for $0 \leq m \leq L$, 
\begin{align*}
\int |\Ls^{\hbar + m + 1} \hat \Psi_b^{(3)}| &\lesssim \sum_{k=1}^L b_1^2 b_1^{2k}\int_{B_1 \leq y \leq 2B_1}\frac{y^{d-1}}{y^{4(\hbar + m + 1) + 2\gamma - 4k + 2}}dy\\
&\quad  \lesssim b_1^{2m + 2 + 2(1 - \delta)(1 + \eta)}\sum_{k=1}^L b_1^{(2m - 2k)\eta}.
\end{align*}
Gathering all the bounds yields
\begin{align*}
\int |\Ls^{\hbar + m + 1}\hat \Psi_b|^2 &\lesssim b_1^{2m + 2 + 2(1 - \delta)(1 + \eta)}\sum_{k=1}^L b_1^{(2m - 2k)\eta}
\lesssim b_1^{2m + 2 + 2(1 - \delta)(1 + \eta)+2\eta(m-L)}.
\end{align*}
The control of  $\int \frac{|\As \Ls^{\hbar + m}\tilde \Psi_b|^2}{1 + y^2}$, $\int \frac{|\Ls^{\hbar + m}\tilde \Psi_b|^2}{1 + y^4}$ and $\int \frac{|\hat \Psi_b|^2}{1 + y^{4(\hbar + m + 1)}}$ are obtained along the exact same lines as above. This concludes the proof  of \eqref{eq:estPsibLarge1} and \eqref{eq:estPsibLarge2} as well as Proposition \eqref{prop:localProfile}.


\end{proof}
\subsection{Study of the dynamical system for $b = (b_1, \cdots, b_L)$.}
The construction of the $Q_b$ profile formally leads to the finite dimensional dynamical system for $b = (b_1, \cdots, b_L)$ by setting to zero the inhomogeneous $\textup{Mod}(t)$ term given in \eqref{eq:Modt}:
\begin{equation}\label{eq:systemb}
(b_k)_s + (2k - \gamma)b_1b_k - b_{k+1} = 0, \;\; 1 \leq k \leq L, \;\; b_{L+1} = 0.
\end{equation}
Unlike the critical case $(d=2)$ treated in \cite{RSapde2014}, there is no more logarithmic correction to be taken into account in the system \eqref{eq:systemb}. In particular, the system \eqref{eq:systemb} admits explicit solutions and the linearized operator near these solutions is explicit. 

\begin{lemma}[Solution to the system \eqref{eq:systemb}]\label{lemm:solSysb} Let 
$$\frac{\gamma}{2} < \ell \ll L, \quad \ell \in \mathbb{N}^*$$
and the sequence
\begin{equation}\label{def:ck}
\left\{\begin{array}{ll}
c_1 &= \frac{\ell}{2\ell - \gamma},\\
c_{k + 1} &= - \frac{\gamma(\ell - k)}{2\ell - \gamma}c_k, \quad 1 \leq k \leq \ell-1,\\
c_{k + 1} &= 0, \quad k \geq \ell.
\end{array} \right.
\end{equation}
Then the explicit choice 
\begin{equation}\label{eq:solbe}
b_k^e(s) = \frac{c_k}{s^k}, \quad s > 0, \quad 1 \leq k \leq L,
\end{equation}
is a solution to \eqref{eq:systemb}.
\end{lemma}
The proof of Lemma \ref{lemm:solSysb} can be left to the reader since it directly follows from an explicit computation. We claim that the linearized flow of \eqref{eq:systemb} near the solution \eqref{eq:solbe} is explicit and exhibits $(\ell - 1)$ unstable directions. Note that the stability is considered in the sense that 
$$\sup_s s^k|b_k(s)| \leq C_k, \quad 1 \leq k \leq L.$$ 
In particular, we have the following result which was proved in \cite{MRRcjm15}:
\begin{lemma}[Linearization of \eqref{eq:systemb} around \eqref{eq:solbe}] \label{lemm:lisysb} Let 
\begin{equation}\label{eq:Ukbke}
b_k(s) = b_k^e(s) + \frac{\Uc_k(s)}{s^k}, \quad 1 \leq k \leq \ell,
\end{equation}
and note $\Uc = (\Uc_1, \cdots, \Uc_\ell)$. Then, for $1 \leq k \leq \ell-1$,
\begin{equation}\label{eq:bkk1}
(b_k)_s + (2k - \gamma)b_1b_k - b_{k+1} = \frac{1}{s^{k+1}}\left[s(\Uc_k)_s - (A_\ell \Uc)_k + \Oc(|\Uc|^2)\right], 
\end{equation}
and 
\begin{equation}\label{eq:bell}
(b_\ell)_s + (2\ell - \gamma)b_1b_\ell = \frac{1}{s^{k+1}}\left[s(\Uc_\ell)_s - (A_\ell \Uc)_\ell + \Oc(|\Uc|^2)\right], 
\end{equation}
where 
$$A_\ell = (a_{i,j})_{1\leq i,j\leq \ell} \quad \text{with}\quad \left\{\begin{array}{ll}
a_{1,1} &= \frac{\gamma(\ell - 1)}{2\ell - \gamma} - (2 - \gamma)c_1,\\
a_{i,i} &= \frac{\gamma(\ell - i)}{2\ell - \gamma}, \quad 2 \leq i \leq \ell,\\
a_{i, i+1}&= 1 , \quad 1 \leq i \leq \ell - 1,\\
a_{1,i} & = -(2i - \gamma)c_i, \quad 2 \leq i \leq \ell,\\
a_{i,j} &= 0 \quad \textup{ortherwise}
\end{array} \right.$$
Moreover, $A_\ell$ is diagonalizable: 
\begin{equation}\label{eq:diagAlPl}
A_\ell = P_\ell^{-1}D_\ell P_\ell, \quad D_\ell = \textup{diag}\left\{-1, \frac{2\gamma}{2\ell - \gamma}, \frac{3\gamma}{2\ell - \gamma}, \cdots, \frac{\ell \gamma}{2\ell - \gamma}\right\}.
\end{equation}
\end{lemma}
\begin{proof} Since we have an analogous system as the one in \cite{MRRcjm15} and the proof is essentially the same as written there, we kindly refer the reader to see Lemma 3.7 in \cite{MRRcjm15} for all details of the proof.  
\end{proof}

\section{Proof of Theorem \ref{Theo:1} assuming technical results.} \label{sec:3}
This section is devoted to the proof of Theorem \ref{Theo:1}. We hope that the explanation of the strategy we give in this section will be reader friendly. We proceed in 3 subsections:\\
- In the first subsection, we give an equivalent formulation of the linearization of the problem in the setting \eqref{eq:dec_i}.\\
- In the second subsection, we prepare the initial data and define the shrinking set $\Sc_K$ (see Definition \ref{def:Skset}) such that the solution trapped in this set satisfies the conclusion of Theorem \ref{Theo:1}.\\
- In the third subsection, we give all arguments of the proof of the existence of solutions trapped in $\Sc_K$ (Proposition \ref{prop:exist}) assuming an important technical result (Proposition \ref{prop:redu}) whose proof is left to the next section. Then we conclude the proof of Theorem \ref{Theo:1}. 

\subsection{Linearization of the problem.} 
Let $L \gg 1$ be an integer and $s_0 \gg 1$, we introduce the renormalized variables:
\begin{equation}
y = \frac{r}{\lambda(t)}, \quad s = s_0 + \int_0^t \frac{d\tau}{\lambda^2(\tau)},
\end{equation} 
and the decomposition
\begin{equation}\label{def:qys}
u(r,t) = w(y,s) = \big(\tilde{Q}_{b(s)} + q\big)(y,s) = \big(\tilde{Q}_{b(t)} + q\big)\left(\frac{r}{\lambda(t)}, t\right),
\end{equation}
where $\tilde Q_b$ is constructed in Proposition \ref{prop:localProfile} and the modulation parameters 
$$\lambda(t) > 0, \quad b(t) = (b_1(t), \cdots, b_L(t))$$
are determined from the $L+1$ orthogonality conditions:
\begin{equation}\label{eq:orthqPhiM}
\left<q, \Ls^k \Phi_M \right> = 0, \quad 0 \leq k \leq L,
\end{equation} 
where $\Phi_M$ is a fixed direction depending on some large constant $M$ defined by
\begin{equation}\label{def:PhiM}
\Phi_M = \sum_{k = 0}^L c_{k,M}\Ls^k(\chi_M \Lambda Q),
\end{equation}
with 
\begin{equation}\label{def:ckM}
c_{0,M} = 1, \quad c_{k, M} = (-1)^{k+1}\frac{\sum_{j = 0}^{k - 1}c_{j,M}\left<\chi_M \Ls^j(\chi_M \Lambda Q), T_k\right> }{\left<\chi_M \Lambda Q, \Lambda Q\right>}, \quad 1 \leq k \leq L.
\end{equation}
Here, $\Phi_M$ is build to ensure the nondegeneracy 
\begin{equation}\label{eq:PhiMLamQ}
\left<\Phi_M, \Lambda Q\right> = \left<\chi_M \Lambda Q, \Lambda Q\right> \gtrsim M^{d - 2\gamma},
\end{equation}
and the cancellation 
\begin{equation}\label{id:TkPhiM0}
\left<\Phi_M, T_k\right> = \sum_{j = 0}^{k - 1}c_{j,M}\left<\Ls^j(\chi_M\Lambda Q), T_k\right> + c_{k,M} (-1)^k \left<\chi_M \Lambda Q, \Lambda Q\right> = 0,
\end{equation}
In particular, we have 
\begin{equation}\label{id:TkPhiMi}
\left<\Ls^iT_k, \Phi_M \right> = (-1)^k\left<\chi_M \Lambda Q, \Lambda Q\right>\delta_{i,k}, \quad 0 \leq i,k\leq L.
\end{equation}

\medskip

From \eqref{eq:wys}, we see that $q$ satisfies the equation:
\begin{equation}\label{eq:qys}
\partial_s q - \frac{\lambda_s}{\lambda} \Lambda q + \Ls q = - \tilde{\Psi}_b  - \widehat{\textup{Mod}} + \Hc(q) - \Nc(q) \equiv \Fc,
\end{equation}
where 
\begin{equation}\label{def:Modhat}
\widehat{\textup{Mod}} = - \left(\frac{\lambda_s}{\lambda} + b_1\right)\Lambda \tilde{Q}_b - \chi_{B_1}\textup{Mod},
\end{equation}
and $\Hc$ is the linear part given by 
\begin{equation}\label{def:Lq}
\Hc(q) = \frac{(d-1)}{y^2}\big[\cos(2Q) - \cos(2\tilde Q_b)\big]q,
\end{equation}
and $\Nc$ is the purely nonlinear term
\begin{equation}\label{def:Nq}
\Nc(q) = \frac{(d-1)}{2y^2}\big[\sin(2\tilde Q_b + 2q) - \sin(2\tilde Q_b) - 2q \cos(2\tilde{Q}_b) \big].
\end{equation}
We also need to write the equation \eqref{eq:qys} in the original variables. To do so, let the rescaled linearized operator:
\begin{equation}\label{def:Llambda}
\Ls_\lambda = - \partial_{rr} - \frac{(d-1)}{r}\partial_r + \frac{Z_\lambda}{r^2},
\end{equation}
and the renormalized function
$$v(r,t) = q(y,s), \quad \partial_t v = \frac{1}{\lambda^2(t)}\left(\partial_s q - \frac{\lambda_s}{\lambda} \Lambda q\right)_\lambda$$
then from \eqref{eq:qys}, $v$ satisfies the equation
\begin{equation}\label{eq:vrt}
\partial_t v + \Ls_\lambda v = \frac{1}{\lambda^2}\Fc_\lambda, \quad \Fc_\lambda(r,t) = \Fc(y,s).
\end{equation}
Note that 
$$\Ls_\lambda = \frac{1}{\lambda^2}\Ls.$$

\subsection{Preparation of the initial data.}
We describe in this subsection the set of initial data $u_0$ of the problem \eqref{Pb} as well as the initial data for $(b, \lambda)$ leading to the blowup scenario of Theorem \ref{Theo:1}. Assume that $u_0 \in H^\infty(\Rd)$ satisfies
\begin{equation}\label{eq:asumUosmall}
\|u_0 - Q\|_{\dot{H}^s} \ll 1, \quad \text{for}\quad \frac{d}{2} \leq s \leq \Bbbk.
\end{equation}
From the continuity of the flow and a standard argument, we can propagate the smallness assumption \eqref{eq:asumUosmall} on a small time interval $[0,t_1)$. Hence, the decomposition \eqref{def:qys},
\begin{equation}\label{def:qrt}
u(r,t) = \left(\tilde Q_{b(t)} + q\right)\left(\frac{r}{\lambda(t)}, t\right), \quad \lambda(t) > 0, \; b = (b_1, \cdots, b_L),
\end{equation}
can be defined uniquely on the interval $t \in [0,t_1]$. 

The existence of the decomposition \eqref{def:qrt} is from the implicit function theorem and the explicit relations
$$\left.\frac{\partial}{\partial \lambda}(\tilde{Q}_{b(t)})_{\lambda}, \frac{\partial}{\partial b_1}(\tilde{Q}_{b(t)})_{\lambda}, \cdots, \frac{\partial}{\partial b_L}(\tilde{Q}_{b(t)})_{\lambda}  \right|_{\lambda = 1, b = 0} = (\Lambda Q, T_1, \cdots, T_L),$$
which induces the nondegeneracy of the Jacobian
$$\left|\left<\frac{\partial}{\partial (\lambda, b_j)}(\tilde{Q}_{b(t)})_{\lambda}, \Ls^i\Phi_M \right>_{1 \leq j \leq L, 0 \leq i \leq L} \right|_{\lambda = 1, b = 0} = \left|\left<\chi_M \Lambda Q, \Lambda Q\right>\right|^{L+1} \ne 0.$$
In fact, the decomposition \eqref{def:qrt} exists as long as  $t < T$ and $q$ remains small in the energy topology. We now set up the shrinking set for the control of the parameters $(b,\lambda)$ and the error $q$. To measure the regularity of the map we use the following coercive norms of $q$:\\
\begin{equation}
\Es_{2k} = \int|\Ls^{k}q|^2 \geq C(M)\sum_{m=0}^{k-1} \int \frac{|\Ls^m q|^2}{1 + y^{4(k - m)}} \quad \text{for}\quad \hbar + 1 \leq k \leq \Bbbk,
\end{equation}

Our construction is build on a careful choice of the initial data for the modulation parameter $b$ and the radiation $q$ at time $s = s_0$. In particular, we will choose it in the following way:

\begin{definition}[Choice of the initial data] \label{def:1}  Given $\eta$ and $\delta$ as in \eqref{def:B0B1} and \eqref{def:kdeltaplus}. Let consider the variable 
\begin{equation}\label{def:VctoUc}
\Vc = P_\ell \Uc,
\end{equation}
where $\Uc = (\Uc_1, \cdots, \Uc_\ell)$ is introduced in the linearization \eqref{eq:Ukbke}, namely that 
$$\Uc_k = s^k b_k - c_k \quad \text{with $c_k$ is given by \eqref{def:ck}},$$

and $P_\ell$ refers to the diagonalization \eqref{eq:diagAlPl} of $A_\ell$.

Let $s_0 \geq 1$, we assume that 
\begin{itemize}
\item Smallness of the initial perturbation for the $b_k$ unstable modes: 
\begin{equation}
|s_0^{\frac{\eta}{2}(1 - \delta)}\Vc_k(s_0)| < 1 \quad \text{for}\;\; 2 \leq k \leq \ell,
\end{equation}
\item Smallness for the $b_k$ stable modes of the initial perturbation: 
\begin{equation}\label{eq:initbk}
|s_0^{\frac{\eta}{2}(1 - \delta)}\Vc_1(s_0)| < 1, \quad |b_k(s_0)| <s_0^{-\frac{5\ell(2k - \gamma)}{2\ell - \gamma}} \text{for}\;\; \ell+1 \leq k \leq L,
\end{equation}
\item Smallness of the data:
\begin{equation}\label{eq:intialbounE2m}
\sum_{k = \hbar + 2}^\Bbbk \Es_{2k}(s_0) < s_0^{-\frac{10L\ell}{2\ell - \gamma}},
\end{equation}
\item Normalization:  we can assume up to a fixed rescaling:
\begin{equation}
\lambda(s_0) = 1.
\end{equation}
\end{itemize}
\end{definition}

In particular, the initial data described in Definition \ref{def:1} belongs to the following set which shrinks to zero as $s \to +\infty$:

\begin{definition}[Definition of the shrinking set]\label{def:Skset} Given $\eta$ and $\delta$ as in \eqref{def:B0B1} and \eqref{def:kdeltaplus}. For all $K \geq 1$ and $s \geq 1$, we define $\Sc_K(s)$ as the set of all $(b_1(s), \cdots, b_L(s), q(s))$ such that
\begin{align*}
&\left|\Vc_k(s) \right| \leq 10s^{-\frac{\eta}{2}(1 - \delta)} \quad \text{for}\;\; 1 \leq k \leq \ell,\\
&|b_k(s)| \leq s^{-k} \quad \text{for}\;\; \ell + 1 \leq k \leq L,\\
&\Es_{2\Bbbk}(s) \leq Ks^{-(2L + 2(1 - \delta)(1 + \eta))},\\
&\Es_{2m}(s) \leq \left\{ \begin{array}{ll}
Ks^{- \frac{\ell}{2\ell - \gamma}(4m - d)} &\quad \text{for} \quad \hbar + 2 \leq m \leq \ell + \hbar,
\\
s^{- 2(m - \hbar -1) - 2(1- \delta) + K\eta} &\quad \text{for} \quad \ell + \hbar  + 1 \leq m \leq \Bbbk - 1.
\end{array}  
\right.
\end{align*}
\end{definition}

\begin{remark} Note from \eqref{eq:Ukbke} that the bounds given in Definition \ref{def:Skset} imply that for $\eta$ small enough, 
$$b_1(s) \sim \frac{c_1}{s}, \quad |b_k(s)| \lesssim |b_1(s)|^k,$$
hence, the choice of the initial data $(b(s_0), q(s_0))$ belongs in $\Sc_K(s_0)$ if $s_0$ is large enough.
\end{remark}

\begin{remark} The introduction of the high Sobolev norm $\Es_{2\Bbbk}$ is reflected on the following relation:
\begin{equation}\label{eq:modeq}
\left|\frac{\lambda_s}{\lambda} + b_1\right| + \sum_{k = 1}^L \left|(b_k)_s + (2k - \gamma)b_1b_{k} - b_{k+1} \right| \lesssim C(M)\sqrt{\Es_{2\Bbbk}} + l.o.t, 
\end{equation}
which is computed thanks to the $(L+1)$ orthogonality conditions \eqref{eq:orthqPhiM} (see lemmas \ref{lemm:mod1} and \ref{lemm:mod2} below).
\end{remark}

\subsection{Existence of solutions trapped in $\Sc_K(s)$ and conclusion of Theorem \ref{Theo:1}.}

We claim the following proposition:
\begin{proposition}[Existence of solutions trapped in $\Sc_K(s)$] \label{prop:exist} There exists $K_1 \geq 1$ such that for $K \geq K_1$, there exists $s_{0,1}(K)$ such that for all $s_0 \geq s_{0,1}$, there exists initial data for the unstable modes 
$$(\Vc_2(s_0), \cdots, \Vc_\ell(s_0)) \in \left[-s_0^{-\frac{\eta}{2}(1 - \delta)},s_0^{-\frac{\eta}{2}(1 - \delta)}\right]^{\ell - 1},$$
such that the corresponding solution $(b(s),q(s)) \in \Sc_K(s)$ for all $s \geq s_0$.
\end{proposition}
Let us briefly give the proof of Proposition \ref{prop:exist}. Let us consider $K \geq 1$ and $s_0 \geq 1$ and $(b(s_0), q(s_0))$ as in Definition \ref{def:1}. We introduce the exit time 
$$s_* = s_*(b(s_0), q(s_0)) = \sup\{s \geq s_0 \; \text{such that}\; (b(s), q(s)) \in \Sc_K(s)\},$$
and assume that for any choice of 
$$(\Vc_2(s_0), \cdots, \Vc_\ell(s_0)) \in \left[-s_0^{-\frac{\eta}{2}(1 - \delta)},s_0^{-\frac{\eta}{2}(1 - \delta)}\right]^{\ell - 1},$$
the exit time $s_* < +\infty$ and look for a contradiction. By the definition of $\Sc_K(s_*)$, at least one of the inequalities in that definition is an equality. Owing the following proposition, this can happens only for the components $(\Vc_2(s_*), \cdots, \Vc_\ell(s_*))$. Precisely, we have the following result which is the heart of our analysis:

\begin{proposition}[Control of $(b(s), q(s))$ in $\Sc_K(s)$ by $(\Vc_2(s), \cdots, \Vc_\ell(s))$] \label{prop:redu} There exists $K_2 \geq 1$ such that for each $K \geq K_2$, there exists $s_{0,2}(K) \geq 1$ such that for all $s_0 \geq s_{0,2}(k)$, the following holds: Given the initial data at $s = s_0$ as in Definition \ref{def:1}, if $(b(s), q(s)) \in \Sc_K(s)$ for all $s \in [s_0, s_1]$, with $(b(s_1), q(s_1)) \in \partial \Sc_K(s_1)$ for some $s_1 \geq s_0$, then:\\
$(i)$ (Reduction to a finite dimensional problem)
$$(\Vc_2(s_1), \cdots, \Vc_\ell(s_1)) \in \partial\left[- \frac{K}{s_1^{\frac{\eta}{2}(1 - \delta)}}, \frac{K}{s_1^{\frac{\eta}{2}(1 - \delta)}} \right]^{\ell - 1}.$$
$(ii)$ (Transverse crossing)
$$\frac{d}{ds}\left(\sum_{i = 2}^\ell \left|s^{\frac{\eta}{2}(1 - \delta)}\Vc_i(s)\right|^2\right)_{\big|_{s = s_1} }> 0.$$
\end{proposition}

Let us assume Proposition \ref{prop:redu} and continue the proof of Proposition \ref{prop:exist}. From part $(i)$ of Proposition \ref{prop:redu}, we see that 
$$(\Vc_2(s_*), \cdots, \Vc_\ell(s_*)) \in \partial\left[- \frac{K}{s_*^{\frac{\eta}{2}(1 - \delta)}}, \frac{K}{s_*^{\frac{\eta}{2}(1 - \delta)}} \right]^{\ell - 1},$$
and the following mapping 
\begin{align*}
\Upsilon: [-1,1]^{\ell - 1} &\mapsto \partial \left([-1,1]^{\ell - 1}\right)\\
s_0^{\frac{\eta}{2}(1 - \delta)}\big(\Vc_2(s_0), \cdots, \Vc_\ell(s_0)\big) & \to \frac{s_*^{\frac{\eta}{2}(1 - \delta)}}{K} \big(\Vc_2(s_*), \cdots, \Vc_\ell(s_*)\big)
\end{align*}
is well defined. Applying the transverse crossing property given in part $(ii)$ of Proposition \ref{prop:redu}, we see that $(b(s), q(s))$ leaves $\Sc_K(s)$ at $s = s_0$, hence, $s_* = s_0$. This is a contradiction since $\Upsilon$ is the identity map on the boundary sphere and it can not be a continuous retraction of the unit ball. This concludes the proof of Proposition \ref{prop:exist}, assuming that Proposition \ref{prop:redu} holds.

\paragraph{Conclusion of Theorem \ref{Theo:1} assuming Proposition \ref{prop:redu}}. From Proposition \ref{prop:exist}, we know that there exist initial data $(b(s_0), q(s_0))$ such that 
$$(b(s), q(s)) \in \Sc_K(s) \quad \text{for all} \quad s \geq s_0.$$ 
From \eqref{eq:lam10}, \eqref{eq:Lamdas}, we have 
\begin{align*}
-\lambda \lambda_t = c(u_0)\lambda^{\frac{2\ell - \gamma}{\ell}} \left[1+o(1)\right],
\end{align*}
which yields
$$-\lambda^{1 - \frac{2\ell - \gamma}{\ell}}\lambda_t = c(u_0)(1 + o(1)).$$
We easily conclude that $\lambda$ vanishes in finite time $T = T(u_0) < +\infty$ with the following behavior near the blowup time:
$$\lambda(t) = c(u_0)(1 + o(1))(T-t)^{\frac{\ell}{\gamma}},$$
which is the conclusion of item $(i)$ of Theorem \ref{Theo:1}. 

For the control of the Sobolev norms, we observe from \eqref{eq:weipykq} and Definition \ref{def:Skset} that 
$$\forall \hbar + 2 \leq m \leq \Bbbk, \quad\int |\py^{2m}q|^2 \lesssim \Es_{2m} \to 0 \quad \text{as}\;\;s \to +\infty.$$
From the relation $d = 4\hbar + 4\delta + 2\gamma$, we deduce that 
$$\forall \sigma \in \left[d/2 + 3, 2\Bbbk\right], \quad \int |\nabla^\sigma q|^2 \to 0  \quad \text{as}\;\;s \to +\infty,$$
which yields $(ii)$ of Theorem \ref{Theo:1}. 

\section{Reduction of the problem to a finite dimensional one.} \label{sec:4}

In this section, we aim at proving Proposition \ref{prop:redu} which is the heart of our analysis. We proceed in three separate subsections:

- In the first subsection, we derive the laws for the parameters $(b,\lambda)$ thanks to the orthogonality condition \eqref{eq:orthqPhiM} and the coercivity of the powers of $\Ls$.

- In the second subsection, we prove the main monotonicity tools for the control of the infinite dimensional part of the solution. In particular, we derive a suitable Lyapunov functional for the $\Es_{2\Bbbk}$ energy as well as the monotonicity formula for the lower Sobolev energy.

- In the third subsection, we conclude the proof of Proposition \ref{prop:redu} thanks to the identities obtained in the first two parts.

\subsection{Modulation equations.}
We derive here the modulation equations for $(b, \lambda)$. The derivation is mainly based on the orthogonality \eqref{eq:orthqPhiM} and the coercivity of the powers of $\Ls$. Let us start with elementary estimates relating to the fixed direction $\Phi_M$.
\begin{lemma}[Estimate for $\Phi_M$]\label{lemm:estPhiM} Given $\Phi_M$ as defined in \eqref{def:PhiM}, we have the followings:
$$|c_{k,M}| \lesssim M^{2k} \quad \text{for all}\;\; 1 \leq k \leq L,$$
$$\int |\Phi_M|^2 \lesssim M^{d-2\gamma}, \quad \int|\Ls \Phi_M|^2 \lesssim M^{d - 2\gamma-4}.$$
\end{lemma}
\begin{proof}
Arguing by induction, we assume that
\begin{equation*}
|c_{j,M}| \lesssim M^{2j}, \quad 1 \leq j \leq k.
\end{equation*}
Using the fact that $\Ls^j T_{i}$ is admissible of degree $(\max\{0,i-j\},i-j)$, we estimate from the definition \eqref{def:ckM},
\begin{align*}
|c_{k+1,M}|& \lesssim \frac{1}{M^{d - 2\gamma}}\sum_{j = 0}^k M^{2j} \int |\chi_M \Lambda Q \Ls^j(T_{k+1})|\\
&\lesssim \frac{1}{M^{d - 2\gamma}}\sum_{j = 0}^k M^{2j} \int_{y \leq M}\frac{y^{d-1}}{y^\gamma}y^{2(k+1 - j) - \gamma}dy \lesssim M^{2(k+1)}.
\end{align*}
Using the estimate for $c_{k,M}$ yields
\begin{equation*}
\int |\Phi_M|^2 \lesssim \int |\chi_M \Lambda Q|^2 + \sum_{j = 1}^L |c_{j,M}|^2\int|\Ls^j(\chi_M \Lambda Q)|^2 \lesssim M^{d-2\gamma-4},
\end{equation*}
and 
\begin{equation*}
\int|\Ls \Phi_M|^2 \lesssim \sum_{j = 0}^L |c_{j,M}|^2\int |\Ls^{j+1}(\chi_M \Lambda Q)|^2  \lesssim M^{d - 2\gamma}.
\end{equation*}
This concludes the proof of Lemma \ref{lemm:estPhiM}.
\end{proof}

From the orthogonality conditions \eqref{eq:orthqPhiM} and equation \eqref{eq:qys}, we claim the following:
\begin{lemma}[Modulation equations] \label{lemm:mod1}  Given $\hbar$, $\delta$ and $\eta$ as defined in \eqref{def:kdeltaplus} and \eqref{def:B0B1}. For $K \geq 1$, we assume that there is $s_0(K) \gg 1$ such that $(b(s), q(s)) \in \Sc_K(s)$ for $s \in [s_0, s_1]$ for some $s_1 \geq s_0$. Then, the followings hold for $s \in [s_0, s_1]$:
\begin{equation}\label{eq:ODEbkl}
\sum_{k = 1}^{L-1}\left|(b_k)_s + (2k - \gamma)b_1b_k - b_{k+1} \right| + \left|b_1 + \frac{\lambda_s}{\lambda}\right| \lesssim b_1^{L + 1 + (1 - \delta)(1 + \eta)},
\end{equation}
and 
\begin{equation}\label{eq:ODEbL}
\left|(b_L)_s + (2L - \gamma)b_1b_L \right| \lesssim  \frac{\sqrt{\Es_{2\Bbbk}}}{M^{2\delta}} + b_1^{L + 1+ (1 - \delta)(1 + \eta)}.
\end{equation}
\end{lemma}
\begin{proof} We start with the law for $b_L$. Let 
$$D(t) = \left|b_1 + \frac{\lambda_s}{\lambda}\right| + \sum_{k = 1}^L \left|(b_k)_s + (2k - \gamma)b_1b_k - b_{k+1}\right|,$$
where we recall that $b_{k} \equiv 0$ if $k \geq L + 1$.

Now, we take the inner product of \eqref{eq:qys} with $\Ls^L \Phi_M$ and use the orthogonality \eqref{eq:orthqPhiM} to write
\begin{align}
\left<\widehat{\text{Mod}}(t), \Ls^L\Phi_M\right> &= -\left<\Ls^L\tilde{\Psi}_b, \Phi_M \right> -\left<\Ls^{L+1} q, \Phi_M \right>\nonumber\\
&\qquad  -\left<-\frac{\lambda_s}{\lambda}\Lambda q - \Lc(q) + \Nc(q), \Ls^L\Phi_M \right>.\label{eq:bL}
\end{align}
From the definition \eqref{def:PhiM}, we see that $\Phi_M$ is localized in $y \leq 2M$. From \eqref{def:Modhat} and \eqref{eq:Modt}, we compute by using the identity \eqref{id:TkPhiMi}, 
$$\left<\widehat{\text{Mod}}(t), \Ls^L\Phi_M\right> = (-1)^L\left<\Lambda Q, \Phi_M \right>\left[(b_L)_s + (2L - \gamma)b_1b_L\right] + \Oc(M^C b_1D(t)).$$
The error term is estimated by using \eqref{eq:estlocalPsib} with $m = L - \hbar - 1$ and Lemma \ref{lemm:estPhiM}, 
\begin{align*}
\left|\left<\Ls^L\tilde{\Psi}_b, \Phi_M \right>\right| &\leq \left(\int_{y \leq 2M}|\Ls^L \tilde{\Psi_b}|^2 \right)^\frac 12 \left(\int_{y \leq 2M} |\Phi_M|^2 \right)^\frac 12\\
&\quad \lesssim M^Cb_1^{L+3} \lesssim b_1^{L+1+(1-\delta)(1 + \eta)}.
\end{align*}
For the linear term, we apply Lemma \ref{lemm:coeLk} with $k = \Bbbk - 1$, 
\begin{align*}
\Es_{2\Bbbk}(q) \gtrsim \int\frac{|{\Ls}^{L+1}q|^2}{y^4(1 + y^{4(\hbar-1)})} \gtrsim \int\frac{|{\Ls}^{L+1}q|^2}{1 + y^{4\hbar}}. 
\end{align*}
Hence, the Cauchy-Schwartz inequality yields,
\begin{align*}
\left|\left<\Ls^{L+1} q,\Phi_M \right>\right| \lesssim M^{2\hbar}\left(\int \frac{|\Ls^{L+1}q|^2}{1 + y^{4\hbar}} \right)^{\frac 12} \left(\int |\Phi_M|^2 \right)^{\frac 12} \lesssim M^{2\hbar + \frac{d}{2} - \gamma}\sqrt{\Es_{2\Bbbk}}.
\end{align*}
The remaining terms are easily estimated by using the following bound coming from Lemma \ref{lemm:coeLk} and Lemma \ref{lemm:coerL},
\begin{equation}\label{est:Es2K1}
\Es_{2\Bbbk}(q) \gtrsim \int \frac{|\Ls q|^2}{y^4(1 + y^{4(\Bbbk - 2)})} \gtrsim \int \frac{|\py q|^2}{y^4(1 + y^{4(\Bbbk - 2) + 2})} + \int \frac{q^2}{y^6(1 + y^{4(\Bbbk - 2) + 2})},
\end{equation}
which implies
$$\left|\left<-\frac{\lambda_s}{\lambda} \Lambda q + \Lc(q) + \Nc(q) , \Ls^L\Phi_M\right>\right| \lesssim M^Cb_1\left(\sqrt{\Es_{2\Bbbk}} + D(t)\right). $$
Put all the above estimates into \eqref{eq:bL} and use \eqref{eq:PhiMLamQ} together with the relation \eqref{def:kdeltaplus}, we arrive at
\begin{equation}\label{eq:bLtm1}
\left|(b_L)_s + (2L - \gamma)b_1b_L\right| \lesssim \frac{\sqrt{\Es_{2\Bbbk}}}{M^{2\delta}} + b_1^{L+1+(1-\delta)(1 + \eta)} + M^Cb_1D(t).
\end{equation}

\medskip

For the modulation equations for $b_k$ with $1 \leq k \leq L-1$, we take the inner product of \eqref{eq:qys} with $\Ls^k \Phi_M$ and use the orthogonality \eqref{eq:orthqPhiM} to write for $1 \leq k \leq L-1$,
\begin{align*}
\left<\widehat{\text{Mod}}(t), \Ls^k\Phi_M\right> &= -\left<\Ls^k\tilde{\Psi}_b, \Phi_M \right> -\left<-\frac{\lambda_s}{\lambda}\Lambda q - \Lc(q) + \Nc(q), \Ls^k\Phi_M \right>.
\end{align*}
Proceed as for $b_L$, we end up with 
\begin{equation}\label{eq:bktm1}
\left|(b_k)_s + (2k - \gamma)b_1b_k - b_{k+1}\right| \lesssim b_1^{L + 1 + (1 - \delta)(1 + \eta)} + M^Cb_1\left(\sqrt{\Es_{2\Bbbk}} + D(t)\right).
\end{equation}
Similarly, we have by taking the inner product of \eqref{eq:qys} with $\Phi_M$,
\begin{equation}\label{eq:lamtm1}
\left|\frac{\lambda_s}{\lambda} + b_1\right| \lesssim b_1^{L + 1 + (1 - \delta)(1 + \eta)} + M^Cb_1\left(\sqrt{\Es_{2\Bbbk}} + D(t)\right).
\end{equation}
From \eqref{eq:bLtm1}, \eqref{eq:bktm1} and \eqref{eq:lamtm1}, we obtain the round bound
$$D(t) \lesssim M^C \sqrt{\Es_{2\Bbbk}} + b_1^{L+1 + (1 - \delta)(1 + \eta)}.$$
The conclusion then follows by substituting this bound into \eqref{eq:bLtm1}, \eqref{eq:bktm1} and \eqref{eq:lamtm1}. This ends the proof of Lemma \ref{lemm:mod1}.

\end{proof}

From the bound for $\Es_{\Bbbk}$ given in Definition \ref{def:Skset} and the modulation equation \eqref{eq:ODEbL}, we only have the pointwise bound
$$|(b_L)_s + (2L - \gamma)b_1b_L| \lesssim b_1^{L + (1 - \delta)(1 + \eta)},$$
which is not good enough to close the expected one
$$|(b_L)_s + (2L - \gamma)b_1b_L| \ll b_1^{L + 1}.$$
We claim the following improved modulation equation for $b_L$ :

\begin{lemma}[Improved modulation equation for $b_L$] \label{lemm:mod2} Under the assumption of Lemma \ref{lemm:mod1}, the following bound holds for all $s \in [s_0, s_1]$:
\begin{align}
&\left|(b_L)_s + (2L - \gamma)b_1b_L + \frac{d}{ds} \left\{ \frac{\left<\Ls^L q, \chi_{B_0}\Lambda Q\right>}{\left<\Lambda Q,\chi_{B_0} \Lambda Q \right>}\right\}\right|\nonumber\\
& \qquad \qquad \qquad \lesssim  \frac 1 {B_0^{2\delta}} \left[C(M)\sqrt{\Ec_{2\Bbbk}} + b_1^{L + 1+ (1 - \delta) - C_L\eta}\right].\label{eq:ODEbLimproved}
\end{align}
\end{lemma}
\begin{proof} We commute \eqref{eq:qys} with $\Ls^L$ and take the inner product with $\chi_{B_0}\Lambda Q$ to get
\begin{align}
&\left<\Lambda Q, \chi_{B_0}\Lambda Q\right> \left\{\frac{d}{ds}\left[\frac{ \left<\Ls^L q, \chi_{B_0}\Lambda Q\right>}{\left<\Lambda Q, \chi_{B_0}\Lambda Q\right>}\right] -  \left<\Ls^L q, \chi_{B_0}\Lambda Q\right> \frac{d}{ds}\left[\frac{1}{ \left<\Lambda Q, \chi_{B_0}\Lambda Q\right>} \right] \right\}\nonumber\\
& = \left<\Ls^L q, \Lambda Q \partial_s(\chi_{B_0})\right> - \left<\Ls^{L+1} q, \chi_{B_0}\Lambda Q\right> + \frac{\lambda_s}{\lambda}\left<\Ls^L\Lambda q,  \chi_{B_0}\Lambda Q\right> \nonumber\\
& \quad - \left<\Ls^L \tilde{\Psi}_b,\chi_{B_0}\Lambda Q\right> - \left< \Ls^L\widehat{\textup{Mod}}(t), \chi_{B_0}\Lambda Q\right>  + \left<\Ls^L\big(\Lc(q) - \Nc(q)\big), \chi_{B_0}\Lambda Q\right>.\label{eq:LqLamQt1}
\end{align}

We recall from \eqref{eq:asymLamQ} that 
\begin{equation}\label{est:LQ2chiB0}
B_0^{d - 2\gamma} \lesssim |\left<\Lambda Q, \chi_{B_0}\Lambda Q\right> | \lesssim B_0^{d - 2\gamma}.
\end{equation}
Let us estimate the second term in the left hand side of \eqref{eq:LqLamQt1}. We use  Cauchy-Schwartz and Lemma \ref{lemm:coeLk} to estimate 
\begin{align}
\left|\left<\Ls^{L} q, \chi_{B_0}\Lambda Q\right>\right|& \lesssim B_0^{2\hbar + 2}\|\chi_{B_0}\Lambda Q\|_{L^2}\left(\int\frac{|\Ls^L q|^2}{1 + y^{4\hbar + 4}} \right)^\frac 12\nonumber\\
&\quad \lesssim B_0^{\frac d2 - \gamma + 2\hbar + 2}\sqrt{\Es_{2\Bbbk}}.\label{est:LsLqchiB0LQ}
\end{align}
We write 
\begin{align*}
\left|\left<\Ls^L q, \chi_{B_0}\Lambda Q\right> \frac{d}{ds}\left[\frac{1}{ \left<\Lambda Q, \chi_{B_0}\Lambda Q\right>} \right]\right|& \lesssim \frac{|\left<\Ls^L q, \chi_{B_0}\Lambda Q\right>|}{\left<\Lambda Q, \chi_{B_0}\Lambda Q\right>^2} \left|\frac{(b_1)_s}{b_1}\right|\int_{B_0 \leq y \leq 2B_0}|\Lambda Q|^2\\
&\quad \lesssim b_1 \frac{B_0^{\frac d2 - \gamma + 2\hbar + 2} \sqrt{\Es_{2\Bbbk}}}{B_0^{2d - 4\gamma}}B_0^{d - 2\gamma} \lesssim \frac{\sqrt{\Es_{2\Bbbk}}}{B_0^{2\delta}},
\end{align*}
where we used the relation \eqref{def:kdeltaplus}.

For the first three terms in the right hand side of \eqref{eq:LqLamQt1}, we use the Cauchy-Schwartz, Lemma \ref{lemm:coeLk} and the fact that $\Ls(\Lambda Q) = 0$ to find that
\begin{align*}
\left|\left<\Ls^{L} q,\Lambda Q \partial_s(\chi_{B_0})\right>\right|& \lesssim \left|\frac{(b_1)_s}{b_1} \right| \left(\int_{B_0 \leq y \leq 2B_0} (1 + y^{4\hbar + 4}) |\Lambda Q|^2 \right)^\frac 12 \left(\int\frac{|\Ls^L q|^2}{1 + y^{4\hbar + 4}} \right)^\frac 12\\
&\quad \lesssim b_1 B_0^{\frac d2 - \gamma + 2\hbar + 2}\sqrt{\Es_{2\Bbbk}} \lesssim B_0^{\frac d2 - \gamma + 2\hbar}\sqrt{\Es_{2\Bbbk}}
\end{align*}
and
\begin{align*}
\left|\left<\Ls^{L+1} q, \chi_{B_0}\Lambda Q\right>\right|& \lesssim \left(\int (1 + y^{4\hbar})|\chi_{B_0}\Lambda Q|^2 \right)^\frac 12 \left(\int\frac{|\Ls^{L+1} q|^2}{1 + y^{4\hbar}} \right)^\frac 12\\
&\quad \lesssim B_0^{\frac d2 - \gamma + 2\hbar}\sqrt{\Es_{2\Bbbk}},
\end{align*}
and 
\begin{align*}
\left|\frac{\lambda_s}{\lambda}\left<\Ls^L\Lambda q,  \chi_{B_0}\Lambda Q\right> \right| & \lesssim b_1 \left(\int (1 + y^{4(L + \hbar)+2})|\Ls^L(\chi_{B_0} \Lambda Q)|^2 \right)^\frac 12 \left(\int \frac{|\py q|^2}{1 + y^{4(L + \hbar) + 2}}\right)^\frac 12 \\
&\lesssim B_0^{\frac d2 - \gamma + 2\hbar}\sqrt{\Es_{2\Bbbk}}.
\end{align*}
The error term is estimated by using \eqref{eq:estPsiblocalB0},
\begin{align*}
\left|\left<\Ls^L \tilde{\Psi}_b,\chi_{B_0}\Lambda Q\right>\right| &\lesssim \left(\int (1 + y^{4(L + \hbar +1)})|\Ls^L(\chi_{B_0} \Lambda Q)|^2 \right)^\frac 12\left(\int\frac{|\tilde{\Psi}_b|^2}{1 + y^{4(L + \hbar + 1)}} \right)^\frac 12\\
&\quad \lesssim B_0^{\frac d2 - \gamma + 2\hbar + 2} b_1^{L + 2 + (1 - \delta) - C_L\eta}.
\end{align*}
The last term in the right hand side of \eqref{eq:LqLamQt1} is estimated in the same way, 
\begin{align*}
\left|\left<\Ls^L\big(\Lc(q) - \Nc(q)\big), \chi_{B_0}\Lambda Q\right> \right|& \lesssim \int |\Lc(q) \Ls^L(\chi_{B_0} \Lambda Q)| + \int |\Nc(q) \Ls^L(\chi_{B_0} \Lambda Q)|\\
& \lesssim \left(\int \frac{|\Lc(q)|^2}{1 + y^{4\Bbbk - 4}}\right)^\frac{1}{2}\left(\int (1 + y^{4\Bbbk - 4}) |\Ls^L(\chi_{B_0}\Lambda Q)|^2) \right)^\frac{1}{2}\\
&\quad  + \left(\int \frac{|\Nc(q)|^2}{1 + y^{4\Bbbk}}\right)^\frac{1}{2}\left(\int (1 + y^{4\Bbbk}) |\Ls^L(\chi_{B_0}\Lambda Q)|^2) \right)^\frac{1}{2}\\
& \qquad \lesssim B_0^{\frac d2 -1 - \gamma + 2\hbar} \sqrt{\Es_{2\Bbbk}} + b_1 B_0^2 B_0^{\frac d2 -1 - \gamma + 2\hbar}\sqrt{\Es_{\Bbbk}}\\
& \qquad \quad\lesssim B_0^{\frac d2 - \gamma + 2\hbar} \sqrt{\Es_{2\Bbbk}}.
\end{align*}
For the remaining term, we recall that $\Ls(\Lambda Q) = 0$, $\Ls^L T_k = 0$ for $1 \leq k \leq L-1$, and $\Ls^L T_L = (-1)^L \Lambda Q$, from which
$$\Ls^L(T_k\chi_{B_1})= -\Ls^L(T_k(1-\chi_{B_1})), \quad 1 \leq k \leq L-1.$$
From \eqref{def:Modhat}, \eqref{eq:Modt} and the fact that $\chi_{B_0}(1 - \chi_{B_1}) = 0$, we write
\begin{align*}
&\left|\left< \Ls^L\widehat{\textup{Mod}}(t), \chi_{B_0}\Lambda Q\right> - (-1)^L\left<\Lambda Q, \chi_{B_0}\Lambda Q\right> \left[(b_L)_s + (2L - \gamma)b_1b_L \right] \right|\\
&\quad \lesssim \sum_{k = 1}^{L} \left|(b_k)_s + (2k - \gamma)b_1b_L - b_{k+1} \right| \left|\left<\sum_{j = k+1}^{L+2} \frac{\partial \tilde{S}_j}{\partial b_k}, \Ls^L (\chi_{B_0}\Lambda Q)\right>\right|\\
&\qquad + \left|\frac{\lambda_s}{\lambda} + b_1\right| \left|\left<\Lambda \tilde{\Theta}_b,\Ls^L( \chi_{B_0}\Lambda Q)\right>\right|.
\end{align*}
Recall that $T_k$ is admissible of degree $(k,k)$ and $S_k$ is homogeneous of degree $(k, k-1, k)$, we derive the round bounds for $y \sim B_0$:
$$|\Lambda \Theta_b| \lesssim b_1 y^{2 - \gamma}, \quad \sum_{j = k+1}^{L+2}\left|\frac{\partial S_j}{\partial b_k} \right| \leq \sum_{j = k+1}^{L+2}b_1^{j - k}y^{2(j - 1) - \gamma} \lesssim b_1y^{2k - \gamma}.$$
Thus, from Lemma \ref{lemm:mod1}, we derive the bound
\begin{align*}
&\left|\frac{\lambda_s}{\lambda} + b_1\right| \left|\left<\Lambda \tilde{\Theta}_b,\Ls^L( \chi_{B_0}\Lambda Q)\right>\right|\\
& \quad  + \sum_{k = 1}^{L} \left|(b_k)_s + (2k - \gamma)b_1b_L - b_{k+1} \right| \left|\left<\sum_{j = k+1}^{L+2} \frac{\partial \tilde{S}_j}{\partial b_k}, \Ls^L (\chi_{B_0}\Lambda Q)\right>\right|\\
& \qquad \lesssim \left( C(M)\sqrt{\Es_{2\Bbbk}} +  b_1^{L + 1 + (1 - \delta)(1 + \eta)}\right)b_1\int_{B_0 \leq y \leq 2B_0} \frac{y^{2L - \gamma }y^{d-1}}{y^{2L + \gamma}}dy\\
& \quad \qquad \lesssim  \left( C(M)\sqrt{\Es_{2\Bbbk}} +  b_1^{L + 1 + (1 - \delta)(1 + \eta)}\right)b_1B_0^{d - 2\gamma}.
\end{align*}

The equation \eqref{eq:ODEbLimproved} follows by gathering all the above estimates into \eqref{eq:LqLamQt1}, dividing both sides of \eqref{eq:LqLamQt1} by $(-1)^L\left<\Lambda Q, \chi_{B_0}\Lambda Q\right>$ and using the relation \eqref{def:kdeltaplus}. This finishes the proof of Lemma \ref{lemm:mod2}.
\end{proof}

\subsection{Monotonicity.}

We derive in this subsection the main monotonocity formula for $\Es_{2k}$ for $\hbar +1 \leq k \leq \Bbbk$. We claim the following which is the heart of this paper:

\begin{proposition}[Lyapounov monotonicity for the high Sobolev norm] \label{prop:E2k} We have 
\begin{align}
&\frac{d}{dt}\left\{\frac{\Es_{2\Bbbk}}{\lambda^{4\Bbbk - d}}\left[1 + \Oc\left(b_1^{\eta(1-\delta)}\right) \right]\right\}\nonumber\\
&\qquad \qquad \leq \frac{b_1}{\lambda^{4\Bbbk - d + 2}}\left[\frac{\Es_{2\Bbbk}}{M^{2\delta}} + b_1^{L+(1 - \delta)(1 + \eta)}\sqrt{\Es_{2\Bbbk}} + b_1^{2L + 2(1 - \delta)(1 + \eta)}\right],\label{eq:Es2sLya}
\end{align}
and for $\hbar + 2  \leq m \leq \Bbbk -1$,
\begin{align}
&\frac{d}{dt}\left\{\frac{\Es_{2m}}{\lambda^{4m - d}}\left[1 + \Oc(b_1) \right]\right\} \leq \frac{b_1}{\lambda^{4m - d + 2}}\left[b_1^{m - \hbar -1 +(1 - \delta) - C\eta}\sqrt{\Es_{2m}} + b_1^{2(m-\hbar -1) + 2(1 - \delta) - C\eta}\right],\label{eq:Es2sLyam}
\end{align}
\end{proposition}

\begin{proof} The proof uses some ideas developed in \cite{RSapde2014} and \cite{MRRcjm15}. Because the proof of \eqref{eq:Es2sLyam} follows exactly the same lines as for \eqref{eq:Es2sLya}, we only deal with the proof of \eqref{eq:Es2sLya}. Let us start the proof of \eqref{eq:Es2sLya}.

\noindent \textbf{Step 1: Suitable derivatives and energy identity.} For $k \in \mathbb{N}$, we define the suitable derivatives of $q$ and $v$ as follows:
\begin{equation}\label{eq:notq2k1}
q_{2k} = \Ls^kq, \quad q_{2k + 1} = \As \Ls^k q, \quad v_{2k} = \Ls^k_\lambda v, \quad v_{2k + 1} = \As_\lambda \Ls^k_\lambda v,
\end{equation}
where $q = q(y,s)$ and $v = v(r,t)$ satisfy \eqref{eq:qys} and \eqref{eq:vrt} respectively, the linearized operator $\Ls$ and $\Ls_\lambda$ are defined by \eqref{def:Lc} and \eqref{def:Llambda}, $\As$ and $\As^*$ are the first order operators defined by \eqref{def:As} and \eqref{def:Astar}, and
$$\As_\lambda f = -\partial_r f + \frac{V_\lambda}{r} f, \quad \As^*_\lambda f = \frac{1}{r^{d-1}}\partial_r(r^{d-1}f) + \frac{V_\lambda}{r} f,$$
with $V = \Lambda \log \Lambda Q$ admitting the asymptotic behaviors as in \eqref{eq:asympV}.

With the notation \eqref{eq:notq2k1}, we note that 
$$q_{2k + 1} = \As q_{2k}, \quad q_{2k+2} = \As^*q_{2k + 1}, \quad v_{2k + 1} = \As_\lambda v_{2k}, \quad v_{2k+2} = \As^*_\lambda v_{2k + 1}.$$
Recall from Lemma \ref{lemm:factorL}, we have the following factorization:
$$\Ls  = \As^* \As, \quad \tilde\Ls = \As \As^*,  \quad \Ls_\lambda = \As^*_\lambda \As_\lambda, \quad \tilde{\Ls}_\lambda = \As_\lambda \As^*_\lambda,$$
where 
\begin{equation}\label{def:Lstil}
\tilde{\Ls} = -\partial_{yy} - \frac{d-1}{y}\py + \frac{\tilde{Z}}{y^2},
\end{equation}
and 
\begin{equation}\label{def:LstilLam}
\tilde{\Ls}_\lambda = -\partial_{rr} - \frac{d-1}{r}\partial_r + \frac{\tilde{Z}_\lambda}{r^2},
\end{equation}
with $\tilde{Z}$ expressed in terms of $V$ as in \eqref{def:ZtilbyV}.\\

We commute equation \eqref{eq:vrt} with $\Ls_\lambda^{\Bbbk - 1}$ and use the notation \eqref{eq:notq2k1} to derive 
\begin{equation}\label{eq:v2k2}
\pt v_{2\Bbbk - 2} + \Ls_\lambda v_{2\Bbbk -2} = [\pt, \Ls^{\Bbbk - 1}_\lambda]v + \Ls_\lambda^{\Bbbk - 1}\left( \frac{1}{\lambda^2} \Fc_\lambda\right).
\end{equation}
Now commuting this equation with $\As_\lambda$ yields
\begin{equation}\label{eq:v2k1}
\pt v_{2\Bbbk - 1} + \tilde \Ls_\lambda v_{2\Bbbk - 1} = \frac{\pt V_\lambda}{r} v_{2\Bbbk - 2} +\As_\lambda [\pt, \Ls^{\Bbbk - 1}_\lambda]v + \As_\lambda \Ls_\lambda^{\Bbbk - 1}\left( \frac{1}{\lambda^2} \Fc_\lambda\right), 
\end{equation}
Since $\Ls_\lambda = \frac{1}{\lambda^2}\Ls$, we then have 
$$\Ls^k_\lambda v = \frac{1}{\lambda^{2k}}\Ls^k q, \quad \text{hence,}\quad \int|\Ls^k_\lambda v|^2 = \frac{1}{\lambda^{4k - d}}\int |\Ls^k q|^2.$$
Using the definition \eqref{def:LstilLam} of $\tilde{\Ls}_\lambda$ and an integration by parts, we write
\begin{align*}
\frac{1}{2}\frac{d}{dt} \left(\frac{1}{\lambda^{4\Bbbk - d}} \Es_{2\Bbbk}\right) &= \frac{1}{2}\frac{d}{dt} \int |\Ls_\lambda^{\Bbbk} v|^2 = \frac{1}{2}\frac{d}{dt}\int \tilde{\Ls}_\lambda v_{2\Bbbk - 1} v_{2\Bbbk - 1}\\
&= \int \tilde{\Ls}_\lambda v_{2\Bbbk - 1} \pt v_{2\Bbbk - 1} + \frac{1}{2}\int \frac{\pt (\tilde Z_\lambda)}{r^2}v^2_{2\Bbbk - 1}\\
&= \int \tilde{\Ls}_\lambda v_{2\Bbbk - 1} \pt v_{2\Bbbk - 1} + b_1\int \frac{(\Lambda \tilde{Z})_\lambda}{2\lambda^2r^2}v^2_{2\Bbbk - 1} - \left(\frac{\lambda_s}{\lambda} + b_1\right) \int \frac{(\Lambda \tilde{Z})_\lambda}{2\lambda^2r^2}v^2_{2\Bbbk - 1}.
\end{align*}
Using the definition \eqref{def:Astar} of $\As^*$ and an integration by parts together with the definition \eqref{def:ZtilbyV} of $\tilde{Z}$, we write
\begin{align*}
\int \frac{b_1(\Lambda V)_\lambda}{\lambda^2 r} v_{2\Bbbk - 1} \As^*_\lambda v_{2\Bbbk - 1} &= \frac{b_1}{\lambda^{4\Bbbk - d + 2}}\int \frac{\Lambda V}{y}q_{2\Bbbk - 1} \As^*q_{2\Bbbk - 1}\\
&= \frac{b_1}{\lambda^{4\Bbbk - d + 2}}\int \frac{\Lambda V(2V + d) - \Lambda^2 V}{2y^2}q^2_{2\Bbbk - 1}\\
&= \frac{b_1}{\lambda^{4\Bbbk - d + 2}} \int \frac{\Lambda \tilde{Z}}{2y^2}q^2_{2\Bbbk - 1} = \int \frac{b_1(\Lambda \tilde{Z})_\lambda}{2\lambda^2r^2}v^2_{2\Bbbk - 1}.
\end{align*}
From \eqref{eq:v2k2}, we write
\begin{align*}
\frac{d}{dt}\int \frac{b_1(\Lambda V)_\lambda}{\lambda^2r}v_{2\Bbbk - 1} v_{2\Bbbk - 2} &= \int \frac{d}{dt}\left(\frac{b_1(\Lambda V)_\lambda}{\lambda^2r}\right)v_{2\Bbbk - 1} v_{2\Bbbk - 2}  + \int \frac{b_1(\Lambda V)_\lambda}{\lambda^2r}  v_{2\Bbbk - 2}\pt v_{2\Bbbk - 1}\\
& + \int \frac{b_1(\Lambda V)_\lambda}{\lambda^2r} v_{2\Bbbk - 1} \left[-\As^*_\lambda v_{2\Bbbk - 1} + [\pt, \Ls^{\Bbbk - 1}_\lambda]v + \Ls_\lambda^{\Bbbk - 1}\left( \frac{1}{\lambda^2} \Fc_\lambda\right)\right].
\end{align*}
Gathering all the above identities and using \eqref{eq:v2k1} yields the energy identity
\begin{align}
&\frac{1}{2}\frac{d}{dt}\left\{ \left(\frac{1}{\lambda^{4\Bbbk - d}} \Es_{2\Bbbk}\right) + 2 \int \frac{b_1(\Lambda V)_\lambda}{\lambda^2r}v_{2\Bbbk - 1} v_{2\Bbbk - 2} \right\} \label{eq:EnerID}\\
&= - \int |\tilde\Ls_\lambda v_{2\Bbbk - 1}|^2 - \left(\frac{\lambda_s}{\lambda} + b_1\right) \int \frac{(\Lambda \tilde{Z})_\lambda}{2\lambda^2r^2}v^2_{2\Bbbk - 1} - \int \frac{b_1(\Lambda V)_\lambda}{\lambda^2r} v_{2\Bbbk - 2} \tilde{\Ls}_\lambda v_{2\Bbbk - 1}\nonumber\\
&\quad + \int \frac{d}{dt}\left(\frac{b_1(\Lambda V)_\lambda}{\lambda^2r}\right)v_{2\Bbbk - 1} v_{2\Bbbk - 2} + \int \frac{b_1(\Lambda V)_\lambda}{\lambda^2r} v_{2\Bbbk - 1} \left[[\pt, \Ls^{\Bbbk - 1}_\lambda]v + \Ls_\lambda^{\Bbbk - 1}\left( \frac{1}{\lambda^2} \Fc_\lambda\right)\right]\nonumber\\
& \quad + \int \left(\tilde{\Ls}_\lambda v_{2\Bbbk -1} +  \frac{b_1(\Lambda V)_\lambda}{\lambda^2r}  v_{2\Bbbk - 2}\right)\left[\frac{\pt V_\lambda}{r} v_{2\Bbbk - 2} +\As_\lambda [\pt, \Ls^{\Bbbk - 1}_\lambda]v + \As_\lambda \Ls_\lambda^{\Bbbk - 1}\left( \frac{1}{\lambda^2} \Fc_\lambda\right) \right].\nonumber
\end{align}
We now estimate all terms in \eqref{eq:EnerID}. The proof uses the coercivity estimate given in Lemma \ref{lemm:coeLk}. In particular, we shall apply Lemma \ref{lemm:coeLk} with $k = \Bbbk - 1$ to have the estimate
\begin{equation}\label{eq:Enercontrol}
\Es_{2\Bbbk} \gtrsim \int \frac{|q_{2\Bbbk - 1}|^2}{y^2} + \sum_{m = 0}^{\Bbbk - 1} \int \frac{|q_{2m}|^2}{y^4(1 + y^{4(\Bbbk - 1 - m)})} + \sum_{m = 0}^{\Bbbk - 2}\int \frac{|q_{2m +1}|^2}{y^6(1 + y^{4(\Bbbk - 2 - m)})}.
\end{equation}
 
\noindent \textbf{Step 2: Control of the lower order quadratic terms.} Let us start with the second term in the left hand side of \eqref{eq:EnerID}. From \eqref{eq:asympV} and \eqref{def:ZtilbyV}, we have the round bound
\begin{equation}\label{eq:estLamZV}
|\Lambda \tilde{Z}(y)| + |\Lambda V(y)| \lesssim \frac{y^2}{1 + y^4}, \quad \forall y \in [0, +\infty).
\end{equation}
Making a change of variables and using the Cauchy-Schwartz inequality together with \eqref{eq:Enercontrol}, we estimate
\begin{align*}
\left|\int \frac{b_1(\Lambda V)_\lambda}{\lambda^2r}v_{2\Bbbk - 1} v_{2\Bbbk - 2}\right| &= \left|\frac{b_1}{\lambda^{4\Bbbk - d}}\int \frac{\Lambda V}{y}q_{2\Bbbk - 1}q_{2\Bbbk - 2}\right|\\
&\lesssim \frac{b_1}{\lambda^{4\Bbbk - d}} \left(\int \frac{|q_{2\Bbbk - 1}|^2}{y^2}\right)^\frac 12 \left(\int \frac{|q_{2\Bbbk - 2}|^2}{1 + y^4}\right)^\frac 12 \lesssim \frac{b_1}{\lambda^{4\Bbbk - d}} \Es_{2\Bbbk}.
\end{align*}
Using \eqref{eq:estLamZV}, \eqref{eq:ODEbkl} and \eqref{eq:Enercontrol}, we estimate
\begin{align*}
\left|\left(\frac{\lambda_s}{\lambda} + b_1\right)\int \frac{(\Lambda \tilde{Z})_\lambda}{\lambda^2 r}v^2_{2\Bbbk - 1} \right| &= \left|\left(\frac{\lambda_s}{\lambda} + b_1\right) \frac{1}{\lambda^{4\Bbbk - d +2}}\int \frac{\Lambda \tilde{Z}}{y^2}q^2_{2\Bbbk - 1} \right|\\
& \lesssim \frac{b_1^{L+1 + (1 - \delta)(1+\eta)}}{\lambda^{4\Bbbk - d +2}}\int \frac{q^2_{2\Bbbk - 1}}{y^2} \lesssim \frac{b_1^{2}}{\lambda^{4\Bbbk - d +2}}\Es_{2\Bbbk}.
\end{align*}
For the third term in the right hand side of \eqref{eq:EnerID}, we write
\begin{align*}
\left|\int \frac{b_1(\Lambda V)_\lambda}{\lambda^2r} v_{2\Bbbk - 2} \tilde{\Ls}_\lambda v_{2\Bbbk - 1}\right| &\leq \frac{1}{4}\int |\tilde{\Ls}_\lambda v_{2\Bbbk - 1}|^2 + 4\int \left(\frac{b_1(\Lambda V)_\lambda}{\lambda^2r}\right)^2 v^2_{2\Bbbk - 2}\\
&= \frac{1}{4}\int |\tilde{\Ls}_\lambda v_{2\Bbbk - 1}|^2 + \frac{4b_1^2}{\lambda^{4 \Bbbk - d + 2}}\int \frac{|\Lambda V|^2}{y^2}q^2_{2\Bbbk - 2}\\
& \leq \frac{1}{4}\int |\tilde{\Ls}_\lambda v_{2\Bbbk - 1}|^2 + \frac{Cb_1^2}{\lambda^{4 \Bbbk - d + 2}} \Es_{2\Bbbk}.
\end{align*}
A direct computation yields the round bound
$$\left|\frac{d}{dt}\left(\frac{b_1(\Lambda V)_\lambda}{\lambda^2}\right)\right| \lesssim \frac{b_1^2}{\lambda^4}(|\Lambda V| + |\Lambda^2V|).$$
Thus, we use \eqref{eq:estLamZV}, the Cauchy-Schwartz inequality and \eqref{eq:Enercontrol} to estimate
\begin{align*}
\left|\int \frac{d}{dt}\left(\frac{b_1(\Lambda V)_\lambda}{\lambda^2r}\right)v_{2\Bbbk - 1} v_{2\Bbbk - 2} \right| &\lesssim \frac{b_1^2}{\lambda^{4\Bbbk - d + 2}}\int \frac{|\Lambda V| + |\Lambda^2V|}{y}|q_{2\Bbbk - 1}q_{2\Bbbk - 2}|\\
& \lesssim \frac{b_1^2}{\lambda^{4\Bbbk - d + 2}}\left(\int \frac{q^2_{2\Bbbk - 1}}{y^2}\right)^\frac 12 \left(\int \frac{q^2_{2\Bbbk - 2}}{1 + y^4}\right)^\frac 12\\
& \lesssim \frac{b_1^2}{\lambda^{4\Bbbk - d + 2}}\Es_{2\Bbbk}.
\end{align*}
Similarly, we have 
\begin{align*}
&\left|\int \left(\tilde{\Ls}_\lambda v_{2\Bbbk -1} +  \frac{b_1(\Lambda V)_\lambda}{\lambda^2r}  v_{2\Bbbk - 2}\right)\frac{\pt V_\lambda}{r} v_{2\Bbbk - 2}\right|\\
&\qquad \leq \frac{1}{4}\int |\tilde{\Ls}_\lambda v_{2\Bbbk -1}|^2 + \frac{Cb_1^2}{\lambda^{4\Bbbk - d + 2}} \int \frac{|\Lambda V|^2}{y^2}q^2_{2\Bbbk - 2}\\
&\qquad \leq \frac{1}{4}\int |\tilde{\Ls}_\lambda v_{2\Bbbk -1}|^2 + \frac{Cb_1^2}{\lambda^{4\Bbbk - d + 2}}\Es_{2\Bbbk},
\end{align*}
and 
\begin{align*}
& \left|\int \frac{b_1(\Lambda V)_\lambda}{\lambda^2r} v_{2\Bbbk - 1}[\pt, \Ls^{\Bbbk - 1}_\lambda]v \right| + \left|\int \left(\tilde{\Ls}_\lambda v_{2\Bbbk -1} +  \frac{b_1(\Lambda V)_\lambda}{\lambda^2r}  v_{2\Bbbk - 2}\right)\As_\lambda [\pt, \Ls^{\Bbbk - 1}_\lambda]v \right|\\
&\leq \frac{1}{4}\int |\tilde{\Ls}_\lambda v_{2\Bbbk -1}|^2 + C\left(\frac{b_1^2}{\lambda^{4\Bbbk - d +2}}\Es_{2\Bbbk} + \int \frac{\big|[\pt, \Ls_\lambda^{\Bbbk - 1}]v\big|^2}{\lambda^2(1 + y^2)} + \int \big|\As_\lambda[\pt, \Ls_\lambda^{\Bbbk - 1}]v\big|^2\right).
\end{align*}
We claim the bound 
\begin{equation}\label{est:comtor}
\int \frac{\big|[\pt, \Ls_\lambda^{\Bbbk - 1}]v\big|^2}{\lambda^2(1 + y^2)} + \int \big|\As_\lambda[\pt, \Ls_\lambda^{\Bbbk - 1}]v\big|^2 \lesssim \frac{b_1^2}{\lambda^{4\Bbbk - d + 2}}\Es_{2\Bbbk},
\end{equation}
whose proof is left to Appendix \ref{ap:EstComm}. 

The collection of all the above estimates to \eqref{eq:EnerID} yields 
\begin{align}
\frac{1}{2}\frac{d}{dt}\left\{ \frac{\Es_{2\Bbbk}}{\lambda^{4\Bbbk - d}} \Big[1 + \Oc(b_1)\Big] \right\} &\leq -\frac{1}{4} \int |\tilde\Ls_\lambda v_{2\Bbbk - 1}|^2 + \frac{Cb_1^2}{\lambda^{4\Bbbk - d + 2}}\Es_{2\Bbbk}\nonumber\\
& \quad + \int \frac{b_1(\Lambda V)_\lambda}{\lambda^2r} v_{2\Bbbk - 1} \Ls_\lambda^{\Bbbk - 1}\left( \frac{1}{\lambda^2} \Fc_\lambda\right)\nonumber\\
& \quad +\int \frac{b_1(\Lambda V)_\lambda}{\lambda^2r}  v_{2\Bbbk - 2} \As_\lambda\Ls_\lambda^{\Bbbk - 1}\left( \frac{1}{\lambda^2} \Fc_\lambda\right)\nonumber\\
&\quad + \int \tilde{\Ls}_\lambda v_{2\Bbbk -1}\As_\lambda \Ls_\lambda^{\Bbbk - 1}\left( \frac{1}{\lambda^2} \Fc_\lambda\right).\label{eq:EnerID1}
\end{align}

\noindent \textbf{Step 3: Further use of dissipation.} We aim at estimating all terms in the right hand side of \eqref{eq:EnerID1}. From \eqref{eq:estLamZV}, \eqref{eq:Enercontrol} and the Cauchy-Schwartz inequality, we write 
\begin{align*}
\left|\int \frac{b_1(\Lambda V)_\lambda}{\lambda^2r} v_{2\Bbbk - 1} \Ls_\lambda^{\Bbbk - 1}\left( \frac{1}{\lambda^2} \Fc_\lambda\right) \right|&= \left|\frac{b_1}{\lambda^{4\Bbbk - d + 2}}\int \frac{\Lambda V}{y}q_{2\Bbbk - 1}\Ls^{\Bbbk - 1}\Fc \right|\\
& \quad \lesssim \frac{b_1}{\lambda^{4\Bbbk - d + 2}} \left(\int \frac{q_{2\Bbbk - 1}^2}{y^2} \right)^\frac{1}{2}\left(\int \frac{|\Ls^{\Bbbk - 1}\Fc|^2}{1 + y^4} \right)^\frac{1}{2}\\
& \quad \lesssim \frac{b_1}{\lambda^{4\Bbbk - d + 2}} \sqrt{\Es_{2\Bbbk}}\left(\int \frac{|\Ls^{\Bbbk - 1}\Fc|^2}{1 + y^4} \right)^\frac{1}{2}.
\end{align*}
Similarly, we have 
\begin{align*}
\left|\int \frac{b_1(\Lambda V)_\lambda}{\lambda^2r} v_{2\Bbbk - 2} \As_\lambda\Ls_\lambda^{\Bbbk - 1}\left( \frac{1}{\lambda^2} \Fc_\lambda\right) \right|&= \left|\frac{b_1}{\lambda^{4\Bbbk - d + 2}}\int \frac{\Lambda V}{y}q_{2\Bbbk - 2}\As\Ls^{\Bbbk - 1}\Fc \right|\\
& \quad \lesssim \frac{b_1}{\lambda^{4\Bbbk - d + 2}} \left(\int \frac{q_{2\Bbbk - 2}^2}{1+y^4} \right)^\frac{1}{2}\left(\int \frac{|\As\Ls^{\Bbbk - 1}\Fc|^2}{1 + y^2} \right)^\frac{1}{2}\\
& \quad \lesssim \frac{b_1}{\lambda^{4\Bbbk - d + 2}} \sqrt{\Es_{2\Bbbk}}\left(\int \frac{|\As\Ls^{\Bbbk - 1}\Fc|^2}{1 + y^2} \right)^\frac{1}{2}.
\end{align*}
For the last term in \eqref{eq:EnerID1}, let us introduce the function
\begin{equation}\label{def:xiL}
\xi_L = \frac{\left<\Ls^Lq, \chi_{B_0}\Lambda Q\right>}{\left<\Lambda Q, \chi_{B_0}\Lambda Q\right>}\tilde{T}_L,
\end{equation}
and the decomposition 
\begin{equation}\label{eq:decomF}
\Fc = \partial_s \xi_L + \Fc_0 + \Fc_1, \quad \Fc_0 = - \tilde{\Psi}_b - \widehat{\text{Mod}} - \partial_s \xi_L, \quad \Fc_1 = \Hc(q) - \Nc(q),
\end{equation}
where $\tilde \Psi_b$ is referred to \eqref{def:Psibtilde}, $\widehat{\text{Mod}}$, $\Hc(q)$ and $\Nc(q)$ are defined as in \eqref{def:Modhat} \eqref{def:Lq} and \eqref{def:Nq} respectively.
Actually, we introduced the decomposition \eqref{eq:decomF} and $\xi_L$ to take advantage of the improved bound obtained in Lemma  \ref{lemm:mod2}.
We now write
\begin{align*}
&\int \tilde{\Ls}_\lambda v_{2\Bbbk -1}\As_\lambda \Ls_\lambda^{\Bbbk - 1}\left( \frac{1}{\lambda^2} \Fc_\lambda\right)\\
& \qquad = \frac{1}{\lambda^{4\Bbbk - d +2}}\left(\int \As^* q_{2\Bbbk - 1}\Ls^\Bbbk(\partial_s \xi_L) + \int \As^* q_{2\Bbbk - 1}\Ls^\Bbbk \Fc_0 + \int \tilde{\Ls}q_{2\Bbbk - 1} \As \Ls^{\Bbbk - 1}\Fc_1\right)\\
& \qquad \leq \frac{1}{\lambda^{4\Bbbk - d +2}}\int \Ls^\Bbbk q \Ls^\Bbbk(\partial_s \xi_L) + \frac{C}{\lambda^{4\Bbbk - d +2}} \left(\int |\Ls^\Bbbk q|^2\right)^\frac{1}{2}\left(\int |\Ls^\Bbbk \Fc_0| \right)^\frac 12\\
&\qquad \qquad \qquad \qquad \qquad \qquad \qquad  \qquad + \frac{1}{8}\int |\tilde{\Ls}_\lambda v_{2\Bbbk - 1}|^2 + \frac{C}{\lambda^{4\Bbbk - d +2}}\int |\As \Ls^{\Bbbk - 1}\Fc_1|^2\\
&\qquad = \frac{1}{\lambda^{4\Bbbk - d +2}}\int \Ls^\Bbbk q \Ls^\Bbbk(\partial_s \xi_L) + \frac{1}{8}\int |\tilde{\Ls}_\lambda v_{2\Bbbk - 1}|^2 \\
&\qquad \qquad+ \frac{C}{\lambda^{4\Bbbk - d +2}}\left( \sqrt{\Es_{2\Bbbk}}\left\|\Ls^\Bbbk \Fc_0 \right\|_{L^2} + \left\|\As \Ls^{\Bbbk -1}\Fc_1 \right\|_{L^2}^2 \right).
\end{align*}
Injecting all these bounds into \eqref{eq:EnerID1} yields
\begin{align}
&\frac{1}{2}\frac{d}{dt}\left\{ \frac{\Es_{2\Bbbk}}{\lambda^{4\Bbbk - d}}\Big[1  + \Oc(b_1)\Big] \right\}\nonumber\\
&\qquad \leq -\frac{1}{8} \int |\tilde\Ls_\lambda v_{2\Bbbk - 1}|^2 + \frac{Cb_1^2}{\lambda^{4\Bbbk - d + 2}}\Es_{2\Bbbk} + \frac{1}{\lambda^{4\Bbbk - d +2}}\int \Ls^\Bbbk q \Ls^\Bbbk(\partial_s \xi_L)\nonumber\\
& \qquad \quad + \frac{b_1}{\lambda^{4\Bbbk - d + 2}} \sqrt{\Es_{2\Bbbk}}\left[\left(\int \frac{|\As\Ls^{\Bbbk - 1}\Fc|^2}{1 + y^2} \right)^\frac{1}{2} + \left(\int \frac{|\Ls^{\Bbbk - 1}\Fc|^2}{1 + y^4} \right)^\frac{1}{2} \right]\nonumber\\
&\qquad  \quad + \frac{C}{\lambda^{4\Bbbk - d +2}}\left( \sqrt{\Es_{2\Bbbk}}\left\|\Ls^\Bbbk \Fc_0 \right\|_{L^2} + \left\|\As \Ls^{\Bbbk -1}\Fc_1 \right\|_{L^2}^2 \right). \label{eq:EnerID2}
\end{align}

\noindent \textbf{Step 4: Estimates for $\tilde \Psi_b$ term.} Recall from \eqref{eq:estPsibLarge2} that we already have the following estimate for $\tilde{\Psi}_b$:
\begin{equation}\label{eq:PsibtilE2k}
\left\|\Ls^\Bbbk \tilde\Psi_b \right\|_{L^2} + \left(\int \frac{|\As\Ls^{\Bbbk - 1}\tilde\Psi_b|^2}{1 + y^2} \right)^\frac{1}{2} + \left(\int \frac{|\Ls^{\Bbbk - 1}\tilde\Psi_b|^2}{1 + y^4} \right)^\frac{1}{2} \lesssim b_1^{L + 1 + (1 -\delta)(1 + \eta)}.
\end{equation}

\noindent \textbf{Step 5: Estimates for $\widehat{\text{Mod}}$ term.} We claim the following:
\begin{align}
& \left(\int \frac{|\Ls^{\Bbbk - 1}\widehat{\text{Mod}}|^2}{1 + y^4} \right)^\frac{1}{2} + \left(\int \frac{|\As\Ls^{\Bbbk - 1}\widehat{\text{Mod}}|^2}{1 + y^2} \right)^\frac{1}{2}\nonumber\\
& \qquad \qquad \qquad \qquad\lesssim b_1^{(1 - \delta)(1 + \eta)}\left(\frac{\sqrt{\Es_{2\Bbbk}}}{M^{2\delta}} + b_1^{L + 1 + (1 -\delta)(1 + \eta)}\right),\label{eq:ModhatE2k}\\
&\left(\int \left|\Ls^\Bbbk \widetilde{\text{Mod}}\right|^2 \right)^\frac{1}{2}\lesssim b_1\left(\frac{\sqrt{\Es_{2\Bbbk}}}{M^{2\delta}} + b_1^{\eta(1 - \delta)}\sqrt{\Es_{2\Bbbk}} + b_1^{L+1 + (1 - \delta)(1 + \eta)}\right),\label{eq:ModtilE2k}
\end{align}
where
\begin{equation*}
\widetilde{\text{Mod}} = \widehat{\text{Mod}} + \partial_s \xi_L.
\end{equation*}
\noindent Let us prove \eqref{eq:ModhatE2k}. We only deal with the first term since the second term is estimated similarly. We recall from \eqref{def:Modhat} the definition of $\widehat{\text{Mod}}$,
\begin{align*}
\widehat{\text{Mod}} = - \left(\frac{\lambda_s}{\lambda} + b_1\right)\Lambda \tilde{Q}_b + \sum_{i = 1}^{L}\big[(b_i)_s + (2i - \gamma)b_1 b_i - b_{i + 1}\big]\left(\tilde T_i + \sum_{j = i + 1}\frac{\partial S_j}{\partial b_i}\chi_{B_1}\right),
\end{align*}
where $\tilde{Q}_b$ is defined as in \eqref{def:Qbtil} and we know from Lemma \ref{lemm:GenLk} that $T_i$ is admissible of degree $(i,i)$ and from Proposition \ref{prop:1} that $S_j$ is homogeneous of degree $(j, j-1, j)$. \\
Since $|b_j|\lesssim b_1^j$ and $\Ls \Lambda Q = 0$, we use Lemma \ref{lemm:actionLL} to estimate
\begin{align*}
\int \frac{|\Ls^{\Bbbk - 1}\Lambda \tilde{Q}_b|^2}{1 + y^4} &\lesssim \sum_{i = 1}^L b_i^2\int \frac{|\Ls^{\Bbbk - 1} \Lambda \tilde{T}_i|^2}{1 + y^4} + \sum_{i = 2}^{L+2} \int \frac{|\Ls^{\Bbbk - 1} \Lambda \tilde{S}_i|^2}{1 + y^4}\\
&\quad\lesssim \sum_{i = 1}^L b_1^{2i}\int_{y \leq 2B_1}\frac{y^{d-1}dy}{1 + y^{4(\Bbbk - i) + 2\gamma}} + \sum_{i = 2}^{L+1}b_1^{2i}\int_{y \leq 2B_1} \frac{y^{d-1}dy}{1 + y^{4(\Bbbk - i + 1) + 2\gamma}}\\
&\qquad \qquad + b_1^{2L + 4}\int_{y \leq 2B_1}\frac{y^{d-1}dy}{1 + y^{4\hbar + 2\gamma}} \lesssim b_1^2,
\end{align*}
where we used the algebra $4(\Bbbk - L) + 2\gamma - d + 1 = 5 - 4\delta > 1$. \\
Using the cancellation $\Ls^{\Bbbk}T_i = 0$ for $1 \leq i \leq L$ and the admissibility of $T_i$, we estimate
$$\sum_{i = 1}^{L} \int \frac{|\Ls^{\Bbbk-1} (\chi_{B_1}T_i)|^2}{1 + y^4} \lesssim \sum_{i = 1}^{L} \int_{B_1 \leq y \leq 2B_1 } y^{4(i - \Bbbk) - 2\gamma + d -1}dy \lesssim b_1^{2(1 - \delta)(1 + \eta)}.$$
Using the homogeneity of $S_j$, we estimate for $1 \leq i \leq L$,
$$\sum_{j = i + 1}^{L+2} \int \frac{1}{1 + y^4}\left|\Ls^{\Bbbk - 1}\left(\chi_{B_1}\frac{\partial S_j}{\partial b_i}\right)\right|^2 \lesssim \sum_{j = i + 1}^{L+2}b_1^{2(j-i)}\int_{B_1 \leq y \leq 2B_1}y^{4(j - 1 -\Bbbk) - 2\gamma}y^{d-1}dy \lesssim b_1^2,$$
provided that $\eta \leq \frac{1}{\delta} - 1$.

The collection of the above bounds together with \eqref{eq:ODEbkl} and \eqref{eq:ODEbL} yields
$$ \left(\int \frac{|\Ls^{\Bbbk - 1}\widehat{\text{Mod}}|^2}{1 + y^4} \right)^\frac{1}{2} \lesssim b_1^{(1 - \delta)(1 + \eta)}\left(\frac{\sqrt{\Es_{2\Bbbk}}}{M^{2\delta}} + b_1^{L+1 + (1-\delta)(1+ \eta)}\right).$$
The same estimate holds for $\left(\int \frac{|\As\Ls^{\Bbbk - 1}\widehat{\text{Mod}}|^2}{1 + y^2} \right)^\frac{1}{2}$ by following the same lines as above. This concludes the proof of \eqref{eq:ModhatE2k}. \\

We now prove \eqref{eq:ModtilE2k}. Let us write
\begin{align*}
\widetilde{\text{Mod}} &= - \left(\frac{\lambda_s}{\lambda} + b_1\right)\Lambda \tilde{Q}_b + \sum_{i = 1}^{L-1}\big[(b_i)_s + (2i - \gamma)b_1 b_i - b_{i + 1}\big]\tilde T_i\\
& \quad + \sum_{i = 1}^{L}\big[(b_i)_s + (2i - \gamma)b_1 b_i - b_{i + 1}\big] \chi_{B_1}\sum_{j = i+1}^{L+2}\frac{\partial S_j}{\partial b_i}\\
&\quad + \left[(b_L)_s + (2i - \gamma)b_1 b_L + \frac{d}{ds}\left\{\frac{\left<\Ls^Lq, \chi_{B_0}\Lambda Q\right>}{\left<\Lambda Q, \chi_{B_0}\Lambda Q\right>} \right\} \right]\tilde T_L + \frac{\left<\Ls^Lq, \chi_{B_0}\Lambda Q\right>}{\left<\Lambda Q, \chi_{B_0}\Lambda Q\right>}\partial_s\tilde{T}_L. 
\end{align*}
Proceeding as for the proof of \eqref{eq:ModhatE2k} yields the estimate
\begin{align*}
\int |\Ls^\Bbbk \Lambda \tilde{Q}_b|^2 + \sum_{i = 1}^{L-1}\int |\Ls^\Bbbk \tilde{T}_i|^2 + \sum_{i = 1}^L\sum_{j = i + 1}^{L+2}\int \left|\Ls^\Bbbk \left(\chi_{B_1}\frac{\partial S_j}{\partial b_i}\right)\right|^2 \lesssim b_1^2,
\end{align*}
and 
\begin{equation}\label{est:LkTL}
\int |\Ls^\Bbbk \tilde{T}_L|^2 \lesssim b_1^{2(1 - \delta)(1 + \eta)}.
\end{equation}
From \eqref{est:LQ2chiB0} and \eqref{est:LsLqchiB0LQ}, we have the bound
\begin{equation}\label{est:CLxiL}
\left|\frac{\left<\Ls^Lq, \chi_{B_0}\Lambda Q\right>}{\left<\Lambda Q, \chi_{B_0}\Lambda Q\right>}\right| \lesssim B_0^{2(1 - \delta)}\sqrt{\Es_{2\Bbbk}} = b_1^{-(1 - \delta)}\sqrt{\Es_{2\Bbbk}}.
\end{equation}
We also have 
$$\int |\Ls^\Bbbk (\partial_s \chi_{B_1} T_L)|^2 \lesssim b_1^2 \int_{B_1 \leq y \leq 2B_1} \frac{y^{d-1}dy}{y^{4(\Bbbk - L) + 2\gamma}} \lesssim b_1^2b_1^{2(1 - \delta)(1 + \eta)}.$$
The collection of the above bounds together with Lemmas \ref{lemm:mod1} and \ref{lemm:mod2} yields
\begin{align*}
\left(\int |\Ls^\Bbbk \widetilde{\text{Mod}}|^2 \right)^\frac{1}{2} &\lesssim b_1 \left(\frac{\sqrt{\Es_{2\Bbbk}}}{M^{2\delta}} + b_1^{L+1 + (1 - \delta)(1 +\eta)}\right)\\
& \quad + b_1^{(1 - \delta)(1 + \eta)} b_1^{\delta}\left(C(M)\sqrt{\Es_{2\Bbbk}} + b_1^{L + 1 + (1 - \delta)(1 + \eta)}\right)\\
&\qquad + b_1^{-(1 - \delta)} \sqrt{\Es_{2\Bbbk}}b_1b_1^{(1 - \delta)(1 + \eta)}\\
&\qquad \quad \lesssim b_1\left(\frac{\sqrt{\Es_{2\Bbbk}}}{M^{2\delta}} + b_1^{\eta(1 - \delta)}\sqrt{\Es_{2\Bbbk}} + b_1^{L+1 + (1 - \delta)(1 + \eta)}\right),
\end{align*}
which is the conclusion of \eqref{eq:ModtilE2k}.\\

Injecting the estimates \eqref{eq:PsibtilE2k}, \eqref{eq:ModhatE2k} and \eqref{eq:ModtilE2k} into \eqref{eq:EnerID2}, we arrive at 
\begin{align}
&\frac{1}{2}\frac{d}{dt}\left\{ \frac{\Es_{2\Bbbk}}{\lambda^{4\Bbbk - d}}\Big[1 + \Oc(b_1)\Big] \right\}\nonumber\\
&\qquad \leq -\frac{1}{8} \int |\tilde\Ls_\lambda v_{2\Bbbk - 1}|^2 + \frac{b_1}{\lambda^{4\Bbbk - d +2}}\left(\frac{\Es_{2\Bbbk}}{M^{2\delta}} + b_1^{\eta(1 - \delta)}\Es_{2\Bbbk} + b_1^{L + (1 - \delta)(1 + \eta)}\sqrt{\Es_{2\Bbbk}}\right) \nonumber\\
&\qquad  \quad + \frac{b_1\sqrt{\Es_{2\Bbbk}}}{\lambda^{4\Bbbk - d +2}}\left[\left(\int \frac{|\As\Ls^{\Bbbk - 1}\Fc_1|^2}{1 + y^2} \right)^\frac{1}{2} + \left(\int \frac{|\Ls^{\Bbbk - 1}\Fc_1|^2}{1 + y^4} \right)^\frac{1}{2} \right]\nonumber\\
&\qquad \qquad + \frac{1}{\lambda^{4\Bbbk - d +2}}\left\|\As \Ls^{\Bbbk -1}\Fc_1 \right\|_{L^2}^2 + \frac{1}{\lambda^{4\Bbbk - d +2}}\int \Ls^\Bbbk q \Ls^\Bbbk(\partial_s \xi_L). \label{eq:EnerID3}
\end{align}

\medskip

\noindent \textbf{Step 6: Estimates for the linear small term $\Hc(q)$.} We claim the following 
\begin{equation}\label{est:HqE2k}
\int |\As \Ls^{\Bbbk - 1}\Hc(q)|^2 + \int  \frac{|\As \Ls^{\Bbbk - 1}\Hc(q)|^2}{1 + y^2} + \frac{| \Ls^{\Bbbk - 1}\Hc(q)|^2}{1 + y^4} \lesssim b_1^2 \Es_{2\Bbbk}. 
\end{equation}
We only deal with the estimate for the first term because the last two terms are estimated similarly. Let us rewrite from \eqref{def:Lq} the definition of $\Hc(q)$, 
$$\Hc(q) = \Phi q \quad \text{with} \quad \Phi = \frac{(d-1)}{y^2}\left[\cos(2Q) - \cos(2Q + 2\tilde{\Theta}_b)\right],$$
where 
$$\tilde{\Theta}_b = \sum_{i = 1}^Lb_i \tilde{T}_i + \sum_{i = 2}^{L+2}\tilde{S}_i(b,y).$$
From the asymptotic behavior of $Q$ given in \eqref{eq:asymQ}, the admissibility of $T_i$ and the homogeneity of $S_i$, we deduce that $\Phi$ is a regular function both at the origin and at infinity. We then apply the Leibniz rule \eqref{eq:LeibnizALk} with $k = \Bbbk - 1$ and $\phi = \Phi$ to write 
$$\As \Ls^{\Bbbk - 1} \Hc(q) = \sum_{m = 0}^{\Bbbk - 1}\Big[ q_{2m +1}\Phi_{2\Bbbk - 1, 2m + 1} + q_{2m}\Phi_{2\Bbbk - 1, 2m}\Big],$$
where $\Phi_{2\Bbbk - 1, i}$ with $0 \leq i \leq 2\Bbbk - 1$ are defined by the recurrence relation given in Lemma \ref{lemm:LeibnizLk}. In particular, we have the following estimate
$$|\Phi_{k, i}| \lesssim \frac{b_1}{1 + y^{\gamma + (k-i)}}\lesssim \frac{b_1}{1 + y^{1 + k - i}}, \quad \forall k \geq 1, \;\; 0 \leq i \leq k.$$
Hence, we estimate from \eqref{eq:Enercontrol},
\begin{align*}
\int |\As \Ls^{\Bbbk - 1} \Hc(q)|^2 &\lesssim \sum_{m = 0}^{\Bbbk - 1}\left[\int|q_{2m + 1} \Phi_{2\Bbbk - 1,2m + 1}|^2 + \int |q_{2m} \Phi_{2\Bbbk - 1,2m}|^2\right]\\
&\quad \lesssim b_1^2\sum_{m = 0}^{\Bbbk - 1}\left[\int \frac{|q_{2m + 1}|^2}{1 + y^{2 + 2(2\Bbbk - 1 - 2m - 1)}} +  \int \frac{|q_{2m}|^2}{1 + y^{2 + 2(2\Bbbk - 1 - 2m)}} \right]\\
&\qquad \lesssim b_1^2\sum_{m = 0}^{\Bbbk - 1}\left[\int \frac{|q_{2m + 1}|^2}{1 + y^{2 + 4(\Bbbk - 1 - m )}} +  \int \frac{|q_{2m}|^2}{1 + y^{4 + 4(\Bbbk - 1 - m)}} \right]\\
&\qquad \quad\lesssim b_1^2 \Es_{2\Bbbk}.
\end{align*}
This concludes the proof of \eqref{est:HqE2k}.\\

\noindent \textbf{Step 7: Estimates for the nonlinear term $\Nc(q)$.} This is the most delicate point in the proof of \eqref{eq:Es2sLya}.  We claim the following 
\begin{equation}\label{est:NqE2k1}
\int |\As \Ls^{\Bbbk - 1}\Nc(q)|^2  \lesssim b_1^{2L+ 1 + 2(1 - \delta)(1 + \eta)}, 
\end{equation}
\begin{equation}\label{est:NqE2k}
\int  \frac{|\As \Ls^{\Bbbk - 1}\Nc(q)|^2}{1 + y^2} + \int \frac{| \Ls^{\Bbbk - 1}\Nc(q)|^2}{1 + y^4} \lesssim b_1^{2L+ 2 + 2(1 - \delta)(1 + \eta)}, 
\end{equation}
provided that $\eta$ and $1/L$ are small enough. We only deal with the proof of \eqref{est:NqE2k1} since the same proof holds for \eqref{est:NqE2k}.\\

\noindent - \textit{Control for $y < 1$.} Let us rewrite from \eqref{def:Nq} the definition of $\Nc(q)$,
$$\Nc(q) = \frac{q^2}{y} \Phi \quad \text{with} \quad \Phi = \left[-\frac{(d-1)}{y}\int_0^1 (1 -\tau)\sin(2\tilde{Q}_b + 2\tau q)d\tau\right].$$
From \eqref{eq:expandqat0} and the admissibility of $T_i$, we write
\begin{equation}\label{eq:expanq2y}
\frac{q^2}{y} = \frac{1}{y}\left(\sum_{i = 0}^{\Bbbk} c_i T_i(y) + r_q(y)\right)^2 = \sum_{i = 0}^{\Bbbk - 1}\tilde{c}_i y^{2i + 1} + \tilde{r}_q \quad \text{for}\;\; y < 1,
\end{equation}
where 
\begin{equation*}
|\tilde c_i|\lesssim \Es_{2\Bbbk}, \quad |\partial^j_y \tilde r_q(y)| \lesssim y^{2\Bbbk - \frac d2 - j} |\ln y|^\Bbbk \Es_{2\Bbbk}, \quad 0\leq j \leq 2\Bbbk - 1, \;\; y < 1.
\end{equation*}
Let $\tau \in [0,1]$ and
$$v_\tau = \tilde{Q}_b + \tau q,$$
we obtain from Proposition \ref{prop:1} and the expansion \eqref{eq:expandqat0}, 
$$v_\tau = \sum_{i = 0}^{\Bbbk - 1}\hat c_i y^{2i + 1} + \hat r_q,$$
with 
$$|\hat c_i|\lesssim 1, \quad |\py^j \hat r_q| \lesssim y^{2\Bbbk - \frac d2 - j}|\ln y|^\Bbbk, \;\; 0 \leq j \leq 2\Bbbk - 1, \;\; y <1.$$
Together with the Taylor expansion of $\sin(x)$ at $x=0$, we write
\begin{equation}\label{eq:expPhiy}
\Phi(q) = \sum_{i = 0}^{\Bbbk - 1}\bar c_i y^{2i} + \bar r_q,
\end{equation}
with 
$$|\bar c_i| \lesssim 1, \quad |\py^j \bar r_q| \lesssim y^{2\Bbbk - \frac{d}{2} - 1 - j}|\ln y|^\Bbbk, \quad 0 \leq j \leq 2\Bbbk -1, \quad y < 1.$$
From \eqref{eq:expanq2y} and \eqref{eq:expPhiy}, we have the expansion of $\Nc$ near the origin,
$$\Nc(q) = \sum_{i = 0}^{\Bbbk - 1}\hat{\tilde{c}}_i y^{2i + 1} + \hat{\tilde{r}}_q,$$
with 
$$|\hat{\tilde{c}}_i| \lesssim \Es_{2\Bbbk}, \quad |\py^j \hat{\tilde{r}}_q| \lesssim y^{2\Bbbk - \frac d2 -j}|\ln y|^\Bbbk \Es_{2\Bbbk}, \quad 0 \leq j \leq 2\Bbbk - 1, \;\; y < 1.$$
From the definition of $\As$ and $\As^*$ (see \eqref{def:As} and \eqref{def:Astar}), one can check that for $y < 1$,
$$|\As \Ls^{\Bbbk - 1} \hat{\tilde{r}}_q | \lesssim \sum_{i = 0}^{2\Bbbk - 1} \frac{\py^i \hat{\tilde{r}}_q}{y^{2\Bbbk - 1 - i}} \lesssim \Es_{2\Bbbk} \sum_{i = 0}^{2\Bbbk -1}\frac{y^{2\Bbbk - \frac{d}{2} - i}|\ln y|^\Bbbk}{y^{2\Bbbk - 1 - i}} \lesssim y^{-\frac{d}{2} + 1}|\ln y|^\Bbbk\Es_{2\Bbbk}.$$
Note from the asymptotic behavior \eqref{eq:asympV} of $V$ that $\As(y) = \Oc(y^2)$ for $y < 1$, which  implies 
$$\left|\As \Ls^{\Bbbk - 1} \left(\sum_{i = 0}^{\Bbbk - 1}\hat{\tilde{c}}_i y^{2i + 1}\right)\right| \lesssim \sum_{i = 0}^{\Bbbk -1}|\hat{\tilde{c}}_i| y^2 \lesssim y^2\Es_{2\Bbbk}.$$
We then conclude 
$$\int_{y < 1} |\As \Ls^{\Bbbk - 1}\Nc(q)|^2 \lesssim \Es_{2\Bbbk}^2\int_{y < 1}y|\ln y|^{2\Bbbk} dy \lesssim \Es_{2\Bbbk}^2 \lesssim b_1^{2L + 1 + 2(1 - \delta)(1 + \eta)}.$$

\medskip

\noindent - \textit{Control for $y \geq 1$.} Let us rewrite from the definition of $\Nc(q)$,
\begin{equation}\label{def:Zpsi}
\Nc(q) = Z^2 \psi, \quad Z = \frac{q}{y}, \quad \psi = -(d-1)\int_0^1(1 - \tau)\sin(2\tilde{Q}_b + 2\tau q) d\tau.
\end{equation}
Note from the definitions of $\As$ and $\As^*$ that 
$$\forall k \in \mathbb{N}, \quad |\As \Ls^{k} f| \lesssim \sum_{i = 0}^{2k + 1} \frac{|\py^i f|}{y^{2k + 1 - i}},$$
from which and the Leibniz rule, we write 
\begin{align*}
\int_{y \geq 1}|\As \Ls^{\Bbbk - 1}\Nc(q)|^2 &\lesssim \sum_{k = 0}^{2\Bbbk - 1}\int_{y \geq 1}\frac{|\py^k\Nc(q)|^2}{y^{4\Bbbk - 2k - 2}}\\
&\lesssim \sum_{k = 0}^{2\Bbbk - 1}\sum_{i = 0}^k \int_{y \geq 1} \frac{|\py^i Z^2|^2|\py^{k - i}\psi|^2}{y^{4\Bbbk - 2k - 2}}\\
&\lesssim \sum_{k = 0}^{2\Bbbk - 1}\sum_{i = 0}^k \sum_{m = 0}^i \int_{y \geq 1} \frac{|\py^m Z|^2 |\py^{i - m}Z|^2 |\py^{k - i}\psi|^2}{y^{4\Bbbk - 2k - 2}}.
\end{align*}

We aim at using the pointwise estimate \eqref{eq:pointwise_yg1} to prove that for $0 \leq k \leq 2\Bbbk - 1$, $0 \leq i \leq k$ and $0 \leq m \leq i$,
\begin{equation}\label{def:Akmi}
A_{k,i,m} := \int_{y \geq 1} \frac{|\py^m Z|^2 |\py^{i - m}Z|^2 |\py^{k - i}\psi|^2}{y^{4\Bbbk - 2k - 2}} \lesssim b_1^{2L + 1 + 2(1 - \delta)(1 + \eta)},
\end{equation}
which concludes the proof of \eqref{est:NqE2k1}.

To prove \eqref{def:Akmi}, we distinguish in 3 cases:\\

\noindent - \underline{\textbf{The initial case: $k = 0$.}} Since $0 \leq m \leq i \leq k$, then $k = i = m = 0$. Although this is the simplest case, it gives us a basic idea to handle other cases. From \eqref{def:Zpsi}, it is obvious to see that $|\psi|$ is uniformly bounded.  We write
\begin{align*}
A_{0,0,0}  = \int_{y \geq 1} \frac{|q|^4|\psi|^2}{y^{4\Bbbk + 2}} y^{d-1}dy  \lesssim \int_{1 \leq y \leq B_0} \frac{|q|^4}{y^{4\Bbbk + 3 - d}} dy + \int_{y \geq B_0} \frac{|q|^4}{y^{4\Bbbk + 3 - d}}dy.
\end{align*}
Using \eqref{eq:pointwise_yg1}, Definition \ref{def:Skset}, $b_1 \sim \frac{1}{s}$ and the fact that $d = 4\hbar + 2\gamma + 4\delta$ (see \eqref{def:kdeltaplus}), we estimate 
\begin{align*}
&\int_{1 \leq y \leq B_0} \frac{|q|^4}{y^{4\Bbbk + 3 - d}}dy\\
& \quad \lesssim \left\|\frac{y^{d-2}|q|^2}{y^{2(2\Bbbk - 1)}}\right\|_{L^\infty(y > 1)}\left\|\frac{y^{d-2}|q|^2}{y^{2(2\ell + 2\hbar + 3)}}\right\|_{L^\infty(y > 1)}\int_{1 \leq y \leq B_0}y^{4\ell + 5 - 4\delta - 2\gamma}dy\\
& \qquad \lesssim \Es_{2\Bbbk}\Es_{2(\ell + \hbar + 2)} B_0^{4\ell + 6 - 4\delta - 2\gamma}\\
& \quad \qquad \lesssim Kb_1^{2L + 2(1 - \delta)(1 + \eta)}b_1^{2(\ell + 1) + 2(1 - \delta) - K\eta}b_1^{-2\ell - 3 +  2\delta + \gamma}\\
&\qquad \qquad \lesssim K b_1^{2L + 2(1 - \delta)(1 + \eta)} b_1^{1 + \gamma - K\eta} \lesssim b_1^{2L + 1+ 2(1 - \delta)(1 + \eta)}.
\end{align*}
For the integral on the domain $y \geq B_0$, let us write
\begin{align*}
&\int_{y \geq B_0} \frac{|q|^4}{y^{4\Bbbk + 3 - d}}dy\\
&\quad \lesssim \left\|\frac{y^{d-2}|q|^2}{y^{2(2\Bbbk  - 2\ell - 1)}}\right\|_{L^\infty(y > 1)}\left\|\frac{y^{d-2}|q|^2}{y^{2(2\ell + 2\hbar + 1)}}\right\|_{L^\infty(y > 1)}\int_{y \geq B_0} \frac{dy}{y^{4\delta + 2\gamma - 1}}\\
& \qquad \lesssim \Es_{2(\Bbbk - \ell)}\Es_{2(\ell + \hbar + 1)} B_0^{2 - 4\delta - 2\gamma}\\
& \quad \qquad \lesssim b_1^{2(\Bbbk - \ell - \hbar -1) + 2(1 - \delta) - K\eta}b_1^{2\ell + 2(1 - \delta) - K\eta}b_1^{2\delta + \gamma -1}\\
& \qquad \qquad \lesssim b_1^{2L + 2(1 - \delta)(1 + \eta)} b_1^{1 + \gamma - (K + 2(1 - \delta))\eta} \lesssim b_1^{2L + 1+ 2(1 - \delta)(1 + \eta)}.
\end{align*}
This concludes the proof of \eqref{def:Akmi} when $k = i = m = 0$. \\

\noindent  - \underline{\textbf{Case II: $k \geq 1$ and $k = i$}}. We first use the Leibniz rule to write
\begin{equation}\label{eq:pylZexp}
\forall l \in \mathbb{N}, \quad |\py^l Z|^2 \lesssim \sum_{j = 0}^l \frac{|\py^j q|^2}{y^{2 + 2l - 2j}},
\end{equation}
from which, 
\begin{align*}
A_{k,k,m} &\lesssim \sum_{j = 0}^m \sum_{l = 0}^{k-m} \int_{y \geq 1} \frac{|\py^j q|^2 |\py^l q|^2}{y^{4 \Bbbk - 2j - 2l + 2}}y^{d-1}dy.
\end{align*}
We claim that for all $(j,l) \in \mathbb{N}^2$ and $1 \leq j + l \leq 2\Bbbk -1$,
\begin{equation}\label{est:Bjl}
B_{j,l,0}:= \int_{y \geq 1} \frac{|\py^j q|^2 |\py^l q|^2}{y^{4 \Bbbk - 2j - 2l + 2}}y^{d-1}dy \lesssim b_1^{2L +1 + 2(1 - \delta)(1 + \eta) + \frac{(\gamma - 1)}{2}},
\end{equation}
which immediately follows \eqref{def:Akmi} for the case when $k = i$. 

To prove \eqref{est:Bjl}, we proceed as for the case $k = 0$ by splitting the integral in two parts as follows:
\begin{align*}
B_{j,l,0} &= \int_{1 \leq y \leq B_0} \frac{\big(y^{d-2}|\py^j q|^2\big)\big(y^{d-2}|\py^l q|^2\big)}{y^{4\Bbbk - 2j - 2l + 4\hbar + 6}}y^{7 - 4\delta - 2\gamma}dy\\
&\quad  + \int_{y \geq B_0} \frac{\big(y^{d-2}|\py^j q|^2\big)\big(y^{d-2}|\py^l q|^2\big)}{y^{4\Bbbk - 2j - 2l + 4\hbar}}\frac{dy}{y^{4\delta + 2\gamma - 1}}\\
&\qquad \lesssim \left\|\frac{\big(y^{d-2}|\py^j q|^2\big)\big(y^{d-2}|\py^l q|^2\big)}{y^{4\Bbbk - 2j - 2l + 4\hbar + 6}} \right\|_{L^\infty(y \geq 1)}b_1^{2\delta + \gamma - 4}\\
&\qquad \quad + \left\|\frac{\big(y^{d-2}|\py^j q|^2\big)\big(y^{d-2}|\py^l q|^2\big)}{y^{4\Bbbk - 2j - 2l + 4\hbar}} \right\|_{L^\infty(y \geq 1)}b_1^{2\delta + \gamma - 1}\\
& \qquad \qquad = \left\|\frac{\big(y^{d-2}|\py^j q|^2\big)\big(y^{d-2}|\py^l q|^2\big)}{y^{2J_1 - 2j + 2J_2 - 2l}} \right\|_{L^\infty(y \geq 1)}b_1^{2\delta + \gamma - 4}\\
&\qquad \qquad \quad + \left\|\frac{\big(y^{d-2}|\py^j q|^2\big)\big(y^{d-2}|\py^l q|^2\big)}{y^{2J_3 - 2j + 2J_4 - 2l}} \right\|_{L^\infty(y \geq 1)}b_1^{2\delta + \gamma - 1}\\
& \qquad \qquad \qquad := B_{j,l,0,J_1, J_2} b_1^{2\delta + \gamma - 4} + B_{j,l,0, J_3, J_4}b_1^{2\delta + \gamma - 1},
\end{align*}
where $J_n (n = 1,2,3,4)$ satisfy
$$J_1 + J_2 = 2\Bbbk + 2\hbar + 3, \quad J_3 + J_4 = 2\Bbbk + 2\hbar.$$
We now estimate $B_{j,l,0,J_1, J_2}$.\\
- If $l$ is even, we take 
$$J_2 = \left\{ \begin{array}{ll}
l + 2 &\quad \text{if}\quad l \leq 2\Bbbk - 4,\\
l &\quad \text{if}\quad l = 2\Bbbk - 2.
\end{array} \right.$$
This gives
$$2\hbar + 4 \leq J_2 \leq 2\Bbbk - 2, \quad 2\hbar + 5 \leq J_1 = 2\Bbbk + 2\hbar + 3 - J_2 \leq 2\Bbbk - 1.$$
Using \eqref{eq:pointwise_yg1}, we have the estimate
\begin{align*}
 B_{j,l,0,J_1, J_2}& \lesssim \left\|\frac{y^{d-2}|\py^j q|^2}{y^{2J_1 - 2j}}\right\|_{L^\infty(y \geq 1)} \left\|\frac{y^{d-2}|\py^l q|^2}{y^{2J_2 - 2l}} \right\|_{L^\infty(y \geq 1)}\\
 &\quad  \lesssim \Es_{J_1 + 1} \sqrt{\Es_{J_2} \Es_{J_2 + 2}}.
\end{align*}
- If $l$ is odd, we simply take $J_2 = l + 1$ which gives
 $$2\hbar + 4 \leq J_2 \leq 2\Bbbk - 2, \quad 2\hbar + 5 \leq J_1 \leq 2\Bbbk - 1.$$
Hence,
\begin{align*}
 B_{j,l,0,J_1, J_2}&\lesssim \Es_{J_1 + 1} \sqrt{\Es_{J_2} \Es_{J_2 + 2}}.
\end{align*}
Recall from Definition \eqref{def:Skset} that for all even integer $m$ in the range $2\hbar + 4 \leq m \leq 2\Bbbk$,
\begin{equation}\label{est:Esmeven}
\Es_{m} \leq \left\{
\begin{array}{ll}
b_1^{\frac{\ell}{2\ell - \gamma}(2m - d)}&\quad \text{for}\quad 2\hbar + 4 \leq m \leq 2\hbar +2\ell.\\
b_1^{m - 2\hbar -2 + 2(1 -\delta) - K\eta}&\quad \text{for}\quad 2\hbar + 2\ell + 2 \leq m \leq 2\Bbbk.
\end{array} 
\right.
\end{equation}
- If $J_1 + 1 \geq 2\hbar + 2\ell + 2$ and  $J_2 \geq 2\hbar + 2\ell + 2$, then
$$B_{j,l,0,J_1, J_2} \lesssim b_1^{J_1 + J_2 - 4\hbar -2 + 4(1 - \delta) - 2K\eta} \lesssim b_1^{2L + 2 + 4(1 - \delta) - K\eta}.$$
- If $J_1 + 1 \leq 2\hbar + 2\ell$, then $J_2 = 2\Bbbk + 2\hbar + 3 - J_1 \geq 2\Bbbk - 2\ell + 4 \geq 2\hbar + 2\ell + 2$ because $\Bbbk \gg \ell$. This follows
$$B_{j,l,0,J_1, J_2} \lesssim b_1^{\frac{\ell}{2\ell - \gamma}(2J_1 + 2 - d) + J_2 + 1 - 2(\hbar +1) + 2(1 - \delta) - K\eta} \lesssim b_1^{2L + 2 + 4(1 - \delta)  - K\eta}.$$
Hence, we obtain 
$$B_{j,l,0,J_1, J_2} \lesssim b_1^{2L + 2 + 4(1 - \delta)  - K\eta} \quad \text{for}\quad J_1 + J_2 = 2\Bbbk + 2\hbar +3.$$
Similarly, one prove that 
$$B_{j,l,0,J_3, J_4} \lesssim b_1^{2L - 1 + 4(1 - \delta) - K\eta} \quad \text{for} \quad J_3 + J_4 = 2\Bbbk + 2\hbar.$$
Therefore, 
\begin{align*}
B_{j,l,0} & \lesssim b_1^{2L + 2 + 4(1 - \delta) - K\eta}b_1^{2\delta + \gamma - 4} + b_1^{2L -1 + 4(1 - \delta) - K\eta}b_1^{2\delta + \gamma - 1}\\
& \quad \lesssim b_1^{2L + 1 + 2(1 - \delta)(1 + \eta) + (\gamma-1) - (K + 2 -2\delta)\eta} \lesssim b_1^{2L + 1 + 2(1 - \delta)(1 + \eta) + \frac{(\gamma - 1)}{2}},
\end{align*}
for $\eta \leq \frac{\gamma - 1}{2(K + 2 - 2\delta)}$. This concludes the proof of \eqref{est:Bjl} as well as \eqref{def:Akmi} when $k = i$. \\

\noindent  - \underline{\textbf{Case III: $k \geq 1$ and $k - i \geq 1$.}} Let us write from \eqref{def:Akmi} and \eqref{eq:pylZexp}, 
\begin{equation}\label{eq:Akmi11}
A_{k,m,i} \lesssim \sum_{j = 0}^{m}\sum_{l = 0}^{i - m} \int_{y \geq 1} \frac{|\py^j q|^2 |\py^l q|^2}{y^{4\Bbbk - 2j - 2l + 2}} \frac{|\py^{k - i} \psi|^2}{y^{-2(k - i)}}.
\end{equation}
At this stage, we need to precise the decay of $|\py^n \psi|$ to archive the bound \eqref{def:Akmi}. To do so, let us recall that $T_i$ is admissible of degree $(i,i)$ (see Lemma \ref{lemm:GenLk}) and $S_i$ is homogeneous of degree $(i,i-1,i)$ (see Proposition \ref{prop:1}). Together with \eqref{eq:asymQ}, we estimate
$$ \forall j \geq 1,\quad  |\py^j \tilde{Q}_b | \lesssim \frac{1}{y^{\gamma + j}} + \sum_{l = 1}^{2L + 2}\frac{b_1^l y^{2l}}{y^{\gamma + j}} \mathbf{1}_{\{y \leq 2B_1\}} \lesssim \frac{b_1^{-(2L+2)\eta}}{y^{\gamma + j}}.$$
Let $\tau \in [0,1]$ and $v_\tau = \tilde{Q}_b + \tau q$. We use the Faa di Bruno formula to write\\ 
\begin{align*}
\forall n \in \mathbb{N}, \quad |\py^n \psi|^2 &\lesssim \int_0^1 \sum_{m^* = n} |\partial_{v_\tau}^{m_1 + \cdots+m_n}\sin(v_\tau)|^2  \prod_{i = 1}^n |\py^i \tilde{Q}_b + \py^iq|^{2m_i} d\tau \nonumber\\
&\quad \lesssim \sum_{m^* = n}  \prod_{i = 1}^n \left(\frac{b_1^{-C(L)\eta}}{y^{2\gamma + 2i}} +  |\py^iq|^2\right)^{m_i}, \quad m^* = \sum_{i = 1}^n im_i.
\end{align*}
For $1 \leq y \leq B_0$, we use \eqref{eq:pointwise_yg1} to estimate 
\begin{equation*}
|\py^iq|^2 = y^{4\Bbbk - 2i -2}\left|\frac{\py^i q}{y^{2\Bbbk - i - 1}}\right|^2 \leq B_0^{4\Bbbk - 2i - d}\Es_{2\Bbbk} \leq b_1^{-C(K)\eta + i + \gamma} \leq \frac{b_1^{-C(K)\eta}}{y^{2\gamma + 2i}},
\end{equation*}
 from which, we have
\begin{equation}\label{est:psin1B0}
|\py^n \psi|^2 \lesssim \sum_{m^* = n}  \prod_{i = 1}^n \left(\frac{b_1^{-C(L)\eta}}{y^{2\gamma + 2i}} +  \frac{b_1^{-C(K)\eta}}{y^{2\gamma + 2i}}\right)^{m_i} \lesssim \frac{b_1^{-C(K, L)\eta}}{y^{2\gamma + 2n}}, \quad \forall 1 \leq y \leq B_0.
\end{equation}
For $y \geq B_0$, we use again \eqref{eq:pointwise_yg1} to write for all $1 \leq n \leq 2\Bbbk - 1$,
\begin{align}
|\py^n \psi|^2 &\lesssim \sum_{m^* = n}  \prod_{i = 1}^{2\hbar + 2\ell + 1}\left(\frac{b_1^{-C(L)\eta}}{y^{2\gamma + 2i}} +  y^{4\hbar + 4\ell + 2 - 2i}\left|\frac{\py^i q}{y^{2\hbar + 2\ell + 1 - i}}\right|^2\right)^{m_i}\nonumber\\
&\quad \times \prod^{n}_{i = 2\hbar + 2\ell + 1}\left(\frac{b_1^{-C(L)\eta}}{y^{2\gamma + 2i}} +  \left|\py^i q\right|^2\right)^{m_i}\nonumber\\
& \lesssim \sum_{m^* = n}  \prod_{i = 1}^{2\hbar + 2\ell + 1} \left(b_1^{-C(L)\eta + \gamma + i} + b_1^{-K\eta + i+ \gamma}b_1^{2\ell + 2(1 - \delta)}y^{4\ell + 4(1 - \delta)} \right)^{m_i}\nonumber\\
& \quad \times \prod^{n}_{i = 2\hbar + 2\ell + 1}\left(b_1^{-C(L)\eta + \gamma + i} + b_1^{-K\eta + \gamma + i}\right)^{m_i}\nonumber\\
&\lesssim b_1^{-C(L,K)\eta + n + \gamma \sum_{i = 1}^n m_i}\left(b_1y^2\right)^{(2\ell + 2(1 - \delta))\sum_{i = 1}^{2\hbar +2\ell + 1}m_i}, \quad \forall y \geq B_0. \label{est:psinB0g}
\end{align}
Injecting \eqref{est:psin1B0} and \eqref{est:psinB0g} into \eqref{eq:Akmi11}, we arrive at 
\begin{align*}
A_{k,i,m} &\lesssim b_1^{-C\eta}\sum_{j = 0}^{m}\sum_{l = 0}^{i - m} \left(\int_{1 \leq y \leq B_0} \frac{|\py^j q|^2 |\py^l q|^2}{y^{4\Bbbk - 2j - 2l + 2 + 2\gamma}} +b_1^\alpha \int_{ y \geq B_0} \frac{|\py^j q|^2 |\py^l q|^2}{y^{4\Bbbk - 2j - 2l + 2 - 2\alpha}}\right),
\end{align*}
where $\alpha = k - i + \big(2\ell + 2(1 - \delta)\big)\sum_{i = 1}^{2\hbar +2\ell + 1}m_i$. Arguing as for the proof of \eqref{est:Bjl}, we end up with 
\begin{align*}
A_{k,i,m} &\lesssim b_1^{-C\eta} \left(b_1^{2L + 1 + \gamma + 2(1 - \delta)(1 -\eta) + \frac{(\gamma - 1)}{2}} + b_1^{2L + 1 + 2(1 - \delta)(1 -\eta) + \frac{(\gamma - 1)}{2}}\right)\\
&\qquad \lesssim b_1^{2L + 1 + 2(1 - \delta)(1 -\eta)},
\end{align*}
for $\eta$ small enough. This finishes the proof of \eqref{def:Akmi} as well as \eqref{est:NqE2k1}. Since the proof of \eqref{est:NqE2k} follows exactly the same lines as for the proof of \eqref{est:NqE2k1}, we omit its proof here. 

\medskip

Inserting \eqref{est:HqE2k}, \eqref{est:NqE2k1} and \eqref{est:NqE2k} into \eqref{eq:EnerID3} and recalling from Definition \eqref{def:Skset} that $\Es_{2\Bbbk} \leq K b_1^{2L + 2(1 - \delta)(1 + \eta)}$, we arrive at
\begin{align}
\frac{1}{2}\frac{d}{dt}\left\{ \frac{\Es_{2\Bbbk}}{\lambda^{4\Bbbk - d}}\Big[1 + \Oc(b_1)\Big] \right\}
&\lesssim \frac{b_1}{\lambda^{4\Bbbk - d +2}}\left(\frac{\Es_{2\Bbbk}}{M^{2\delta}} + b_1^{L + (1 - \delta)(1 + \eta)}\sqrt{\Es_{2\Bbbk}} + b_1^{2L + 2(1 - \delta)(1 + \eta)}\right)\nonumber\\
&\qquad + \frac{1}{\lambda^{4\Bbbk - d +2}}\int \Ls^\Bbbk q \Ls^\Bbbk(\partial_s \xi_L). \label{eq:EnerID4}
\end{align}

\noindent \textbf{Step 8: Time oscillations.} In this step, we want to find the contribution of the last term in \eqref{eq:EnerID4} to the estimate \eqref{eq:Es2sLya}. Let us write
\begin{align}
\frac{1}{\lambda^{4\Bbbk - d +2}}\int \Ls^\Bbbk q \Ls^\Bbbk(\partial_s \xi_L) &= \frac{d}{ds} \left\{\frac{1}{\lambda^{4\Bbbk - d +2}}\left[\int \Ls^\Bbbk q \Ls^\Bbbk \xi_L - \frac{1}{2}\int |\Ls^\Bbbk \xi_L|^2 \right]\right \}\nonumber\\
&\quad +\frac{4\Bbbk - d + 2}{\lambda^{4\Bbbk - d +2}} \frac{\lambda_s}{\lambda} \left[\int \Ls^\Bbbk q \Ls^\Bbbk \xi_L + \frac{1}{2}\int |\Ls^\Bbbk \xi_L|^2\right]\nonumber\\
&\qquad  - \frac{1}{\lambda^{4\Bbbk - d +2}}\int \Ls^\Bbbk(\ps q - \ps \xi_L)  \Ls^\Bbbk \xi_L.\label{id:Aextract}
\end{align}
From \eqref{est:LkTL} and \eqref{est:CLxiL}, we have
\begin{equation}\label{est:LkxiL2}
\int|\Ls^\Bbbk \xi_L|^2 \lesssim b_1^{2\eta(1 - \delta)}\Es_{2\Bbbk}.
\end{equation} 
This follows
\begin{align*}
\left|\int \Ls^\Bbbk q \Ls^\Bbbk \xi_L \right| &\lesssim \left(\int |\Ls^\Bbbk q|^2\right)^\frac{1}{2}\left(\int |\Ls^\Bbbk \xi_L|^2\right)^\frac{1}{2}\\
 &\quad \lesssim \sqrt{\Es_{2\Bbbk}} \, b_1^{-(1-\delta)}\sqrt{\Es_{2\Bbbk}}\, b_1^{(1-\delta)(1+\eta)} = b_1^{\eta(1 - \delta)}\Es_{2\Bbbk}.
\end{align*}
Since $\frac{dt}{ds} = \lambda^2$, we then write 
\begin{equation}\label{est:A1}
\frac{d}{ds} \left\{\frac{1}{\lambda^{4\Bbbk - d +2}}\Big[\int \Ls^\Bbbk q \Ls^\Bbbk \xi_L - \frac{1}{2}\int |\Ls^\Bbbk \xi_L|^2\Big]\right\} = \frac{d}{dt}\left(\frac{\Es_{2\Bbbk}}{\lambda^{4\Bbbk - d}}\mathcal{O}(b_1^{\eta(1 - \delta)}) \right).
\end{equation}
Noting from \eqref{eq:ODEbkl} that $\left|\frac{\lambda_s}{\lambda}\right| \lesssim b_1$, this gives
\begin{equation}\label{est:A2}
\left|\frac{\lambda_s}{\lambda} \left[\int \Ls^\Bbbk q \Ls^\Bbbk \xi_L + \frac{1}{2}\int |\Ls^\Bbbk \xi_L|^2\right]\right| \lesssim b_1 b_1^{\eta(1 - \delta)}\Es_{2\Bbbk}.
\end{equation}
For the last term in \eqref{id:Aextract}, we use equation \eqref{eq:qys} and the decomposition \eqref{eq:decomF} to write
\begin{align}
\int \Ls^\Bbbk(\ps q - \ps \xi_L)  \Ls^\Bbbk \xi_L &= \left[-\int \Ls^\Bbbk q \Ls^{\Bbbk + 1}\xi_L + \frac{\lambda_s}{\lambda}\int \Lambda q \Ls^{2\Bbbk}\xi_L\right] \nonumber\\
&+ \int \Ls^{\Bbbk}\big[-\tilde \Psi_b - \widetilde{Mod} + \Hc(q) + \Nc(q)\big] \Ls^{\Bbbk}\xi_L.\label{eq:A3exp}
\end{align}
Using \eqref{est:CLxiL}, the admissibility of $T_L$ and the fact that $\Ls^k T_i = 0$ if $i < k$, we estimate 
\begin{align*}
\int |\Ls^{\Bbbk +1}\xi_L|^2 & \lesssim \left|\frac{\left<\Ls^Lq, \chi_{B_0}\Lambda Q\right>}{\left<\Lambda Q, \chi_{B_0}\Lambda Q\right>} \right|^2\int|\Ls^{\Bbbk + 1}\big[(1 - \chi_{B_1})T_L\big]| ^2\\
&\lesssim b_1^{-2(1 - \delta)}\Es_{2\Bbbk}\int_{y \geq B_1}y^{2(2L - \gamma - 2(\Bbbk + 1))} y^{d-1}dy\\
&\quad \lesssim b_1^{-2(1 - \delta)}\Es_{2\Bbbk}b_1^{(4 - 2\delta)(1 + \eta)} \lesssim b_1^2b_1^{2\eta(1 - \delta)}\Es_{2\Bbbk}.
\end{align*}
from which we obtain
\begin{equation*}
\left|\int \Ls^\Bbbk q \Ls^{\Bbbk + 1}\xi_L\right| \lesssim b_1 b_1^{\eta(1 - \delta)}\Es_{2\Bbbk},
\end{equation*}
Similarly, we have the estimate
\begin{align*}
\int (1 + y^{4\Bbbk})|\Ls^{2\Bbbk}\xi_L|^2 &\lesssim b_1^{-2(1 - \delta)}\Es_{2\Bbbk}\int_{y \geq B_1}y^{4\Bbbk}y^{2(2L - \gamma - 4\Bbbk)}y^{d-1}dy\lesssim b_1^{2\eta(1-\delta)}\Es_{2\Bbbk}, 
\end{align*}
hence, using \eqref{est:Es2K1} and \eqref{eq:ODEbkl}, we get
\begin{equation*}
\left|\frac{\lambda_s}{\lambda}\int \Lambda q \Ls^{2\Bbbk}\xi_L \right| \lesssim b_1 \left(\int \frac{|\py q|^2}{1 + y^{4\Bbbk - 2}} \right)^\frac{1}{2}\left(\int (1 + y^{4\Bbbk})|\Ls^{2\Bbbk}\xi_L|^2  \right)^\frac{1}{2} \lesssim b_1 b_1^{\eta(1 - \delta)}\Es_{2\Bbbk}.
\end{equation*}
From \eqref{est:LkxiL2}, \eqref{eq:PsibtilE2k} and \eqref{eq:ModtilE2k}, we have
\begin{align*}
\left|\int \Ls^\Bbbk (\tilde{\Psi}_b + \widetilde{Mod}) \Ls^\Bbbk \xi_L\right| &\lesssim \left(\int |\Ls^\Bbbk \xi_L|^2 \right)^\frac{1}{2}\left(\int|\Ls^\Bbbk (\tilde \Psi_b + \widetilde{Mod})|^2 \right)^\frac{1}{2}\\
&\quad \lesssim b_1b_1^{\eta(1 - \delta)}\Es_{2\Bbbk} + b_1 b_1^{L + (1 - \delta)(1 + \eta)}\sqrt{\Es_{2\Bbbk}}.
\end{align*}
In the same manner, we have the estimate
\begin{align*}
\int (1 + y^{4})|\Ls^{\Bbbk + 1}\xi_L|^2 &\lesssim b_1^{-2(1 - \delta)}\Es_{2\Bbbk}\int_{y \geq B_1}y^{4}y^{2(2L - \gamma - 2(\Bbbk+1))}y^{d-1}dy\lesssim b_1^{2\eta(1-\delta)}\Es_{2\Bbbk}, 
\end{align*}
from which together \eqref{est:HqE2k} and \eqref{est:NqE2k}, we get the bound
\begin{align*}
\left|\int \Ls^{\Bbbk - 1} (\Hc(q) + \Nc(q)) \Ls^{\Bbbk + 1}\xi_L \right|& \lesssim \left(\int \frac{|\Ls^{\Bbbk - 1}(\Hc(q) + \Nc(q))|^2}{1 + y^4} \right)^\frac{1}{2}\left(\int (1 + y^4)|\Ls^{\Bbbk + 1}\xi_L|^2\right)^\frac{1}{2}\\
&\quad \lesssim b_1 b_1^{\eta(1 - \delta)}\Es_{2\Bbbk} + b_1 b_1^{L + (1- \delta)(1 + \eta)}\sqrt{\Es_{2\Bbbk}}.
\end{align*}
Collecting these final bounds into \eqref{eq:A3exp} yields 
\begin{equation}\label{est:A3}
\left|\int \Ls^\Bbbk(\ps q - \ps \xi_L)  \Ls^\Bbbk \xi_L\right| \lesssim b_1 b_1^{\eta(1 - \delta)}\Es_{2\Bbbk} + b_1 b_1^{L + (1- \delta)(1 + \eta)}\sqrt{\Es_{2\Bbbk}}.
\end{equation}

\medskip

\noindent Substituting \eqref{id:Aextract}, \eqref{est:A1}, \eqref{est:A2} and \eqref{est:A3} into \eqref{eq:EnerID4} concludes the proof of \eqref{eq:Es2sLya} as well as Proposition \ref{prop:E2k}.
\end{proof}

\subsection{Conclusion of Proposition \ref{prop:redu}.}
We give the proof of Proposition \ref{prop:redu} in this subsection in order to complete the proof of Theorem \ref{Theo:1}. Note that this section corresponds to Section 6.1 of \cite{MRRcjm15}. Here we follow exactly the same lines as in \cite{MRRcjm15} and no new ideas are needed. We divide the proof into 2 parts:

- \textit{Part 1: Reduction to a finite dimensional problem.} Assume that for a given $K > 0$ large and an initial time $s_0 \geq 1$ large, we have $(b(s), q(s)) \in \Sc_K(s)$ for all $s \in [s_0, s_1]$ for some $s_1 \geq s_0$. By using \eqref{eq:ODEbkl}, \eqref{eq:ODEbLimproved}, \eqref{eq:Es2sLya} and \eqref{eq:Es2sLyam}, we derive new bounds on $\Vc_1(s)$, $b_k(s)$ for $\ell + 1 \leq k \leq L$ and $\Es_{2(\hbar + m)}$ for $1 \leq m \leq L+1$, which are better than those defining $\Sc_K(s)$ (see Definition \ref{def:Skset}). It then remains to control $(\Vc_2(s), \cdots, \Vc_\ell(s))$. This means that the problem is reduced to the control of a finite dimensional function $(\Vc_2(s), \cdots, \Vc_\ell(s))$ and then get the conclusion $(i)$ of Proposition \ref{prop:redu}.

- \textit{Part 2: Transverse crossing.} We aim at proving that if $(\Vc_2(s), \cdots, \Vc_\ell(s))$ touches 
$$\partial \hat \Sc _K(s):= \partial\left(-\frac{K}{s^{\frac{\eta}{2}}(1 - \delta)}, \frac{K}{s^{\frac{\eta}{2}}(1 - \delta)}\right)^{\ell - 1}$$
at $s = s_1$, it actually leaves $\partial \hat \Sc_K(s)$ at $s=s_1$ for $s_1 \geq s_0$, provided that $s_0$ is large enough. We then get the conclusion $(ii)$ of Proposition \ref{prop:redu}.

\paragraph{$\bullet\,$ Reduction to a finite dimensional problem.} We give the proof of item $(i)$ of Proposition \ref{prop:redu} in this part. Given $K > 0$, $s_0 \geq 1$ and the initial data at $s = s_0$ as in Definition \ref{def:1}, we assume for all $s \in [s_0, s_1]$, $(b(s), q(s)) \in \Sc_K(s)$ for some $s_1 \geq s_0$. We claim that for all $s \in [s_0, s_1]$,
\begin{align}
&|\Vc_1(s)| \leq s^{-\frac \eta 2(1- \delta)},\label{est:impV1}\\
&|b_k(s)| \lesssim s^{-(k + \eta(1 - \delta))} \quad \text{for}\;\; \ell + 1 \leq k \leq L,\label{est:impbk}\\
&\Es_{2m} \leq \left\{ \begin{array}{ll}
\frac{K}{2} s^{-\frac{\ell(4m - d)}{2\ell - \gamma}} &\quad \text{for}\quad \hbar + 2 \leq m \leq \ell + \hbar,\\
\frac{1}{2}s^{-2(m - \hbar -1) - 2(1 - \delta) + K\eta} &\quad \text{for}\quad \ell + \hbar + 1 \leq m \leq \Bbbk - 1.
\end{array}\right.\label{est:impE2m}\\
&\Es_{2\Bbbk} \leq \frac{K}{2} s^{-(2L + 2(1 - \delta)(1 + \eta))},\label{est:impE2k}
\end{align}
Once these estimates are proved, it immediately follows from Definition \ref{def:Skset} of $\Sc_K$ that if $(b(s_1), q(s_1)) \in \partial\Sc_K(s_1)$, then $(\Vc_2, \cdots, \Vc_\ell))(s_1)$ must be in $\partial \hat \Sc_K(s_1)$, which concludes the proof of item $(i)$ of Proposition \ref{prop:redu}.\\

Before going to the proof of \eqref{est:impV1}-\eqref{est:impE2k}, let us compute explicitly the scaling parameter $\lambda$. To do so, let us note from \eqref{eq:Ukbke} and the a priori bound on $\Uc_1$ given in Definition \ref{def:Skset} that 
$b_1(s) = \frac{c_1}{s} + \frac{\Uc_1}{s} = \frac{\ell}{(2\ell - \gamma)s} + \Oc\left(\frac{1}{s^{1 + c\eta}}\right).$ 
Using \eqref{eq:ODEbkl} yields
\begin{equation}\label{eq:lam10}
-\frac{\lambda_s}{\lambda} = \frac{\ell}{(2\ell - \gamma)s} + \Oc\left(\frac{1}{s^{1 + c\eta}}\right),
\end{equation}
from which we write
\begin{equation*}
\left|\frac{d}{ds}\left\{\log \Big(s^{\frac{\ell}{2\ell - \gamma}}\lambda(s)\Big)\right\}\right| \lesssim \frac{1}{s^{1 + c\eta}}.
\end{equation*}
We now integrate by using the initial data value $\lambda(s_0) = 1$ to get 
\begin{equation}\label{eq:Lamdas}
\lambda(s) = \left(\frac{s_0}{s}\right)^\frac{\ell}{2\ell - \gamma}\left[1 + \Oc\left(s^{-c\eta}\right)\right] \quad \text{for}\;\; s_0 \gg 1.
\end{equation}
This implies that 
\begin{equation}\label{est:b1so_s}
s_0^{-\frac{\ell}{2\ell - \gamma}} \lesssim \frac{s^{-\frac{\ell}{2\ell - \gamma}} }{\lambda(s)}\lesssim s_0^{-\frac{\ell}{2\ell - \gamma}}.
\end{equation}

\noindent - \textit{Improved control of $\Es_{2\Bbbk}$}: We aim at using \eqref{eq:Es2sLya} to derive the improved bound \eqref{est:impE2k}. To do so, we inject the bound of $\Es_{2\Bbbk}$ given in Definition \ref{def:Skset} into the monotonicity formula \eqref{eq:Es2sLya} and integrate in time by using $\lambda(s_0) = 1$: for all $s \in [s_0, s_1)$,
\begin{align*}
\Es_{2\Bbbk}(s) \leq C \lambda(s)^{4\Bbbk - d}\left[\Es_{2\Bbbk}(s_0) + \left(\frac{K}{M^{2\delta}} + \sqrt K +  1\right)\int_{s_0}^s\frac{\tau^{-(2L + 1 + 2(1 - \delta)(1 + \eta))}}{\lambda(\tau)^{4\Bbbk - d}}d\tau\right].
\end{align*}
Using \eqref{est:b1so_s}, we estimate
\begin{align*}
&\lambda(s)^{4\Bbbk - d}\int_{s_0}^s\frac{\tau^{-(2L + 1 + 2(1 - \delta)(1 + \eta))}}{\lambda(\tau)^{4\Bbbk - d}}d\tau\\
& \quad \lesssim s^{- \frac{\ell(4\Bbbk - d)}{2\ell - \gamma}}\int_{s_0}^s \tau^{\frac{\ell(4\Bbbk - d)}{2\ell - \gamma} - (2L + 1 + 2(1 - \delta)(1 + \eta))}d\tau \lesssim s^{-(2L + 2(1 - \delta)(1 + \eta))}.
\end{align*}
Here we used the fact that the integral is divergent because
$$\frac{\ell(4\Bbbk - d)}{2\ell - \gamma} - [2L + 1 + 2(1 - \delta)(1 + \eta)] = \frac{2\gamma L}{2\ell - \gamma} + \Oc_{L \to +\infty}(1) \gg -1.$$
Using again \eqref{est:b1so_s} and the initial bound \eqref{eq:intialbounE2m}, we estimate
$$\lambda(s)^{4\Bbbk - d}\Es_{2\Bbbk}(s_0) \leq \left(\frac{s_0}{s}\right)^\frac{\ell(4\Bbbk - d)}{2\ell - \gamma} s_0^{-\frac{10L\ell}{2\ell - \gamma}} \lesssim s^{-(2L + 2(1 - \delta)(1 + \eta))},$$
for $L$ large enough. Therefore, we obtain 
$$\Es_{2\Bbbk}(s) \leq C\left(\frac{K}{M^{2\delta}} + \sqrt K + 1\right)s^{-(2L + 2(1 - \delta)(1 + \eta))} \leq \frac{K}{2}s^{-(2L + 2(1 - \delta)(1 + \eta))}, $$
for $K = K(M)$ large enough. This concludes the proof of \eqref{est:impE2k}.\\

\noindent - \textit{Improved control of $\Es_{2m}$.} We can improve the control of $\Es_{2m}$ by using the monotonicity formula \eqref{eq:Es2sLyam}. We distinguish two cases: \\
- \underline{Case 1:} $\hbar +2 \leq m \leq \ell + \hbar$. From the bound of $\Es_{2m}$ given in Definition \ref{def:Skset} and $b_1(s) \sim \frac{1}{s}$, we integrate \eqref{eq:Es2sLyam} in time $s$ by using $\lambda(s_0) = 1$ to find that
\begin{align*}
\Es_{2m}(s) &\leq C \lambda(s)^{4m - d}\left[\Es_{2m}(s_0)  + \sqrt{K}\int_{s_0}^s \frac{ \tau^{-\frac{\ell}{2\ell - \gamma}\left(2m - \frac{d}{2}\right) -(m - \hbar + 1-\delta - C\eta)}}{\lambda(\tau)^{4m - d}}d\tau\right.\\
& \qquad \qquad \qquad \qquad \left. + \int_{s_0}^s \frac{ \tau^{-(2m - 2\hbar - 1 + 2(1 - \delta) - C\eta)}}{\lambda(\tau)^{4m - d}}d\tau\right]
\end{align*}
Using the initial bound \eqref{eq:intialbounE2m} and \eqref{est:b1so_s}, we estimate 
$$C\lambda(s)^{4m - d}\Es_{2m}(s_0)\lesssim s^{-\frac{\ell}{2\ell - \gamma}(4m - d)},$$
for $s_0$ large. \\
Using \eqref{est:b1so_s} and the identity 
\begin{align*}
&\frac{\ell}{2\ell - \gamma}\left(2m - \frac{d}{2}\right) - (m - \hbar + 1 - \delta -C\eta) \\
& \quad = - \frac{\gamma}{2} - 1 + C\eta  + \frac{\gamma}{2\ell - \gamma}(m - \hbar - \delta - \frac{\gamma}{2}) \leq -1 - \frac{\gamma \delta}{2\ell - \gamma} + C\eta < -1,
\end{align*}
we estimate
\begin{align*}
&\lambda(s)^{4m - d}\int_{s_0}^s \frac{ \tau^{-\frac{\ell}{2\ell - \gamma}\left(2m - \frac{d}{2}\right) -(m - \hbar + 1-\delta - C\eta)}}{\lambda(\tau)^{4m - d}}d\tau \\
&\lesssim s^{-\frac{\ell}{2\ell - \gamma}(4m - d)}\int_{s_0}^s \tau^{\frac{\ell}{2\ell - \gamma}\left(2m -\frac d 2\right) - (m - \hbar + 1 - \delta - C\eta)} d\tau \\
& \lesssim s^{-\frac{\ell}{2\ell - \gamma}(4m - d)} \int_{s_0}^s \frac{d\tau}{\tau^{1 + \epsilon}} \lesssim  s^{-\frac{\ell}{2\ell - \gamma}(4m - d)}.
\end{align*}
Similarly, thanks to the identity 
\begin{align*}
&\frac{\ell}{2\ell - \gamma}\left(4m - d\right) - (2m - 2\hbar - 1 +2(1- \delta) -C\eta) \\
& \quad = - \gamma - 1 + C\eta  + \frac{\gamma}{2\ell - \gamma}(2m - 2\hbar - 2\delta - \gamma) \leq -1 - \frac{2\gamma \delta}{2\ell - \gamma} + C\eta < -1,
\end{align*}
we obtain
$$\lambda(s)^{4m - d}\int_{s_0}^s \frac{ \tau^{-(2m - 2\hbar - 1 + 2(1 - \delta) - C\eta)}}{\lambda(\tau)^{4m - d}}d\tau \lesssim s^{-\frac{\ell}{2\ell - \gamma}(4m - d)}.$$
Therefore, we deduce that
$$\Es_{2m}(s) \leq C(1 + \sqrt{K})s^{-\frac{\ell}{2\ell - \gamma}(4m - d)} \leq \frac{K}{2}s^{-\frac{\ell}{2\ell - \gamma}(4m - d)},$$
for $K$ large, which yields the improved bound \eqref{est:impE2m} for $\hbar + 2 \leq m \leq \ell + \hbar$.\\
- \underline{Case 2:} $\ell + \hbar +1 \leq m \leq \Bbbk - 1$. Proceeding as the previous case, we arrive at 
\begin{align*}
\Es_{2m}(s) &\leq C \lambda(s)^{4m - d}\left[\Es_{2m}(s_0) + \int_{s_0}^s \frac{\tau^{-\left[2m - 2\hbar - 1 + 2(1 - \delta) - \left(C + \frac K2\right)\eta\right]}}{\lambda(\tau)^{4m - d}}d\tau\right].
\end{align*}
From the identity
\begin{align}
&\frac{\ell}{2\ell - \gamma}\left(4m - d\right) - (2m - 2\hbar - 1 +2(1- \delta) - \left(C + \frac K2 \right)\eta) \nonumber\\
& \quad = - \gamma - 1 + \left(C + \frac K2 \right)\eta)  + \frac{\gamma}{2\ell - \gamma}(2m - 2\hbar - 2\delta - \gamma)\nonumber\\
& \qquad \geq -1 + \frac{2\gamma (1 -\delta)}{2\ell - \gamma} + \left(C + \frac K2 \right)\eta) > -1,\label{ide2g}
\end{align}
together with \eqref{est:b1so_s}, we estimate 
\begin{align*}
&\lambda(s)^{4m - d}\int_{s_0}^s \frac{\tau^{-\left[2m - 2\hbar - 1 + 2(1 - \delta) - \left(C + \frac K2\right)\eta\right]}}{\lambda(\tau)^{4m - d}}d\tau\\
& \quad \lesssim s^{-\frac{\ell(4m - d)}{2\ell - \gamma}}\int_{s_0}^s\tau^{\frac{\ell(4m - d)}{2\ell - \gamma} -\left[2m - 2\hbar - 1 + 2(1 - \delta) - \left(C + \frac K2\right)\eta\right]} d\tau\\
& \qquad \lesssim s^{-\left[2(m - \hbar - 1) + 2(1 - \delta) - \left(C + \frac K2\right)\eta\right]}\leq  \frac 14 s^{-\left[2(m - \hbar - 1) + 2(1 - \delta) - K\eta\right]}.
\end{align*}
Using \eqref{ide2g}, \eqref{est:b1so_s} and the initial bound \eqref{eq:intialbounE2m}, we derive 
\begin{align*}
C\lambda(s)^{4m - d}\Es_{2m}(s_0) \lesssim s^{-\frac{\ell(4m - d)}{2\ell - \gamma}} & \lesssim s^{-\left[2(m - \hbar - 1) + 2(1 - \delta) - \left(C + \frac K2\right)\eta\right]} \\
&\quad \leq  \frac 14 s^{-\left[2(m - \hbar - 1) + 2(1 - \delta) - K\eta\right]}.
\end{align*}
This concludes the proof of \eqref{est:impE2m}. \\

\noindent - \textit{Control of the stable modes $b_k$'s.} We now end the control of the stable modes $(b_{\ell + 1}, \cdots, b_L)$, in particular, we prove \eqref{est:impbk}. We first treat the case when $k = L$. Let 
$$\tilde{b}_L = b_L + \frac{\left<\Ls^Lq,\chi_{B_0}\Lambda Q \right>}{\left<\Lambda Q, \chi_{B_0}\Lambda Q \right>},$$
then from \eqref{est:CLxiL} and \eqref{est:impE2k}, 
$$|\tilde b_L - b_L| \lesssim b_1^{-(1 - \delta)}\sqrt{\Es_{2\Bbbk}} \lesssim b_1^{L + \eta(1 - \delta)}.$$
It follows from the improved modulation equation \eqref{eq:ODEbLimproved}, 
\begin{align*}
|(\tilde{b}_L)_s + (2L - \gamma)b_1\tilde{b}_L|&\lesssim b_1|\tilde b_L - b_L| + \frac{1}{B_0^{2\delta}}\left[C(M)\sqrt{\Es_{2\Bbbk}} + b_1^{L +(1 - \delta)}\right]\\
&\quad \lesssim b_1^{L + 1 + \eta(1 -\delta)}.
\end{align*}
This follows
$$\left|\frac{d}{ds} \left\{\frac{\tilde b_L}{\lambda^{2L - \gamma}} \right\}\right| \lesssim \frac{b_1^{L + 1 + \eta(1 - \delta)}}{\lambda^{2L - \gamma}}.$$
Integrating this identity in time from $s_0$ and recalling that $\lambda(s_0) = 1$ yields
\begin{align*}
\tilde{b}_L(s) \lesssim C\lambda(s)^{2L - \gamma}\left(\tilde{b}_L(s_0) + \int_{s_0}^s\frac{b_1(\tau)^{L + 1 + \eta(1 - \delta)}}{\lambda(\tau)^{2L - \gamma}} d\tau\right).
\end{align*}
Using \eqref{est:CLxiL}, $b_1(s) \sim \frac{1}{s}$, the initial bounds \eqref{eq:initbk} and \eqref{eq:intialbounE2m} together with \eqref{est:b1so_s}, we estimate
$$\lambda(s)^{2L - \gamma}\tilde{b}_L(s_0) \lesssim \left(\frac{s_0}{s}\right)^{\frac{\ell(2L - \gamma)}{2\ell - \gamma}}\left(s_0^{-\frac{5\ell(2L - \gamma)}{2\ell - \gamma}} + s_0^{\eta(1 - \delta)}s_0^{-\frac{5L\ell}{2\ell - \gamma}} \right) \lesssim s^{- L - \eta(1 - \delta)},$$
and 
\begin{align*}
\lambda(s)^{2L - \gamma}\int_{s_0}^s\frac{b_1(\tau)^{L + 1 + \eta(1 - \delta)}}{\lambda(\tau)^{2L - \gamma}} d\tau &\lesssim s^{-\frac{\ell(2L - \gamma)}{2\ell - \gamma}}\int_{s_0}^s \tau ^{ \frac{\ell(2L - \gamma)}{2\ell - \gamma} - L - 1 - \eta(1 - \delta)}d\tau\\
&\quad \lesssim s^{-L - \eta(1 - \delta)}.
\end{align*}
Therefore, 
$$b_L(s) \lesssim |\tilde{b}_L(s)| + |\tilde{b}_L(s) - b_L(s)| \lesssim s^{-L - \eta(1 - \delta)},$$
which concludes the proof of \eqref{est:impbk} for $k = L$.
Now we will propagate this improvement that we found for the bound of $b_L$ to all $b_k$ for all $\ell + 1 \leq k \leq L - 1$. To do so we do a descending induction where the initialization is for $k=L$. Let assume  the bound 
$$|b_k|\lesssim  b_1^{k + \eta(1 - \delta)},$$
for $k+1$ and let's prove it for $k$. 
Indeed, from \eqref{eq:ODEbkl} and the induction bound, we have 
$$\left|(b_k)_s - (2k - \gamma)\frac{\lambda_s}{\lambda}b_k\right| \lesssim b_1^{L + 1} + |b_{k + 1}| \lesssim b_1^{k + 1 + \eta(1 - \delta)},$$
which follows
$$\left|\frac{d}{ds}\left\{\frac{b_k}{\lambda^{2k - \gamma}}\right\} \right| \lesssim \frac{b_1^{k + 1 + \eta(1 - \delta)}}{\lambda^{2k - \gamma}}.$$
Integrating this identity in time as for the case $k = L$, we end-up with 
\begin{align*}
b_k(s) &\lesssim C\lambda(s)^{2k - \gamma}\left(b_k(s_0) + \int_{s_0}^s\frac{b_1(\tau)^{k + 1 + \eta(1 - \delta)}}{\lambda(\tau)^{2k - \gamma}} d\tau\right)\\
&\quad \lesssim s^{-k - \eta(1 - \delta)},
\end{align*}
where we used the initial bound \eqref{eq:initbk}, \eqref{est:b1so_s} and $k \geq \ell + 1$. This concludes the proof of \eqref{est:impbk}.

\noindent - \textit{Control of the stable mode $\Vc_1$.} We recall from \eqref{eq:Ukbke} and \eqref{def:VctoUc} that 
$$b_k = b_k^e + \frac{\Uc_k}{s^k}, \quad 1 \leq k \leq \ell, \quad \Vc = P_\ell \Uc,$$
where $P_\ell$ diagonalize the matrix $A_\ell$ with spectrum \eqref{eq:diagAlPl}. From \eqref{eq:bkk1}, and \eqref{eq:ODEbkl}, we estimate for $1 \leq k \leq \ell - 1$, 
$$|s(\Uc_k)_s - (A_\ell \Uc)_k| \lesssim s^{k + 1}|(b_k)_s + (2k - \gamma)b_1 b_k - b_{k + 1}| + |\Uc|^2 \lesssim s^{-L + k} + |\Uc|^2,$$
From \eqref{eq:bell}, \eqref{eq:ODEbkl} and the improved bound \eqref{est:impbk}, we have 
$$|s(\Uc_\ell)_s - (A_\ell \Uc)_\ell| \lesssim s^{\ell + 1}\left(|(b_k)_s + (2k - \gamma)b_1 b_\ell - b_{\ell + 1}| + |b_{\ell + 1}|\right) + |\Uc|^2 \lesssim s^{-\eta(1 - \delta)} + |\Uc|^2.$$
Using the diagonalization \eqref{eq:diagAlPl}, we obtain 
\begin{equation}\label{eq:sVs}
s\Vc_s = D_\ell \Vc + \Oc(s^{-\eta(1 - \delta)}).
\end{equation}
s
Using \eqref{eq:diagAlPl} again yields the control of the stable mode $\Vc_1$:
$$|(s\Vc_1)_s| \lesssim s^{-\eta(1 - \delta)}.$$
Thus from the initial bound \eqref{eq:initbk}, 
$$|s^{\eta(1 - \delta)}\Vc_1(s)| \leq \left(\frac{s_0}{s}\right)^{1 - \eta(1 - \delta)}s_0^{\eta(1 - \eta)}\Vc_1(s_0) + 1\lesssim s_0^{\eta(1 - \delta)},$$
which yields \eqref{est:impV1} for $s_0 \geq s_0(\eta)$ large enough.\\

\paragraph{$\bullet \,$ Transverse crossing.} We give the proof of item $(ii)$ of Proposition \ref{prop:redu} in this part. We compute from \eqref{eq:sVs} and \eqref{eq:diagAlPl} at the exit time $s = s_1$:
\begin{align*}
&\frac{1}{2}\frac{d}{ds}\left(\sum_{k = 2}^\ell|s^{\frac{\eta}{2}(1 - \delta)}\Vc_k(s)|^2 \right)_{\big|_{s= s_1}} \\
&\quad = \left(s^{\eta(1 - \delta) - 1} \sum_{k = 2}^\ell \left[\frac{\eta}{2}(1 - \delta)\Vc_k^2(s) + s\Vc_k(\Vc_k)_s \right]\right)_{\big|_{s= s_1}}\\
&\qquad = \left(s^{\eta(1 - \delta) - 1} \Bigg[\sum_{k = 2}^\ell \left[\frac{k\gamma}{2k - \gamma} + \frac{\eta}{2}(1 - \delta)\right]\Vc_k^2(s) + \Oc\left(\frac{1}{s^{\frac{3}{2}\eta(1 - \delta)}}\right)\Bigg] \right)_{\big|_{s= s_1}}\\
&\quad \qquad \geq \frac{1}{s_1}\left[c(d,\ell) \sum_{k = 2}^\ell |s_1^{\frac{\eta}{2}(1 - \delta)}\Vc_k(s_1)|^2 +  \Oc\left(\frac{1}{s_1^{\frac{\eta}{2}(1 - \delta)}}\right)\right]\\
& \qquad \qquad \geq \frac{1}{s_1}\left[c(d,\ell) +  \Oc\left(\frac{1}{s_1^{\frac{\eta}{2}(1 - \delta)}}\right)\right] > 0,
\end{align*}
where we used item $(i)$ of Proposition \ref{prop:redu} in the last step. This completes the proof of Proposition \ref{prop:redu}.

\appendix
\section{Coercivity of the adapted norms.}
We give in this section the coercivity estimates for the operator $\Ls$ as well as the iterates of $\Ls$  under some suitable orthogonality condition. We first recall the standard Hardy type inequalities for the class of radially symmetric functions, 
$$\Dc_{rad} = \{f \in \Cc_c^\infty(\Rd) \; \text{with radial symmetry}\}.$$
For simplicity, we write
$$\int f := \int_0^{+\infty} f(y)y^{d-1}dy.$$
and
$$D^k = \left\{\begin{array}{ll} \Delta^m &\quad \text{if}\;\; k = 2m,\\
 \partial_y \Delta^m &\quad \text{if} \;\; k= 2m + 1.
\end{array}\right.$$

We have the following:
\begin{lemma}[Hardy type inequalities]\label{lemm:Hardy} Let $d \geq 7$ and $f \in \Dc_{rad}$, then\\
$(i)$ (Hardy near the origin)
\begin{equation*}
\int_{0}^1 \frac{|\py f|^2}{y^{2i}} \geq \frac{(d - 2 - 2i)^2}{4}\int_{0}^1 \frac{f^2}{y^{2+2i}} - C(d)f^2(1), \quad i = 0, 1, 2.
\end{equation*}
$(ii)$ (Hardy away from zero for the non-critical exponent) Let $\alpha > 0, \alpha \ne \frac{d-2}{2}$, then
\begin{align*}
&\int_1^{+\infty} \frac{|\py f|^2}{y^{2\alpha}} \geq \left(\frac{d - (2\alpha + 2)}{2}\right)^2\int_1^{+\infty} \frac{f^2}{y^{2 + 2\alpha}} - C(\alpha,d)f^2(1).
\end{align*}
$(iii)$ (Hardy away from zero for the critical exponent) Let $\alpha = \frac{d-2}{2}$, then
$$
\int_1^{+\infty} \frac{|\py f|^2}{y^{2\alpha}} \geq \frac 14 \int_1^{+\infty} \frac{f^2}{y^{2 + 2\alpha} (1 + \log y)^2} - C(d)f^2(1).$$
$(iv)$ (General weighted Hardy) For any $\mu > 0$, $k \geq 2$ be an integer and $1 \leq j \leq k-1$,
$$\int \frac{|D^jf|^2}{1 + y^{\mu + 2(k - j)}} \lesssim_{j,\mu} \int \frac{|D^k f|^2}{1 + y^\mu} + \int \frac{f^2}{1 + y^{\mu + 2k}}.$$
\end{lemma}
\begin{proof} The proof can be found in \cite{MRRcjm15}, Lemma B.1.
\end{proof}

From the Hardy type inequalities, we derive the following coercivity of $\As^*$: 
\begin{lemma}[Weight coercivity of $\As^*$]\label{lemm:coerAst} Let $\alpha \geq 0$, there exists $c_\alpha > 0$ such that for all $f \in \Dc_{rad}$:
\begin{equation}\label{eq:coerAst}
\int \frac{|\As^* f|^2}{y^{2i}(1 + y^{2\alpha})} \geq c_\alpha\left(\int \frac{|\py f|^2}{y^{2i}(1 + y^{2\alpha})} + \int \frac{f^2}{y^{2i + 2}(1 + y^{2\alpha})}\right), \quad i = 0, 1,2.
\end{equation}
\end{lemma}
\begin{proof} We proceed into two steps:

\noindent \textit{- Step 1: Subcoercive estimate for $\As^*$:} We first prove the following subcoercive bound for $\As^*$: for $i = 0, 1, 2$ and $\alpha \geq 0$, 
\begin{align}
\int\frac{|\As^* f|^2}{y^{2i}(1 + y^{2\alpha})} \gtrsim \int\frac{f^2}{y^{2i + 2}(1 + y^{2\alpha})} &+ \int \frac{|\py f|^2}{y^{2i}(1 + y^{2\alpha})}\nonumber\\
 &\quad - f^2(1) - \int \frac{f^2}{1 + y^{2i + 2\alpha + 4}}.\label{eq:subcoerAst}
\end{align}
From the definition \eqref{def:Astar} of $\As^*$ and the asymptotic of $V$ given in \eqref{eq:asympV}, we use an integration by parts to estimate near the origin:
\begin{align*}
\int_{y \leq 1}\frac{|\As^* f|^2}{y^{2i}(1 + y^{2\alpha})} &\gtrsim \int_{y \leq 1}\frac{1}{y^{2i}}\left|\py f + \frac{d}{y}f + \Oc(|yf|)\right|^2\\
&\quad \gtrsim \int_{y \leq 1}\frac{|\py f|^2}{y^{2i}} + d\int_{y \leq 1}\frac{\py (f^2)}{y^{2i + 1}} + d^2\int_{y \leq 1}\frac{f^2}{y^{2i + 2}} + \Oc\left(\int_{y \leq 1}\frac{f^2}{y^{2i - 2}}\right)\\
&\qquad \gtrsim \int_{y\leq 1} \frac{|\py f|^2}{y^{2i}} + (2 + 2i)d\int_{y\leq 1}\frac{f^2}{y^{2i +2}} + df^2(1)+ \Oc\left(\int_{y\leq 1} \frac{f^2}{y^{2i - 2}}\right)\\
& \quad \qquad \gtrsim \int_{y \leq 1}\left(\frac{|\py f|^2}{y^{2i}} + \frac{f^2}{y^{2i + 2}}\right) - \int_{y \leq 1}y^2f^2.
\end{align*}
Away from the origin, we use \eqref{eq:asympV} to estimate:
\begin{align*}
\int_{y \geq 1}\frac{|\As^* f|^2}{y^{2i}(1 + y^{2\alpha})} &\gtrsim \int_{y \geq 1} \frac{1}{y^{2i + 2\alpha}}\left(\py f + \frac{d - 1 - \gamma}{y}f \right)^2 - \int_{y\geq 1}\frac{f^2}{y^{2i + 2\alpha + 4}}.
\end{align*}
We make the change of variable $g = y^{d - 1 - \gamma} f$ and use the Hardy inequality given in part $(ii)$ of Lemma \ref{lemm:Hardy} to write
\begin{align*}
\int_{y \geq 1}\frac{|\py(y^{d - 1 - \gamma} f)|^2}{y^{2i + 2\alpha  +2(d - 1 - \gamma)}}dy &= \int_{y \geq 1}\frac{|\py g|^2}{y^{2i + 2\alpha + 2(d-1-\gamma)}}dy \gtrsim \int_{y \geq 1}\frac{g^2}{y^{2i + 2\alpha + 2(d-1-\gamma) + 2}}dy - g^2(1)\\
&\quad \gtrsim \int_{y \geq 1}\frac{f^2}{y^{2i + 2\alpha + 2}} - f^2(1).
\end{align*}
Gathering the above bounds together with the trivial bound from \eqref{eq:asympV},
$$\int_{y \geq 1} \frac{|\py f|^2}{y^{2i + 2\alpha}} \lesssim \int_{y \geq 1} \frac{|\As^* f|^2}{y^{2i + 2\alpha}} + \int_{y \geq 1}\frac{f^2}{y^{2i + 2\alpha + 2}},$$
yields the subcoercivity \eqref{eq:subcoerAst}.\\

\noindent \textit{- Step 2: Coercivity of $\As^*$:} We now argue by contradiction to show the coercivity of $\As^*$. Assume that \eqref{eq:coerAst} does not hold. Up to a renormalization, we consider the sequence $f_n \in \Dc_{rad}$ with 
\begin{equation}\label{eq:asumpfn}
\int \frac{f_n^2}{y^{2 i +2}(1 + y^{2\alpha})} + \int \frac{|\partial_y f_n|^2}{y^{2i}(1 + y^{2\alpha})} = 1 \quad \text{and} \quad \int \frac{|\As^* f_n|^2}{y^{2i}(1 + y^{2\alpha})} \leq \frac{1}{n}.
\end{equation}
This implies from \eqref{eq:subcoerAst},
\begin{equation}\label{eq:boundfn}
f_n^2(1) + \int \frac{f_n^2}{1 + y^{2i + 2\alpha + 4}} \gtrsim 1.
\end{equation}
From \eqref{eq:asumpfn}, the sequence $f_n$ is bounded in $H^1_{loc}$. Hence, from a standard diagonal extraction argument, there exists $f_\infty \in H^1_{loc}$ such that up to a subsequence, 
$$f_n \rightharpoonup f_\infty \quad \text{in} \quad H^1_{loc},$$
and from the local compactness of one dimensional Sobolev embeddings:
$$f_n \rightarrow f_\infty \quad \text{in}\quad L^2_{loc}, \quad f_n(1) \rightarrow f_\infty(1).$$
This implies from \eqref{eq:asumpfn} and \eqref{eq:boundfn},
\begin{equation}\label{eq:regufinA}
f_\infty^2(1) + \int \frac{f_\infty^2}{1 + y^{2i + 2\alpha + 4}} \gtrsim 1 \quad \text{and} \quad \int \frac{f_\infty^2}{y^{2i + 2}(1 + y^{2\alpha})} \lesssim 1,
\end{equation}
which means that $f_\infty \ne 0$. 
On the other hand, from \eqref{eq:asumpfn} and the lower semi continuity of norms for the weak topology, we have 
$$\As^* f_\infty = 0, $$
hence, 
$$f_\infty = \frac{\beta}{y^{d-1} \Lambda Q} \quad \text{for some $\beta \ne 0$}.$$
Since $\Lambda Q \sim y$ near the origin, we have 
$$\int_{y \leq 1} \frac{f_\infty^2}{y^{2i + 2}} \gtrsim \int_{y \leq 1}\frac{y^{d-1}}{y^{2d + 2i + 2}}dy = \int_{y \leq 1}\frac{dy}{y^{d+2i + 3}} = +\infty,$$
which contradicts the a priori regularity of $f_\infty$ given in \eqref{eq:regufinA}. This concludes the proof of Lemma \ref{lemm:coerAst}.
\end{proof}

We also need the following subcoercivity of $\As$. 
\begin{lemma}[Weight coercivity of $\As$] \label{lemm:subcoerA} Let $p \geq 0$ and $i = 0, 1, 2$ such that $|2p + 2i - (d - 2 - 2\gamma)| \ne 0$, where $\gamma \in (1,2]$ is defined by \eqref{def:gamome}, there holds:
\begin{align}
\int \frac{|\As f|^2}{y^{2i}(1 + y^{2p})} \gtrsim \int \frac{|\py f|^2}{y^{2i}(1 + y^{2p})} & + \int \frac{f^2}{y^{2i + 2}(1 + y^{2p})} \nonumber\\
&\quad - \left[f^2(1) + \int \frac{f^2}{1 +y^{2i + 2p + 4}} \right].\label{eq:subcoerA}
\end{align}
Assume in addition that 
$$\left<f, \Phi_M\right> = 0 \quad \text{if}\quad 2i + 2p > d - 2\gamma - 2,$$
where $\Phi_M$ is defined in \eqref{def:PhiM}, we have 
\begin{equation}\label{eq:coerA}
\int \frac{|\As f|^2}{y^{2i}(1 + y^{2p})} \gtrsim \int \frac{|\py f|^2}{y^{2i}(1 + y^{2p})} + \int \frac{f^2}{y^{2i + 2}(1 + y^{2p})}.
\end{equation}
\end{lemma}
\begin{proof} The proof is very similar to the proof of Lemma \ref{lemm:coerAst}. We proceed into two steps. The first step is to derive the subcoercive estimate \eqref{eq:subcoerA}. In the second step, we use a compactness argument to show the coercivity of $\As$ under a suitable condition.

\noindent \textit{- Step 1: Subcoercive estimate for $\As$.} From the definition \eqref{def:As} of $\As$ and the asymptotic of $V$ given in \eqref{eq:asympV}, we estimate near the origin
\begin{align*}
\int_{y \leq 1} \frac{|\As f|^2}{y^{2i}(1 + y^{2p})} &\gtrsim \int_{y \geq 1} \frac 1 {y^{2i}}\left|-\py f + \frac{f}{y} + \Oc(|yf|)\right|^2\\
& \quad \gtrsim \int_{y \leq 1} \frac{|\py f|^2}{y^{2i}} + \int_{y \leq 1} \frac{f^2}{y^{2i+2}} - \int_{y \leq 1}\frac{\py(f^2)}{y^{2i+1}} - \int_{y\leq 1}\frac{f^2}{y^{2i-2}}\\
& \qquad \gtrsim \int_{y \leq 1} \frac{|\py f|^2}{y^{2i}} + (d-2i-1)\int_{y \leq 1} \frac{f^2}{y^{2i+2}} - f^2(1) - \int_{y\leq 1}\frac{f^2}{y^{2i-2}}\\
& \qquad \quad \gtrsim \int_{y \leq 1} \frac{|\py f|^2}{y^{2i}} + \int_{y \leq 1} \frac{f^2}{y^{2i+2}} - f^2(1) - \int_{y\leq 1}y^2f^2.
\end{align*}
Away from the origin, we estimate from \eqref{eq:asympV}:
$$\int_{y \geq 1} \frac{|\As f|^2}{y^{2i}(1 + y^{2p})} \gtrsim \int_{y \geq 1}\frac{1}{y^{2i + 2p }}\left(\py f + \frac{\gamma}{y}f\right)^2 - \int_{y \geq 1}\frac{f^2}{y^{2i + 2p + 4}}.$$
We make the change of variable $g = y^\gamma f$. From the assumption $|2i + 2p - (d - 2 - 2\gamma)| \ne 0$, we use the Hardy inequality given in part $(ii)$ of Lemma \ref{lemm:Hardy} to write
\begin{align*}
\int_{y \geq 1}\frac{|\partial_y(y^\gamma f)|^2}{y^{2i + 2p  + 2\gamma}} = \int_{y \geq 1}\frac{|\py g|^2}{y^{2i + 2p+ 2\gamma}} &\gtrsim \int_{y \geq 1}\frac{g^2}{y^{2i + 2p + 2 + 2\gamma}} - g^2(1)\\
&\quad \gtrsim \int_{y \geq 1}\frac{f^2}{y^{2i + 2p + 2}} - f^2(1).
\end{align*}
Note also that we have the trivial bound from \eqref{eq:asympV},
$$\int_{y \geq 1} \frac{|\As f|^2}{y^{2i + 2p}} + \int_{y \geq 1}\frac{f^2}{y^{2i + 2p + 2}} \gtrsim \int_{y \geq 1} \frac{|\py f|^2}{y^{2i + 2p}}.$$
The collection of the above bounds yields the subcoercivity \eqref{eq:subcoerA}.\\

\noindent \textit{- Step 2: Coercivity of $\As$:} Arguing as the proof of \eqref{eq:coerAst}, we end up with the existence of $f_\infty \ne 0$ such that
$$ \int \frac{f_\infty^2}{y^{2i + 2}(1 + y^{2p})} \lesssim 1 \quad \text{and} \quad \As f_\infty = 0.$$
Hence, from the definition \eqref{def:As} of $\As$, we have
$$f_\infty = \beta \Lambda Q \quad \text{for some $\beta \ne 0$.}$$
If $2i + 2p > d - 2\gamma - 2$, we use the orthogonality condition to deduce that
$$0 = \left<f_\infty, \Phi_M\right> = \beta\left<\Lambda Q, \chi_M \Lambda Q\right>,$$
thus, $\beta = 0$. If $2i + 2p \leq d - 2\gamma - 2$, we use the fact that $\Lambda Q \sim \frac{1}{y^\gamma}$ as $y \to +\infty$ to estimate
$$\int_{y \geq 1}\frac{|\Lambda Q|^2 y^{d-1}dy}{y^{2i + 2}(1 + y^{2p})} \gtrsim \int_{y \geq 1} y^{d - 1 - 2\gamma - 2i - 2p - 2} dy  \gtrsim \int_{y \geq 1} y^{-1}dy= +\infty,$$
which contradicts with the regularity of $f_\infty$. This concludes the proof of Lemma \ref{lemm:subcoerA}.
\end{proof}

From the coercivities of $\As$ and $\As^*$, we claim the following coercivity for $\Ls$:
\begin{lemma}[Weighted coercivity of $\Ls$ under a suitable orthogonality condition] \label{lemm:coerL}  Let $k \in \mathbb{N}$, $i = 0, 1, 2$ and $M = M(k)$ large enough, then there exists $c_{M,k} > 0$ such that for all $f \in \Dc_{rad}$ satisfying the orthogonality 
$$\left<f, \Phi_M\right> = 0 \quad \text{if}\;\; 2i + 2k > d - 2\gamma - 4,$$ 
where $\Phi_M$ is defined by \eqref{def:PhiM}, $\hbar$ is given in \eqref{def:kdeltaplus}, there holds:

\begin{equation}\label{eq:coerL}
\int \frac{|{\Ls} f|^2}{y^{2i}(1 + y^{2k})} \geq c_{M,k} \int\left( \frac{|\partial_{yy}f|^2}{y^{2i}(1 + y^{2k})} +  \frac{|\py f|^2}{y^{2i}(1 + y^{2k + 2})} +  \frac{|f|^2}{y^{2i + 2}(1 + y^{2k + 2})}\right).
\end{equation}
and 
\begin{equation}\label{eq:coerL1}
\int \frac{|{\Ls} f|^2}{y^{2i}(1 + y^{2k})} \geq c_{M,k} \int \left(\frac{|\As f|^2}{y^{2i + 2}(1 + y^{2k})} + \int \frac{|f|^2}{y^{2i}(1 + y^{2k + 4})} \right).
\end{equation}
\end{lemma}
\begin{proof} We proceed into two steps: \\
\noindent - \textit{Step 1: Subcoercivity of $\Ls$.}  We apply Lemma \ref{lemm:coerAst} to $\As f$ with $\alpha = k$ and note that 
$$\py (\As f) = \As(\py f) + \py \left(\frac Vy\right) f,$$
to write
\begin{align}
\int \frac{|{\Ls} f|^2}{y^{2i}(1 + y^{2k})} &\gtrsim \int \frac{|\As f|^2}{y^{2i + 2}(1 + y^{2k})} + \int \frac{|\partial_y (\As f)|^2}{y^{2i}(1 + y^{2k})}\label{eq:tmpcoerLA}\\
& \quad \gtrsim \int \frac{|\As f|^2}{y^{2i}(1 + y^{2k + 2})} + \int \frac{|\partial_y (\As f)|^2}{y^{2i}(1 + y^{2k})}\nonumber\\
&\qquad \gtrsim \int \frac{|\As f|^2}{y^{2i}(1 + y^{2k + 2})} + \int \frac{|\As (\py f)|^2}{y^{2i}(1 + y^{2k})} - \int \frac{|f|^2}{y^{2i + 2}(1 + y^{2k})}.\label{eq:subcoerLax1}
\end{align}
Applying Lemma \ref{lemm:subcoerA} to $f$ with $p = k + 1$ and noting that the condition $|2(k + 1) + 2i - (d - 2 - 2\gamma)| \neq 0$ is always satisfied (if not, we have $d = 4 + 2\sqrt{(k + 1 + i)^2 + 2} \not \in \mathbb{N}$), we have 
\begin{align*}
\int \frac{|\As f|^2}{y^{2i}(1 + y^{2k  +2})} \gtrsim \int \frac{|\py f|^2}{y^{2i}(1 + y^{2k  +2})} & + \int \frac{f^2}{y^{2i+2}(1 + y^{2k + 2})}\\
&\quad - \left[f^2(1) + \int \frac{f^2}{1 +y^{2k + 2i + 6}} \right].
\end{align*}
We apply again Lemma \ref{lemm:subcoerA} to $\py f$ with $p = k$ to estimate
\begin{align*}
\int \frac{|\As (\py f)|^2}{y^{2i}(1 + y^{2k})} \gtrsim \int \frac{|\partial_{yy} f|^2}{y^{2i}(1 + y^{2k })} & + \int \frac{|\py f|^2}{y^{2i + 2}(1 + y^{2k})}\\
&\quad - \left[|\py f(1)|^2 + \int \frac{|\py f|^2}{1 +y^{2k + 2i + 4}} \right].
\end{align*}
Injecting these bounds into \eqref{eq:subcoerLax1} yields the subcoercive estimate for $\Ls$,
\begin{align}
\int \frac{|{\Ls} f|^2}{y^{2i}(1 + y^{2k})} &\gtrsim \int \frac{|\partial_{yy}f|^2}{y^{2i}(1 + y^{2k})} + \int \frac{|\py f|^2}{y^{2i}(1 + y^{2k + 2})} + \int \frac{f^2}{y^{2i + 2}(1 + y^{2k + 2})}\nonumber\\
& \quad - \left[ f^2(1) + |f_y(1)|^2 + \int \frac{|f_y|^2}{1 + y^{2k + 2i + 4}} + \int \frac{f^2}{1 + y^{2k + 2i + 6}} \right]. \label{eq:subcoerL}
\end{align}

\noindent - \textit{Step 2: Coercivity of $\Ls$.} We argue by contradiction. Assume that \eqref{eq:coerL} does not hold. Up to a renormalization, there exists a a sequence of functions $f_n \in \Dc_{rad}$ such that 
\begin{equation}\label{eq:asumpfnL}
\int \frac{|\Ls f_n|^2}{y^{2i}(1 + y^{2k})} \leq \frac{1}{n}, \quad \int \frac{|\partial_{yy}f_n|^2}{y^{2i}(1 + y^{2k})} + \int \frac{|\partial_y f_n|^2}{y^{2i}(1 + y^{2k + 2})} + \int \frac{|f_n|^2}{y^{2i + 2}(1 + y^{2k + 2})} = 1.
\end{equation}
This implies from \eqref{eq:subcoerL},
\begin{equation}\label{eq:boundfnL}
f_n^2(1) + |\partial_y f_n(1)|^2 + \int \frac{|\partial_y f_n|^2}{1 + y^{2k + 2i + 4}} + \int \frac{f_n^2}{y^2(1 + y^{2k + 2i + 6})} \gtrsim 1.
\end{equation}
From \eqref{eq:asumpfnL}, the sequence $f_n$ is bounded in $H^2_{loc}$. Hence, from a standard diagonal extraction argument, there exists $f_\infty \in H^2_{loc}$ such that up to a subsequence, 
$$f_n \rightharpoonup f_\infty \quad \text{in} \quad H^2_{loc},$$
and from the local compactness of one dimensional Sobolev embeddings:
$$f_n \rightarrow f_\infty \quad \text{in}\quad H^1_{loc},$$
and
$$f_n(1) \rightarrow f_\infty(1),\quad  \partial_y f_n(1) \rightarrow \partial_y f_\infty(1).$$
This implies from \eqref{eq:asumpfnL} and \eqref{eq:boundfnL},
$$f_\infty^2(1) + |\partial_y f_\infty(1)|^2 + \int \frac{|\partial_y f_\infty|^2}{1 + y^{2k + 2i + 4}} + \int \frac{f_\infty^2}{y^2(1 + y^{2k + 2i + 6})} \gtrsim 1,$$
which means that $f_\infty \ne 0$.  
On the other hand, from \eqref{eq:asumpfnL} and the lower semi continuity of norms for the weak topology, we deduce that $f_\infty$ is a non trivial function in the kernel of $\Ls$, namely that
$$ \Ls f_\infty = 0, $$
which follows
$$f_\infty =  \mu\Gamma + \beta \Lambda Q,$$
where $\mu$ and $\beta$ two real numbers. \\
From \eqref{eq:asumpfnL} and the lower semicontinuity, we have 
$$\int \frac{f_\infty^2}{y^{2 i + 2}(1 + y^{2k+2})} < +\infty.$$
Recall from \eqref{eq:asymGamma} that $\Gamma \sim \frac{1}{y^{d-1}}$ as $y \to 0$. This yields the estimate
$$\int_{y \leq 1} \frac{\Gamma^2}{y^{2i + 2}(1 + y^{2k + 2})} \gtrsim \int_{y \leq 1} \frac{dy}{y^{2i + 2 + d - 1}} = +\infty,$$
hence, $\mu = 0$.

From \eqref{eq:asymLamQ}, we have $\Lambda Q \sim \frac{1}{y^{\gamma}}$ as $y \to +\infty$. If $2i + 2k \leq d - 2\gamma - 4$, we have
$$\int_{y \geq 1} \frac{|\Lambda Q|^2 y^{d-1}dy }{y^{2i + 2}(1 + y^{2k + 2})} \gtrsim \int_{y \geq 1} y^{d - 1 - 2i - 2k - 4 - 2\gamma} dy \gtrsim \int_{y \geq 1} y^{-1}dy = + \infty,$$
hence, $\beta = 0$. If $2i + 2k > d - 2\gamma - 4$, we use the orthogonality condition to deduce
$$0= \left<f_\infty, \Phi_M\right> = \beta\left<\Lambda Q, \chi_M \Lambda Q\right>,$$
which yields $\beta = 0$, hence $f_\infty = 0$. The contradiction then follows and the coercivity \eqref{eq:coerL} is proved. The estimate \eqref{eq:coerL1} simply follows from \eqref{eq:coerL} and \eqref{eq:tmpcoerLA}. This finished the proof of Lemma \ref{lemm:coerL}.
\end{proof}

We are now in a position to  prove the coercivity of ${\Ls}^k$ under a suitable orthogonality condition.  We claim the following:
\begin{lemma}[Coercivity of the iterate of $\Ls$] \label{lemm:coeLk}  Let $k \in \mathbb{N}$ and $M = M(k)$ large enough, then there exists $c_{M,k} > 0$ such that for all $f \in \Dc_{rad}$ satisfying the orthogonal condition 
$$\left<f, \Ls^m \Phi_M\right> = 0, \quad 0 \leq m \leq k - \hbar,$$
where $\hbar$ is defined as in \eqref{def:kdeltaplus}, there holds:
\begin{align} \label{eq:coerLk}
\Es_{2k + 2}(f) = \int |\Ls^{k+1} f|^2 &\geq c_{M,k}\left\{\int \frac{|\As (\Ls^k f)|^2}{y^2} \right.\nonumber \\
& \quad +\left. \sum_{m = 0}^k \int \frac{|\Ls^m f|^2}{y^4(1 + y^{4(k - m)})} + \sum_{m = 0}^{k-1}\frac{|\As(\Ls^{m}f)|^2}{y^6(1 + y^{4(k - m - 1)})}\right\}.
\end{align}
\end{lemma}
\begin{proof} We argue by induction on $k$. For $k = 0$, we apply Lemma \ref{lemm:coerAst} to $\As f$ with $i = 0$ and $\alpha = 0$, then  Lemma \ref{lemm:subcoerA} to $f$ with $i = 1$ and $p = 0$ to write
$$\Es_2(f) = \int |\Ls f|^2 \gtrsim \int \frac{|\As f|^2}{y^2} \gtrsim \int \frac{|\As f|^2}{y^2} + \int \frac{f^2}{y^4}.$$
Note that we had to use the orthogonality condition $\left< f, \Phi_M\right>$ when $\hbar = 0$. In fact, the case $\hbar = 0$ only happens when $d = 7$. In this case, the condition $2 > d - 2\gamma - 2$ is fulfilled when applying Lemma \ref{lemm:coerAst} with $i = 1$ and $p = 0$.

We now assume the claim for $k \geq 0$ and prove it for $k + 1$. We have the orthogonality condition 
$$\left<f, \Ls^{m}\Phi_M\right> = 0, \quad 0 \leq m \leq k + 1-\hbar.$$
Let $g = \Ls f$, then we have 
$$\left<g, \Ls^m \Phi_M\right> = 0, \quad 0 \leq m \leq k - \hbar.
$$
By induction hypothesis, we write
\begin{align*}
\int |\Ls ^{k+2} f|^2 &= \int |\Ls^{k+1} g|^2\\
&\gtrsim \int \frac{|\As (\Ls^k g)|^2}{y^2} + \sum_{m = 0}^k \int \frac{|\Ls^m g|^2}{y^4(1 + y^{4(k-m)})} + \sum_{m = 0}^{k-1}\frac{|\As(\Ls^{m}g)|^2}{y^6(1 + y^{4(k - m - 1)})}\\
&= \int \frac{|\As (\Ls^{k+1}f)|^2}{y^2} + \sum_{m = 1}^{k+1} \int \frac{|\Ls^m f|^2}{y^4(1 + y^{4(k + 1 -m)})} + \sum_{m = 1}^{k}\frac{|\As(\Ls^{m}f)|^2}{y^6(1 + y^{4(k - m)})}.
\end{align*}
Note that we have the orthogonality condition $\left<f, \Phi_M\right> = 0$ when $k \geq \hbar - 1$. The case $k \leq \hbar - 2$ implies that 
$$4 + 4k \leq 4 + 4\left(\frac{d}{4} - \frac{\gamma}{2} - \delta\right) - 8 \leq d - 2\gamma - 4.$$
Hence, we use the coercivity bound \eqref{eq:coerL1} to derive
$$\int \frac{|\Ls f|^2}{y^4(1 + y^{4k})} \gtrsim \int \frac{|\As f|^2}{y^6(1  +y^{4k})} + \int \frac{f^2}{y^4(1 + y^{4k + 4})},$$
which concludes the proof of Lemma \ref{lemm:coeLk}.
\end{proof}

\section{Interpolation bounds.}
We derive in this section interpolation bounds on $q$ which are the consequence of the coercivity property given in Lemma \ref{lemm:coeLk}. We have the following:
\begin{lemma}[Interpolation bounds]\label{lemm:interbounds} $\quad$\\
$(i)$ Weighted bounds for $q_i$: for $1 \leq m \leq \Bbbk$,
\begin{equation}\label{eq:qmbyE2k}
\int |q_{2m}|^2 + \sum_{i = 0}^{2k - 1}\int \frac{|q_i|^2}{y^2(1 + y^{4m - 2i - 2})} \leq C(M)\Es_{2m}.
\end{equation}
$(ii)$ Development near the origin:
\begin{equation}\label{eq:expandqat0}
q = \sum_{i = 1}^{\Bbbk}c_i T_{\Bbbk - i} + r_q,
\end{equation}
with bounds 
$$|c_i| \lesssim \sqrt{\Es_{2\Bbbk}},$$
$$|\py^j r_q| \lesssim y^{2\Bbbk - \frac{d}{2} - j}|\ln(y)|^{\Bbbk}\sqrt{\Es_{2\Bbbk}}, \quad 0 \leq j \leq 2\Bbbk - 1, \;\; y < 1.$$
$(iii)$ Bounds near the origin for $q_i$ and $\py^i q$: for $y \leq \frac 12$,
\begin{align*}
|q_{2i}| + |\py^{2i}q| &\lesssim y^{-\frac{d}{2} + 2}|\ln y|^\Bbbk\sqrt{\Es_{2\Bbbk}}, \quad \text{for}\quad 0 \leq i \leq \Bbbk - 1,\\
|q_{2i - 1}| + |\py^{2i-1}q|& \lesssim y^{-\frac{d}{2} + 1}|\ln y|^\Bbbk\sqrt{\Es_{2\Bbbk}}, \quad \text{for} \quad 1 \leq i \leq \Bbbk.
\end{align*}
$(iv)$ Weighted bounds for $\py^iq$: for $1 \leq m \leq \Bbbk$,
\begin{equation}\label{eq:weipykq}
\sum_{i = 0}^{2m}\int \frac{|\py^i q|^2}{1 + y^{4m - 2i}} \lesssim \Es_{2m}.
\end{equation}
Moreover, let $(i,j) \in \mathbb{N}\times \mathbb{N}^*$ with $2 \leq i + j \leq 2\Bbbk$, then 
\begin{equation}\label{eq:weipyijq}
\int \frac{|\py^i q|^2}{1 + y^{2j}} \lesssim \left\{\begin{array}{ll}
\Es_{2m} &\quad \text{for}\quad i+j = 2m, \;\; 1 \leq m \leq \Bbbk,\\
\sqrt{\Es_{2m}}\sqrt{\Es_{2(m+1)}}&\quad \text{for}\quad  i+j = 2m + 1,\;\; 1 \leq m \leq \Bbbk - 1.
\end{array} \right.
\end{equation}
$(v)$ Pointwise bound far away. Let $(i,j) \in \mathbb{N}\times \mathbb{N}$ with $1 \leq i + j \leq 2\Bbbk - 1$, we have for $y \geq 1$,
\begin{equation}\label{eq:pointwise_yg1}
\left|\frac{\py^i q}{y^{j}}\right|^2 \lesssim \frac 1{y^{d-2}}\left\{\begin{array}{ll}
\Es_{2m} &\quad \text{for}\quad i+j + 1 = 2m, \;\; 1 \leq m \leq \Bbbk,\\
\sqrt{\Es_{2m}}\sqrt{\Es_{2(m+1)}}&\quad \text{for}\quad  i+j = 2m,\;\; 1 \leq m \leq \Bbbk - 1.
\end{array} \right.
\end{equation}
\end{lemma}
\begin{proof} $(i)$ The estimate \eqref{eq:qmbyE2k} directly follows from Lemma \ref{lemm:coeLk}. \\
$(ii)$ We claim that for $1 \leq m \leq \Bbbk$, $q_{2\Bbbk -2m}$ admits the Taylor expansion at the origin
\begin{equation}\label{eq:expandq2k}
q_{2\Bbbk - 2m} = \sum_{i = 1}^m c_{i,m}T_{m - i} + r_{2m},
\end{equation}
with the bounds 
$$|c_{i,m}| \lesssim \sqrt{\Es_{2\Bbbk}},$$
$$|\py^j r_{2m}| \lesssim y^{2m - \frac{d}{2} - j}|\ln(y)|^m\sqrt{\Es_{2\Bbbk}}, \quad 0 \leq j \leq 2m - 1, \;\; y < 1,$$
The expansion \eqref{eq:expandqat0} then follows from \eqref{eq:expandq2k} with $m = \Bbbk$.

We proceed by induction in $m$ for the proof of \eqref{eq:expandq2k}. For $m = 1$, we write from the definition \eqref{def:Astar} of $\As^*$,
$$r_1(y) = q_{2\Bbbk - 1}(y) = \frac{1}{y^{d-1}\Lambda Q}\int_0^y q_{2\Bbbk} \Lambda Q x^{d-1}dx + \frac{d_1}{y^{d-1}\Lambda Q}.$$
Note from \eqref{eq:qmbyE2k} that $\int \frac{|q_{2\Bbbk - 1}|^2}{y^2} \lesssim \Es_{2\Bbbk}$ and from \eqref{eq:asymLamQ} that $\Lambda Q \sim y$ as $y \to 0$, we deduce that $d_1 = 0$. Using the Cauchy-Schwartz inequality, we derive the pointwise estimate
$$|r_1(y)| \leq \frac{1}{y^d} \left(\int_0^y |q_{2\Bbbk}|^2 x^{d-1}dx\right)^\frac{1}{2}\left(\int_0^y x^2x^{d-1} dx\right)^\frac{1}{2} \lesssim y^{-\frac{d}{2} + 1}\sqrt{\Es_{2\Bbbk}}, \quad y < 1.$$ 
We remark that there exists $a \in (1/2,1)$ such that 
$$|q_{2\Bbbk - 1}(a)|^2 \lesssim \int_{y \leq 1} |q_{2\Bbbk - 1}|^2 \lesssim \Es_{2\Bbbk}.$$
We then define 
$$r_2(y) = -\Lambda Q \int_a^y \frac{r_1}{\Lambda Q}dx,$$
and obtain from the pointwise estimate of $r_1$, 
$$|r_2(y)| \lesssim y y^{-\frac{d}{2} + 1}\sqrt{\Es_{2\Bbbk}}\int_a^y \frac{dx}{x} \lesssim y^{-\frac{d}{2} + 2}|\ln(x)|\sqrt{\Es_{2\Bbbk}}, \quad y < 1.$$
By construction and the definition \eqref{def:As} of $\As$, we have 
$$\As r_2 = r_1 = q_{2\Bbbk - 1}, \quad \Ls r_2 = \As^*q_{2\Bbbk-1} = q_{2\Bbbk} = \Ls q_{2\Bbbk - 2}.$$
Recall that $\text{span}(\Ls) = \{\Lambda Q, \Gamma\}$ where $\Gamma$ admits the singular behavior \eqref{eq:asymGamma}. From \eqref{eq:qmbyE2k}, we have $\int\frac{|q_{2\Bbbk - 2}|^2}{y^4} \lesssim \Es_{2\Bbbk} < +\infty$. This implies that there exists $c_2 \in \Rb$ such that 
$$q_{2\Bbbk - 2} = c_2 \Lambda Q + r_2.$$
Moreover, there exists $a \in (1/2,1)$ such that
$$|q_{2\Bbbk - 2}(a)|^2 \lesssim \int_{|y| \leq 1} |q_{2\Bbbk - 2}|^2\lesssim \Es_{2\Bbbk},$$
which follows
$$|c_2| \lesssim \sqrt{\Es_{2\Bbbk}}, \quad |q_{2\Bbbk - 2}| \lesssim y^{-\frac{d}{2} + 2}|\ln(y)|\sqrt{\Es_{2\Bbbk}}, \quad y < 1.$$
Since $\As r_2 = r_1$, we then write from the definition \eqref{def:As} of $\As$, 
$$|\py r_2| \lesssim |r_1| + \left|\frac{r_2}{y}\right| \lesssim y^{-\frac{d}{2} +2}|\ln(y)|\sqrt{\Es_{2\Bbbk}}, \quad y < 1.$$
This concludes the proof of \eqref{eq:expandq2k} for $m = 1$.   

We now assume that \eqref{eq:expandq2k} holds for $m \geq 1$ and prove it for $m + 1$. The term $r_{2m}$ is built as follows:
$$r_{2m - 1} = \frac{1}{y^{d-1}\Lambda Q}\int_0^y r_{2m - 2} \Lambda Q x^{d-1}dx, \quad r_{2m} = - \Lambda Q\int_a^y \frac{r_{2m -1}}{\Lambda Q} dx, \quad a\in (1/2,1).$$
We now use the induction hypothesis to estimate
\begin{align*}
|r_{2m + 1}| &= \left|\frac{1}{y^{d-1} \Lambda Q}\int_0^y r_{2m} \Lambda Q x^{d-1}dx \right| \\
&\quad \lesssim \frac{1}{y^d}\sqrt{\Es_{2\Bbbk}}\int_0^yx^{2m + \frac{d}{2}}|\ln(x)|^m dx\\
&\qquad \lesssim y^{2m - \frac{d}{2}} \sqrt{\Es_{2\Bbbk}}\int_0^y |\ln(x)|^m dx \\
& \qquad \quad  \lesssim y^{2m - \frac{d}{2} + 1}|\ln(y)|^m \sqrt{\Es_{2\Bbbk}}.
\end{align*}
Here we used the following identity 
$$I_m = \int_0^y [\ln(x)]^mdx \lesssim y|\ln(y)|^m, \quad m \geq 1, \;\; y < 1.$$
Indeed, we have $I_1 = \int_0^y \ln (x) dx = y\ln(y) - y \lesssim y|\ln(y)|$ for $y < 1$. Assume the claim for $m \geq 1$, we use an integration by parts to estimate for $m +1$
\begin{align*}
I_{m+1} &= \int_0^y[\ln(x)]^m (x\ln(x) - x)'dx\\
&\quad  = y[\ln(y)]^{m+1} - y[\ln(y)]^m - m(I_m - I_{m-1})\lesssim y|\ln(y)|^{m+1}.
\end{align*}
Using an integration by parts yields $\int_{a}^y \frac{[\ln(x)]^m}{x}dx = \frac{[\ln(y)]^{m+1} - [\ln(a)]^{m+1}}{m+1}$. Hence, we have the estimate
\begin{align*}
|r_{2m + 2}| = \left|\Lambda Q\int_a^y \frac{r_{2m +1}}{\Lambda Q} dx\right| &\lesssim y^{2m - \frac{d}{2} + 2}\sqrt{\Es_{2\Bbbk}}\int_a^y \frac{|\ln(x)|^m}{x}dx\\
&\quad \lesssim y^{2m - \frac{d}{2} + 2} |\ln(y)|^{m+1} \sqrt{\Es_{2\Bbbk}}.
\end{align*}

By construction, we have 
$$\As r_{2m+2} = r_{2m + 1}, \quad \Ls r_{2m+2} = r_{2m}.$$
From the induction hypothesis and the definition \eqref{def:Tk} of $T_k$, we write
$$\Ls q_{2\Bbbk - 2(m + 1)} = q_{2\Bbbk - 2m} = \sum_{i = 1}^m c_{i,m}T_{m - i} + r_{2m} = \sum_{i = 1}^mc_{i,m}\Ls T_{m + 1 - i} + \Ls r_{2m + 2}.$$
The singularity \eqref{eq:asymGamma} of $\Gamma$ at the origin and the bound $\int_{y \leq 1} \frac{|q_{2 \Bbbk - 2(m+1)}|^2}{y^4} \lesssim \Es_{2\Bbbk}$ allows us to deduce that
$$q_{2\Bbbk - 2(m + 1)} = \sum_{i=1}^m c_{i,m}T_{m +1 - i} + c_{2m + 2}\Lambda Q + r_{2m + 2}.$$
From \eqref{eq:qmbyE2k}, we see that there exists $a \in (1/2,1)$ such that 
$$|q_{2\Bbbk - 2(m+1)}(a)|^2 \lesssim \int_{y \leq 1}|q_{2\Bbbk - 2(m+1)}|^2 \lesssim \Es_{2\Bbbk}.$$
Together with the induction hypothesis $|c_{i, m}| \lesssim \sqrt{\Es_{2\Bbbk}}$ and the pointwise estimate on $r_{2m + 2}$, we get the bound $|c_{2m + 2}| \leq \sqrt{\Es_{2\Bbbk}}$.

A brute force computation using the definitions of $\As$ and $\As^*$ and the asymptotic behavior \eqref{eq:asympV}  ensure that for any function $f$, 
\begin{equation}\label{eq:pyjf}
\py^j f = \sum_{i = 0}^j P_{i,j}f_i, \quad |P_{i,j}| \lesssim \frac{1}{y^{j - i}},
\end{equation}
and we estimate
\begin{align*}
|\py^j r_{2m + 2}|&\lesssim \sum_{i = 0}^j \frac{|r_{2m + 2 - i}|}{y^{j-i}}\\
&\quad \lesssim \sqrt{\Es_{2\Bbbk}}\sum_{i = 0}^j \frac{y^{2m + 2 - i - \frac{d}{2}}|\ln(y)|^{m+1}}{y^{j - i}} \lesssim y^{2m + 2 -\frac{d}{2} - j}|\ln(y)|^{m+1} \sqrt{\Es_{2\Bbbk}}.
\end{align*}
This concludes the proof of \eqref{eq:expandq2k} as well as \eqref{eq:expandqat0}.\\

$(iii)$ The proof of $(iii)$ directly follows from \eqref{eq:expandq2k}.\\

$(iv)$ We have from \eqref{eq:pyjf},
$$|\py^kq| \lesssim \sum_{j = 0}^k \frac{|q_j|}{y^{k - j}},$$
and thus, using \eqref{eq:qmbyE2k} and the pointwise bounds given in part $(iii)$ yields
\begin{align*}
\sum_{i = 0}^{2m}\int \frac{|\py^i q|^2}{1 + y^{4m - 2i}} & \lesssim \Es_{2m} + \sum_{i = 0}^{2m - 1} \int_{y < 1} |\py^i q|^2 + \sum_{i = 0}^{2m - 1}\int_{y > 1}\frac{|\py^i q|^2}{y^{4m - 2i}}\\
&\quad \lesssim \Es_{2m} + \Es_{2\Bbbk}\int_{y< 1}y|\ln y|^\Bbbk dy  + \sum_{i = 0}^{2m - 1} \sum_{j = 0}^i \int_{y > 1}\frac{|q_j|^2}{y^{4m - 2j}} \lesssim \Es_{2m}, 
\end{align*}
which concludes the proof of \eqref{eq:weipykq}. 

The estimate \eqref{eq:weipyijq} simply follows from \eqref{eq:weipykq}. Indeed, if $i + j = 2m$ with $1 \leq m \leq \Bbbk$, we have 
$$\int \frac{|\py^i q|^2}{1 + y^{2j}} = \int \frac{|\py^iq|^2}{1 + y^{4m - 2i}} \lesssim \Es_{2m}.$$
If $i + j = 2m + 1$ with $1 \leq m \leq \Bbbk - 1$, we write 
\begin{align*}\int \frac{|\py^i q|^2}{1 + y^{2j}} = \int \frac{|\py^iq|^2}{1 + y^{4m - 2i + 2}} &\lesssim \left(\int \frac{|\py^iq|^2}{1 + y^{4m - 2i}} \right)^\frac 12\left(\int \frac{|\py^iq|^2}{1 + y^{4m - 2i + 4}} \right)^\frac 12\\
&\quad \lesssim \sqrt{\Es_{2m}} \sqrt{\Es_{2(m+1)}}.
\end{align*}
 
$(v)$ Let $i,j \geq 0$ with $1 \leq i + j \leq 2\Bbbk -1$, then $2 \leq i + j + 1 \leq 2\Bbbk$ and we conclude from \eqref{eq:weipyijq} that for $y \geq 1$,
\begin{align*}
\left|\frac{\py^i q}{y^j} \right|^2 &\lesssim \left| \int_y^{+\infty} \partial_x \left(\frac{(\partial_x^i q)^2}{x^{2j}} \right)dx \right| \lesssim \frac{1}{y^{d-2}} \left\{\int_y^{+\infty} \frac{|\partial_x^i q|^2}{x^{2j + 2}} + \int_y^{+\infty} \frac{|\partial_x^{i+1} q|^2}{x^{2j}} \right\}\\
&\quad \lesssim \frac 1{y^{d-2}}\left\{\begin{array}{ll}
\Es_{2m} &\quad \text{for}\quad i+j + 1 = 2m, \;\; 1 \leq m \leq \Bbbk,\\
\sqrt{\Es_{2m}}\sqrt{\Es_{2(m+1)}}&\quad \text{for}\quad  i+j+1 = 2m + 1,\;\; 1 \leq m \leq \Bbbk - 1.
\end{array} \right.
\end{align*} 
This ends the proof of Lemma \ref{lemm:interbounds}.
 
\end{proof}

\section{Proof of \eqref{est:comtor}.}\label{ap:EstComm}
We give here the proof of \eqref{est:comtor}. Before going to the proof, we need the following Leibniz rule for $\Ls^k$.
\begin{lemma}[Leibniz rule for $\Ls^k$] \label{lemm:LeibnizLk}  Let $\phi$ be a smooth function and $k \in \mathbb{N}$, we have 
\begin{equation}\label{eq:LeibnizLk}
\Ls^{k + 1}(\phi f) = \sum_{m = 0}^{k+1}f_{2m}\phi_{2k+2, 2m} + \sum_{m = 0}^k f_{2m + 1}\phi_{2k + 2, 2m + 1},
\end{equation}
and 
\begin{equation}\label{eq:LeibnizALk}
\As\Ls^{k}(\phi f) = \sum_{m = 0}^{k}f_{2m + 1}\phi_{2k+1, 2m + 1} + \sum_{m = 0}^k f_{2m}\phi_{2k+1, 2m},
\end{equation}
where\\
- for $k = 0$,
\begin{align*}
&\phi_{1,0} = -\py \phi, \quad \phi_{1,1} = \phi,\\
&\phi_{2,0} = -\py^2 \phi - \frac{d-1 +2V}{y}\py \phi, \quad \phi_{2,1}= 2\py \phi, \quad \phi_{2,2} = \phi,
\end{align*}
- for $k \geq 1$
\begin{align*}
&\phi_{2k + 1, 0} = -\py \phi_{2k,0},\\
&\phi_{2k + 1, 2i} = - \py \phi_{2k, 2i} - \phi_{2k, 2i - 1}, \quad 1 \leq i \leq k,\\
&\phi_{2k+1, 2i + 1}= \phi_{2k, 2i} + \frac{d-1 + 2V}{y}\phi_{2k, 2i + 1} - \py\phi_{2k, 2i+1}, \quad 0 \leq i \leq k-1,\\
&\phi_{2k + 1, 2k + 1} = \phi_{2k, 2k} = \phi,\\
& \quad\\
&\phi_{2k + 2, 0} = \py \phi_{2k+1,0} + \frac{d-1 + 2V}{y}\phi_{2k+1, 0},\\
&\phi_{2k + 2, 2i} = \phi_{2k + 1, 2i - 1} + \py \phi_{2k + 1, 2i} + \frac{d-1 + 2V}{y}\phi_{2k + 1, 2i},\quad 1 \leq i \leq k,\\
&\phi_{2k+2, 2i + 1}= -\phi_{2k + 1,2i} + \py\phi_{2k+1, 2i + 1}, \quad 0 \leq i \leq k,\\
&\phi_{2k + 2, 2k + 2} = \phi_{2k + 1, 2k + 1} = \phi.
\end{align*}
\end{lemma}
\begin{proof} We use the following relation,
\begin{align*}
\As(\phi f) = \phi \As f - \py \phi f, \quad  \As^*(\phi f) = \phi \As^* f + \py \phi f, \quad \As f + \As^*f = \frac{d-1 + 2V}{y}f,
\end{align*}
to compute
\begin{align*}
\As(\phi f) &= f_1 \phi + f(-\py \phi),\\
\Ls(\phi f) &= \As^*\As(\phi f) = f_2 \phi + f_1(2\py \phi) + f\left(-\py^2 \phi - \frac{d-1 + 2V}{y}\py \phi\right),
\end{align*}
which is the conclusion of \eqref{eq:LeibnizLk} and \eqref{eq:LeibnizALk} for $k = 0$.

Assume that \eqref{eq:LeibnizLk} and \eqref{eq:LeibnizALk} hold for $k \in \mathbb{N}$,  let us compute for $k \to k+1$. Using \eqref{eq:LeibnizLk}, we write 
\begin{align*}
\As\Ls^{k+1}(\phi f) &= \sum_{m = 0}^{k+1}\As \big[f_{2m}\phi_{2k+2, 2m}\big] + \sum_{m = 0}^k \left[-\As^* + \frac{d-1 +2V}{y}\right]f_{2m + 1}\phi_{2k + 2, 2m + 1}\\
&= \sum_{m = 0}^{k+1}\big\{f_{2m + 1} \phi_{2k + 2, 2m} + f_{2m}(-\py \phi_{2k+2, 2m})\big\}\\
& \quad + \sum_{m = 0}^k \Big\{f_{2m+2} (-\phi_{2k + 2, 2m + 1}) + f_{2m + 1} (-\py \phi_{2k + 2, 2m + 1})\\
&\qquad \qquad \qquad \qquad  \left.+ f_{2m + 1}\left(\frac{d-1 + 2V}{y}\phi_{2k + 2, 2m + 1}\right)\right\}\\
& = \sum_{m = 0}^k f_{2m + 1} \left(\phi_{2k + 2, 2m} - \py\phi_{2k + 2, 2m+1} + \frac{d-1 + 2V}{y} \phi_{2k+2, 2m + 1}\right)\\
& \quad + \sum_{m = 1}^k f_{2m}\left(-\py \phi_{2k+2, 2m} - \phi_{2k + 2, 2m + 1}\right) + f_{2k + 3} \phi_{2k+2, 2k + 2} + f (-\py \phi_{2k + 2, 0}),
\end{align*}
which yields the recurrence relation for $\phi_{2k + 3, j}$ with $0 \leq j \leq 2k + 3$.

Similarly, we write $\Ls^{k + 2}(\phi f) = \As^*\big[\As \Ls^{k+1}(\phi f)\big]$ and use the formula \eqref{eq:LeibnizALk} with $k+1$ to obtain the recurrence relation for $\phi_{2k + 4, j}$ with $0 \leq j \leq 2k + 4$. This concludes the proof of Lemma \ref{lemm:LeibnizLk}.
\end{proof}

\bigskip

Let us now give the proof of \eqref{est:comtor}. By induction and the definition \eqref{def:Llambda}, we have 
$$[\pt, \Ls_\lambda^{\Bbbk - 1}]v = \sum_{m = 0}^{\Bbbk - 2} \Ls_\lambda^m \left([\pt, \Ls_\lambda] \Ls_\lambda^{\Bbbk - 2 - m}v\right) = \sum_{m = 0}^{\Bbbk - 2} \Ls_\lambda^m \left(\frac{\pt Z_\lambda}{r^2} \Ls_\lambda^{\Bbbk - 2 - m}v\right).$$
Noting that $\frac{\pt Z_\lambda}{r^2} = \frac{b_1 \Lambda Z}{\lambda^4 y^2}$, we make a change of variables to obtain
\begin{align*}
\int \frac{1}{\lambda^2(1 + y^2)}\left|[\pt , \Ls_{\lambda}^{\Bbbk - 1}]v\right|^2 &= \frac{b_1^2}{\lambda^{4\Bbbk - d + 2}}\int \frac{1}{1 + y^2} \left|\sum_{m = 0}^{\Bbbk - 2}\Ls^m\left(\frac{\Lambda Z}{y^2} \Ls^{\Bbbk - 2 - m}q \right) \right|^2\\
& \lesssim \frac{b_1^2}{\lambda^{4\Bbbk - d + 2}} \sum_{m = 0}^{\Bbbk - 2}\int \frac{1}{1 + y^2}\left|\Ls^m\left(\frac{\Lambda Z}{y^2} \Ls^{\Bbbk - 2 - m}q \right) \right|^2.
\end{align*}
For $m = 0$, we use \eqref{eq:estLamZV} and \eqref{eq:Enercontrol} to estimate
$$\int \frac{1}{1 + y^2}\left|\left(\frac{\Lambda Z}{y^2} \Ls^{\Bbbk - 2}q \right) \right|^2 \lesssim \int \frac{|q^2_{2\Bbbk - 4}|}{1 + y^{10}} \lesssim \Es_{2\Bbbk}.$$
For $m = 1, \cdots, \Bbbk - 2$, we apply \eqref{eq:LeibnizLk} with $\phi = \frac{\Lambda Z}{y^2} = \frac{(d-1)\Lambda \cos(2Q)}{y^2}$ and note from \eqref{eq:asymQ} that
$$|\phi_{k, i}| \lesssim \frac{1}{1 + y^{2\gamma + 2 + (2k - i)}} \lesssim \frac{1}{1 + y^{4 + (2k - i)}}, \quad k \in \mathbb{N}^*, \;\; 0 \leq i \leq 2k,$$
which yields
\begin{align*}
\int \frac{1}{1 + y^2}\left|\Ls^m\left(\frac{\Lambda Z}{y^2} \Ls^{\Bbbk - 2 - m}q \right) \right|^2 \lesssim \sum_{i = 0}^{2m}\int\frac{q^2_{2\Bbbk - 4 - 2m - i}}{(1 + y^{10 + (4m - 2i)})} \lesssim \Es_{2\Bbbk}.
\end{align*}
Thus, 
$$\int \frac{1}{\lambda^2(1 + y^2)}\left|[\pt , \Ls_{\lambda}^{\Bbbk - 1}]v\right|^2 \lesssim \frac{b_1^2}{\lambda^{4\Bbbk - d + 2}} \Es_{2\Bbbk}.$$
Similarly, we use \eqref{eq:LeibnizALk} to get the estimate 
$$\int \left|\As[\pt , \Ls_{\lambda}^{\Bbbk - 1}]v\right|^2 \lesssim  \frac{b_1^2}{\lambda^{4\Bbbk - d + 2}}\Es_{2\Bbbk}.$$
This concludes the proof of \eqref{est:comtor}.


\def\cprime{$'$}

\end{document}